\documentclass[12pt,reqno]{amsart}
\usepackage{color}
\usepackage{stackengine}
\usepackage{latexsym}
\usepackage{amsmath,amssymb,dsfont,amsfonts}
\usepackage{mathrsfs}
\usepackage{verbatim}
\usepackage{graphicx}
\usepackage{graphics}
\usepackage{geometry}
\usepackage{booktabs}
\usepackage[shortlabels]{enumitem}
\usepackage{xcolor}
\definecolor{fore}{RGB}{249,242,215}
\definecolor{back}{RGB}{51,51,51}
\definecolor{title}{RGB}{255,0,90}
\definecolor{dgreen}{rgb}{0.,0.6,0.}
\definecolor{gold}{rgb}{1.,0.84,0.}
\definecolor{JungleGreen}{cmyk}{0.99,0,0.52,0}
\definecolor{BlueGreen}{cmyk}{0.85,0,0.33,0}
\definecolor{RawSienna}{cmyk}{0,0.72,1,0.45}
\definecolor{Magenta}{cmyk}{0,1,0,0}

\usepackage[shortlabels]{enumitem}
\newtheorem{proposition}{Proposition}[section]
\newtheorem{theorem}{Theorem}[section]
\newtheorem{definition}{Definition}[section]
\newtheorem{corollary}{Corollary}[section]
\newtheorem{lemma}{Lemma}[section]
\newtheorem{remark}{Remark}[section]
\newtheorem{example}{Example}[section]

\numberwithin{equation}{section}

\allowdisplaybreaks

\title[Rough paths and symmetric-Stratonovich
 integrals]{Rough paths and
  symmetric-Stratonovich integrals driven by singular
  covariance Gaussian processes}

\author{Alberto OHASHI$^1$}
\author{Francesco RUSSO$^2$}

\address{$1$ Departamento de Matem\'atica, Universidade de Bras\'ilia, 13560-970, Bras\'ilia - Distrito Federal, Brazil}\email{amfohashi@gmail.com}
\address{$2$ ENSTA Paris, Institut Polytechnique de Paris,
 Unit\'e de Math\'ematiques appliqu\'ees, 828, boulevard des Mar\'echaux, F-91120 Palaiseau, France}
 \email{francesco.russo@ensta-paris.fr}

\begin{document}

\begin{abstract}
We examine the relation between a stochastic version of the rough integral with the symmetric-Stratonovich integral in the sense of regularization. Under mild regularity conditions in the sense of Malliavin calculus, we establish equality between stochastic rough and symmetric-Stratonovich integrals driven by a class of Gaussian processes. As a by-product, we show that solutions of multi-dimensional rough differential equations driven by a large class of Gaussian rough paths they are actually solutions to Stratonovich stochastic differential equations. We obtain almost sure convergence rates of the first-order Stratonovich scheme to rough integrals in the sense of Gubinelli. In case the time-increment of the Malliavin derivative of the integrands is regular enough, the rates are essentially sharp. The framework applies to a large class of Gaussian processes whose the second-order derivative of the covariance function is a sigma-finite non-positive measure on $\mathbb{R}^2_+$ off diagonal.
\end{abstract}

\maketitle

\section{Introduction}



Let $X$ be a $d$-dimensional continuous Gaussian process over a bounded interval $[0,T]$ and equipped with a second-order process $\mathbb{X}$ so that $\mathbf{X} = (X,\mathbb{X})$ is a $\theta$-H\"{o}lder rough path for $\frac{1}{3} < \theta < \frac{1}{2}$ (see e.g. \cite{lyons} and \cite{gubinelli2004}). Let $\mathcal{D}^{2\theta}_X(\mathbb{R}^d)$ be the space of controlled rough paths of pairs $(Y,Y')$ satisfying

\begin{equation}\label{conintr}
Y_{t} - Y_s = Y'_s (X_{t} - X_s) + R^Y_{s,t}; 0\le s\le t\le T,
\end{equation}
where $R^Y_{s,t} = O(|t-s|^{2\theta})$ a.s. and $Y'$ is an $\mathbb{R}^{d\times d}$-valued $\theta$-H\"older continuous process.

The Sewing lemma (see e.g. \cite{feyel2006,gubinelli2004}) plays a fundamental role in the construction of the so-called rough integral $(Y,Y')\mapsto \big(\int Y d\mathbf{X},Y\big)$ which is described by

\begin{equation}\label{riemann}
\int_0^t Y_sd\mathbf{X}_s = \lim_{\|\Pi\|\rightarrow 0} \sum_{t_i \in \Pi} \Big\{ \langle Y_{t_i},X_{t_{i+1}} - X_{t_i}\rangle + Y'_{t_i}\mathbb{X}_{t_i,t_{i+1}}\Big\},
\end{equation}
almost surely, as the mesh of partitions $\|\Pi\|\rightarrow 0$, for $0\le t\le T$. The role of the underlying probability measure is totally restricted to the construction of the second-order process $\mathbb{X}$ and the Sewing lemma is applied pathwisely. See also \cite{le2020} for a stochastic version of Sewing lemma and other extensions by \cite{matsuda2022}, \cite{butkovsky2021} and \cite{friz2021}.

For a given $d$-dimensional process $Y$, let $I^0(\epsilon,Y,dX)$ be the first-order symmetric Stratonovich scheme given by

\begin{equation}\label{rs}
I^0(\epsilon,Y,dX)(t):=\frac{1}{2\epsilon}\int_0^t \Big\langle Y_s, X_{s+\epsilon} - X_{s-\epsilon} \Big\rangle ds;~0\le t\le T,
\end{equation}
where by convention, we set $X_t=X_0$ for $t\le 0$ and $X_t=X_T$ for $t\ge T$.
The symmetric-Stratonovich integral (in the sense of stochastic calculus via regularizations, see e.g. \cite{russo1993forward, Russo_Vallois_Book}) is defined by

\begin{equation}\label{Sintr}
\int_0^t Y_s d^0 X_s:= \lim_{\epsilon\downarrow 0}I^0(\epsilon,Y,dX)(t)~\quad (\text{in probability}),
\end{equation}
when it exists (see also Remark \ref{russorem}). In the language of regularization, the rough integral (\ref{riemann}) can be naturally formulated as

\begin{equation}\label{stint}
\int_0^t Y_s d\mathbf{X}_s=\lim_{\epsilon\downarrow 0}\frac{1}{\epsilon}\int_0^t \Big( \langle Y_s, X_{s+\epsilon} - X_s\rangle + Y'_s\mathbb{X}_{s,s+\epsilon}\Big)ds\quad (\text{in probability}),
\end{equation}
where the area process $\mathbb{X}$ is given by

\begin{equation}\label{secin}
\mathbb{X}_{u,v} = \int_u^v (X_r-X_u)\otimes d^0X_r; ~0\le u\le v\le T.
\end{equation}
The goal of this paper is to establish equality of the rough integral (\ref{stint}) with the symmetric-Stratonovich integral (\ref{Sintr}) for a given \textit{stochastically controlled} process $Y$ w.r.t. $X$ in the sense that there exists a $\mathbb{R}^{d\times d}$-valued process $Y'$ such that
\begin{equation}\label{taylor}
Y_t-Y_s = Y'_s (X_t-X_s) + R^Y_{s,t},
\end{equation}
where a two-parameter process $R^Y$ implicitly defined by (\ref{taylor}) satisfies
\begin{equation}\label{ortcond}
\lim_{\epsilon\downarrow 0}\frac{1}{\epsilon}\int_0^t \langle R^Y_{s,s+\epsilon}, X_{s+\epsilon}-X_s \rangle ds=0~\quad (\text{in probability}),
\end{equation}
for each $t >0$. This class of processes was recently introduced by \cite{gor} when the reference driving noise $X$ is a continuous semimartingale. A typical example of a pair $(Y,Y')$ satisfying (\ref{taylor}) and (\ref{ortcond}) is a controlled rough path in $\mathcal{D}^{2\theta}_X(\mathbb{R}^d)$ as described in (\ref{conintr}) (see Example 3.4 in \cite{gor}). In particular, we study the problem of the (almost sure) convergence rate of the first-order Stratonovich approximation scheme $I^0(\epsilon,Y,dX)(T)$ to rough integrals $\int_0^T Y_sd\mathbf{X}_s$ driven by $\mathbf{X}$ and $(Y,Y') \in \mathcal{D}^{2\theta}_X(\mathbb{R}^d)$ as described in (\ref{conintr}).

Stratonovich integrals play a prominent role in stochastic analysis. Since the pioneering work of \cite{wzakai}, we know that the Stratonovich formulation of stochastic differential equations (SDEs)

\begin{equation}\label{sdeint}
dY_t = f(Y_t)dX_t
\end{equation}
has the important interpretation of being approximated by a sequence of ordinary differential equations driven by smooth approximations $X^n$ for a continuous semimartingale driving noise $X$. In his seminal work, \cite{lyons} uses a Wong-Zakai-type argument to establish well-posedness of SDEs (\ref{sdeint}) driven by rather general noises. Rough path theory provides a robust pathwise solution which is
continuous with respect to the driving path $X$. Lyons's deep insight was to realize that what
really controls the dynamics in (\ref{sdeint}) is not just the path of $X$ but rather a ``natural'' lift of $X$ to a random rough path $\mathbf{X}$. \cite{gubinelli2004} observes that a consistent integration theory can be formulated by fixing $\mathbf{X}$ which results in (\ref{riemann}).

Another approach of stochastic calculus for irregular noises is via regularization (\cite{russo1993forward}) which is based on integral-type approximations of the form

\begin{equation}\label{regin}
\int_0^t Yd^*X_s=\lim_{\epsilon\rightarrow 0^+}\frac{1}{\epsilon}\int_0^t \langle Y_s, X(\epsilon,s)\rangle ds, \quad *=+,-,0,
\end{equation}
where $\frac{1}{\epsilon}X(\epsilon,s)$ encodes a sort of ``derivative approximation'' of $X$ and convergence (\ref{regin}) should be interpreted in probability. This gives rise to three different types of stochastic integrals called backward $(+)$, forward $(-)$ and symmetric-Stratonovich $(0)$ integrals. In this approach, none higher-order approximation scheme is employed. The connection with semimartingale theory, Young and Skorohod integrals has been studied over the years by many authors (see e.g. \cite{russo2007elements}, \cite{leon2022}, \cite{alosleon}, \cite{cheridito2005}, \cite{gradno}). When the driving noise has very low regularity, it turns out that symmetric-Stratonovich integral is the correct choice (see e.g. \cite{cheridito2005}, \cite{gnrv}, \cite{grv}). A one-dimensional theory of symmetric-Stratonovich SDEs is constructed by \cite{coviello1}, where the driving noise is
a combination of a general finite-cubic variation process (in the sense of \cite{er2}) with a semimartingale. We refer the reader to \cite{Russo_Vallois_Book} for a complete list of references.

In order to study the relation between symmetric-Stratonovich and rough integrals, we make use of the set $\mathcal{D}_X(\mathbb{R}^d)$ of all stochastically controlled processes $(Y,Y')$ realizing (\ref{taylor}) and (\ref{ortcond}). In the case the driving noise $X$ is a continuous local martingale, \cite{gor} show that $\mathcal{D}_X\big(\mathbb{R}^d)$ coincides with the space of weak Dirichlet processes. Motivated by the study of rough SDEs driven by Brownian motion on a given filtration $(\mathcal{F}_t)$ and a deterministic $\theta$-H\"older rough path, \cite{friz2021} have recently introduced a different notion of stochastic controllability w.r.t. a deterministic rough path, where remainders $R^Y_{s,t}$ satisfy a higher-order $2\theta$-H\"older-type condition based on the two-increment process $\mathbb{E}[R^Y_{s,t}|\mathcal{F}_s]$.








In the case the reference driving noise $X$ is a continuous semimartingale and the class of integrands is $\mathcal{D}_X(\mathbb{R}^d)$, \cite{gor} show that the classical Stratonovich stochastic integral coincides with the stochastic rough integral (\ref{stint}) driven by a Stratonovich second-order process $\mathbb{X}$. We also drive attention to  \cite{tindelliu} where the authors produce first-order trapezoidal approximations for (\ref{riemann}) in case $X$ belongs to a rather general class of Gaussian processes and $Y'$ is also controllable in the sense of \cite{gubinelli2004}. \cite{perkowski} show that the presence of $\mathbb{X}$ in (\ref{riemann}) can be neglected in case $X$ is a ``typical price path'' with finite quadratic variation, which confirms earlier considerations given by \cite{coviello2bis}. In the case $Y$ is a gradient system or a solution of a rough differential equation driven by a class of Gaussian geometric rough paths, then it is known that Skorohod correction terms can be derived. In this direction, see e.g. \cite{tindelhu} and \cite{cass2019,cass2021}, respectively.

The above results suggest that the rough integral (\ref{stint}) of any pair $(Y,Y') \in \mathcal{D}_X\big(\mathbb{R}^d)$ driven by Gaussian geometric rough paths $\mathbf{X}=(X,\mathbb{X})$ can be recast as a purely first-order symmetric-Stratonovich stochastic integral in the sense of regularization. The main result of this paper demonstrates that this is almost the case, at least for a large class of Gaussian driving noises and pairs $(Y,Y') \in \mathcal{D}_X(\mathbb{R}^d)$ whose derivative $Y'$ satisfies weak regularity conditions in the sense of Malliavin calculus.


\subsection{Summary of the main results}
In this article, we show \textit{equivalence} between (\ref{stint}) and (\ref{Sintr}) in the following sense. For a given Hilbert space $E$ and $1\le p < \infty$, let $\big(\mathbf{D}, \mathbb{D}^{1,p}(E)\big)$ be the Malliavin derivative defined on the Malliavin-Watanabe space $\mathbb{D}^{1,p}(E)$ of $E$-valued random elements supported by a probability measure $\mathbb{P}$ (see e.g. \cite{nualart2006}). The equivalence is stated below in an informal way.


\begin{theorem}\label{thintr0}
Let $X$ be a $d$-dimensional continuous Gaussian process with covariance kernel $R$ whose Schwartz second-order derivative is a non-positive sigma-finite measure $d\mu = \partial^2Rdx$ which is absolutely continuous w.r.t. Lebesgue on $[0,T]^2$ off diagonal. Assume $(Y,Y') \in \mathcal{D}_X(\mathbb{R}^d)$ and there exist $p,q>2$ such that $t\mapsto Y^{'}_t$ is a $\mathbb{D}^{1,p}(\mathbb{R}^{d\times d})$-valued continuous function (see section \ref{gsubsection}) and

 \begin{equation}\label{intr1}
 \int_0^T \int_{v_2}^T \sup_{s\ge v_1~\text{or}~s < v_2}\| \mathbf{D}_{v_1}Y^{'}_s - \mathbf{D}_{v_2}Y^{'}_s\|^{q}_{L^q(\mathbb{P})}|\partial^2 R(v_1,v_2)\big|^{\frac{q}{2}} dv_1dv_2 < \infty.
\end{equation}
Then, (\ref{Sintr}) exists if and only if (\ref{stint}) exists. Moreover, when $(Y,Y') \in \mathcal{D}_X(\mathbb{R}^d)$ is integrable (either in the sense of (\ref{Sintr}) or in the sense of (\ref{stint})), then for each $ t\in[0,T]$, we have
\begin{eqnarray}
\nonumber\int_0^t Y_s d\mathbf{X}_s &=& \int_0^t Y_s d^0X_s\\
\label{r=stra}&=& \lim_{\epsilon\rightarrow 0^+}\frac{1}{\epsilon}\int_0^t \Big( \langle Y_s, X_{s+\epsilon} - X_s\rangle + Y'_s\text{Sym}(\mathbb{X}_{s,s+\epsilon})\Big)ds,
\end{eqnarray}
in probability, where $\text{Sym}(\mathbb{X})$ is the symmetric part of $\mathbb{X}$ given by (\ref{secin}).
\end{theorem}
Theorem \ref{thintr0} is the loose summary of Theorem \ref{gaussiancase}. In particular, the second-order process $\mathbb{X}$ in (\ref{r=stra}) is given by the symmetric-Stratonovich integral (\ref{secin}) whose existence is guaranteed by the assumption $d\mu = \partial^2Rdx$ up to some technical conditions on the growth of the Radon-Nikodym derivative $\partial^2R$. See Proposition \ref{liftingTH} for details. We point out the choice of the symmetric-Stratonovich integral in (\ref{secin}) is fully dictated by the singularity of the covariance of $X$. See Remark \ref{Rforward}.


The Gaussian techniques employed in the proof of Theorem \ref{thintr0} relies on \cite{krukrusso} who develop a Malliavin calculus directly associated with the sigma-finite measure $d\mu = \partial^2Rdx$. In this direction, we also mention \cite{friz2016} who explore $d\mu$ and give ``complementary Young regularity'' for a large class of Gaussian processes. In the present work, we explore regularity of the covariance away the diagonal to connect the rough integral with first-order Stratonovich schemes. In general, complementary Young regularity for singular covariance structures does not imply equivalence of Stratonovich calculus with rough path integration. We give a precise condition for such equivalence in Theorem \ref{thintr0}. We stress that equality (\ref{r=stra}) provided by Theorem \ref{gaussiancase} can fail outside the class of integrands $(Y,Y') \in \mathcal{D}_X(\mathbb{R}^d)$. Indeed, in general, the existence of $\lim_{\epsilon\downarrow 0}I^0(\epsilon,Y,dX)(t)$ does not require the structure (\ref{taylor}) and (\ref{ortcond}) imposed on the set $\mathcal{D}_X(\mathbb{R}^d)$. See Lemma \ref{TRPlemma} for a concrete example.

From a qualitative point of view, Theorem \ref{thintr0} implies (see Proposition \ref{rdeex}) that solutions of rough differential equations driven by $\mathbf{X} = (X,\mathbb{X})$ are also solutions to multi-dimensional Stratonovich SDEs of the form

\begin{equation}\label{stsdei}
Y_t = Y_0 + \int_0^t V(Y_s) d^0X_s,
\end{equation}
for smooth coefficients $V$. A first-order Stratonovich-scheme for one-dimensional rough differential equations was studied by \cite{nourdin2007}.



The equivalence presented in Theorem \ref{thintr0} yields the investigation of the ($L^2(\mathbb{P})$ and almost sure) rate of convergence of the first-order Stratonovich approximation scheme $I^0(\epsilon,Y,dX)(T)$ to a rough integral. For sake of conciseness, the present article presents the convergence rates in the case of the fractional Brownian motion driving noise with $\frac{1}{3} < H < \frac{1}{2}$. In the sequel, we give a loose summary of Theorems \ref{thintr0} and \ref{mainTH}.


\begin{theorem}\label{rateINTRO}
Let $X$ be a $d$-dimensional fractional Brownian motion with exponent $\frac{1}{3} < H < \frac{1}{2}$. Let $(Y,Y') \in \mathcal{D}^{2\theta}_X(\mathbb{R}^d)$ be a controlled rough path, where $\frac{1}{3} < \theta < H$. Assume that $Y$ is adapted w.r.t. $X$ and $Y'$ satisfies the assumptions of Theorem \ref{thintr0}. Then,

\begin{equation}\label{rINT}
\mathbb{E}\Bigg| \int_0^T Y_sd\mathbf{X}_s - I^0(\epsilon,Y,dX)(T)  \Bigg|^2\lesssim \epsilon^{2\gamma + 2H-1} + \epsilon^{2(\eta + 2H -1)},
\end{equation}
as $\epsilon \downarrow 0$, where $(\gamma,\eta) \in (0,H]\times (0,1]$ are parameters such that $2\gamma + 2H-1>0$ and $\eta+2H-1>0$.
\end{theorem}
For the precise meaning of the parameters $(\gamma,\eta)$ in Theorem \ref{rateINTRO}, we refer the reader to the statement of Theorem \ref{mainTH} and, in particular, Assumptions S1 and S2 related to the integrand $Y$. At this point, we only stress $(\gamma,\eta)$ in Theorem \ref{rateINTRO} are related to the regularity of the increments of $Y$ in $\mathbb{D}^{1,2}(\mathbb{R}^d)$ and the increments of the Malliavin derivative of $Y$ on the simplex of $[0,T]^2$, respectively. If $\eta \ge \gamma +\frac{1}{2}-H$, then the leading term in the right-hand side of (\ref{rINT}) is $\epsilon^{2\gamma+ 2H-1}$ and the rate becomes $\epsilon^{(4H-1)-}$ as long as $\gamma \uparrow H$. In this case, the rate is essentially sharp considering that the L\'evy area diverges when $H=\frac{1}{4}$, see e.g. \cite{coutin2002}. Unfortunately, in case $\gamma +\frac{1}{2}-H > \eta$, the leading term in the right-hand side of (\ref{rINT}) is $\epsilon^{2(\eta + 2H-1)}$ and then the $L^2(\mathbb{P})$-rate becomes $\epsilon^{(6H-2)-}$ as long as $\eta \uparrow H$. In this case, it may not be sharp.

At this point, it is important to stress one fundamental difference between the rates derived in the present work and the previous literature on rough path theory. All approximations schemes reminiscent from rough path theory rely either on the Lipschitz continuity of the It\^o-Lyons map (see \cite{lyons}) or the stability of the rough integral (in the sense of \cite{gubinelli2004}) w.r.t. smooth approximations (typically Wong-Zakai-type) of the driving noises via Gaussian rough path lifts. In this direction, we refer the reader to e.g. \cite{friz,friz2010,deya2012,friz2014,coutin2007,hairerbook,gubinelli2004}. Inspired by \cite{hu2016}, one notable exception is \cite{liu2019} which constructs convergence rates of a simplified Euler scheme for rough differential equations based on the continuity of an ``augmented'' rough path lift associated to a triple of processes involving the driving noise, the scheme process and the normalized error process.

Theorem \ref{rateINTRO} establishes convergence rates to rough integrals solely through the increments of the driving Gaussian noise
$$\frac{1}{2 \epsilon} \big\{X_{(s+\epsilon)} - X_{(s-\epsilon)}\big\},$$
without relying on any sort of continuity of random rough path lifts w.r.t. approximations. This is possible because the equality of rough path and symmetric-Stratonovich integrals as described in Theorem \ref{thintr0} only depends on the (stochastic) regularity (see (\ref{intr1})) of the derivative $Y'$ of a pair $(Y,Y') \in \mathcal{D}_X(\mathbb{R}^d)$ and \textit{not} on continuity properties of random rough path lifts w.r.t. smooth approximations.


\subsection{Idea of the proofs}
In this section, we discuss the proofs of Theorems \ref{thintr0} and \ref{rateINTRO}. In the sequel, we fix a Gaussian process $X$ with covariance kernel $R$ satisfying the assumptions of Theorem \ref{thintr0}. We fix a parameter $-\frac{4}{3} < \alpha < -1$ which encodes the singularity of the Radon-Nikodym derivative $\partial^2R$ on the diagonal of $[0,T]^2$. One typical example is $\alpha=2HK-2$, where $R$ is the covariance of the bifractional Brownian motion (see Example \ref{biFBMex}) with exponents $H \in (0,1)$ and $K\in (0,1]$ such that $\frac{1}{3} < HK < \frac{1}{2}$. See e.g. \cite{russo2006bifractional} for basic properties of the bifractional Brownian motion.

Under the regularity condition (\ref{intr1}), Theorem \ref{thintr0} implies that if one relax almost sure convergence to convergence in probability, the anti-symmetric part $\text{Anti}(\mathbb{X})$ plays no role in the convergence of the integral in (\ref{r=stra}). Moreover, one can compute the stochastic rough integral (\ref{stint}) through a first-order symmetric-Stratonovich scheme $I^0(\epsilon,Y,dX)$ without involving the higher-order term $\mathbb{X}$. Let $\mathbf{X} = (X,\mathbb{X})$ be the geometric process defined by (\ref{secin}) and let $\text{Anti}(\mathbb{X})$ be the antisymmetric part of $\mathbb{X}$. The main argument in the proof of Theorem \ref{thintr0} is the verification that the convergence

\begin{equation}\label{keyqintr}
\lim_{\epsilon \downarrow 0}\frac{1}{\epsilon} \int_0^t \Big\langle Y'_s, \text{Anti}(\mathbb{X}_{s,s+\epsilon})\Big\rangle_{\mathbf{F}} ds=0\quad (\text{in probability})
\end{equation}
holds true in typical situations for $(Y,Y') \in \mathcal{D}_X(\mathbb{R}^d)$, where $\langle \cdot, \cdot \rangle_{\mathbf{F}}$ denotes the Frobenius inner product on the space of $d\times d$-matrices. The analysis of (\ref{keyqintr}) starts with the representation of $\text{Anti}(\mathbb{X})$ in terms of the divergence operator. In a second step, we provide delicate estimates on Skorohod integrals involving $Y'$ and components of $\text{Anti}(\mathbb{X})$. Convergence (\ref{keyqintr}) (in the sense of Riemann sum) is analyzed by \cite{tindelliu}, where the authors assume a pathwise second-order additional decomposition for $Y$, where $Y'$ follows (\ref{conintr}) equipped with a second Gubinelli's derivative $Y{''}$. On the one hand, in contrast to \cite{tindelliu}, none second-order pathwise expansion of $Y$ is employed in our framework. On the other hand, we assume Malliavin-type regularity on $Y'$. We believe (\ref{intr1}) is the natural stochastic regularity condition to insure (\ref{keyqintr}) and, indeed, (\ref{intr1}) it is fulfilled for a large class of examples.

We stress the simplest possible case takes place when $(Y,Y') \in \mathcal{D}_X(\mathbb{R}^d)$ is a controlled rough path in the Gubinelli's sense and $Y'$ is symmetric. This case is examined by \cite{hairerbook} and one can reduce the relevant information to the \textit{reduced rough path} $\mathbf{X} = (X, \text{Sym}(\mathbb{X}))$. In this work, we show that this phenomenon takes place in typical situations much beyond the symmetric case. For instance, when $Y'$ is deterministic, condition (\ref{intr1}) requires only continuity of $t\mapsto Y'_t$. We emphasize Theorem \ref{thintr0} can be extended to a less regular case $-\frac{3}{2} < \alpha \le - \frac{4}{3}$ by working with a corresponding level-3 Stratonovich geometric process. We postpone this analysis to a future work.


Theorem \ref{mainTH} presents the precise limiting behavior of $I^0(\epsilon,Y,dX)(T)$ to a symmetric-Stratonovich integral in a broader regime $-\frac{3}{2} < \alpha < -1$. The rate of convergence to a rough integral given by (\ref{rINT}) in Theorem \ref{rateINTRO} is then obtained by Theorem \ref{mainTH} and restricting to the case $-\frac{4}{3} < \alpha < -1$, where $\alpha=2H-2$. In the proof of (\ref{rINT}), we exploit a decomposition of (\ref{rs}) in terms of ``Skorohod component plus a trace term''. This type of approximate decomposition has already appeared in the seminal work of \cite{np1988} in the Brownian motion context and also in the fractional Brownian motion context in \cite{alos2003,alosleon,leon2022}. They both exploited undirect density-type arguments of simple processes which do not allow the obtention of convergence rates. Recently, in the particular case of rough differential equations, \cite{cass2019,cass2021} also exploit such type of decomposition without convergence rates.

In this work, under some natural conditions (see Theorem \ref{gaussiancase} and Assumptions S1 and S2) on $(Y,Y') \in \mathcal{D}_X(\mathbb{R}^d)$, we decompose

\begin{eqnarray}
\nonumber\int_0^t Y_s d\mathbf{X}_s &=& \int_0^t Y_s \boldsymbol{\delta}X_s + \frac{1}{2}\int_0^t \textit{tr}[\mathbf{D}_{s-}Y_s]dR(s,s)\\
\label{decin}&+& \int_{0 \le r_1 < r_2 \le t}\textit{tr}[\mathbf{D}_{r_1}Y_{r_2} - \mathbf{D}_{r_2-}Y_{r_2}]\partial^2R(r_1,r_2)dr_1dr_2,
\end{eqnarray}
where $\boldsymbol{\delta}X$ denotes the Skorohod integral. In the present work, the convergence rate (\ref{rINT}) is derived by means of representation (\ref{decin}), the regularity of the shifted process $Y_{\cdot + r}$ when $r\rightarrow 0$ (see Proposition \ref{epsilonlemma}) and a detailed analysis on $\textit{Tr}~(\mathbf{D}Y)_\epsilon$ (see (\ref{secs})) in terms of Assumptions S1 and S2. We stress that \cite{tindelsong} obtains a different representation of the rough integral for a non-anticipative second-order controlled integrand process, where the Stratonovich-Skorohod correction term mixes the Malliavin derivative trace with Gubinelli's derivative.

The paper is organized as follows. In Section \ref{prel}, we fix some notation and we define some basic objects. Section \ref{gaussiansection} presents the basic elements of the Gaussian space of the driving noise and some important tools from Malliavin calculus. Section \ref{resultsection} presents the main results of this article, namely Theorems \ref{gaussiancase} and \ref{mainTH}. Sections \ref{proofTh1} and \ref{proofTh2} present the proofs of Theorems \ref{gaussiancase} and \ref{mainTH}, respectively. Several technical lemmas are presented in Section \ref{appendixsec}.

\section{Preliminaries}\label{prel}
At first, we introduce some notation. In the sequel, finite-dimensional spaces will be equipped with a norm $|\cdot|$ and $T$ is a finite terminal time. The notation $\mathcal{C}^\alpha$ is reserved for $\alpha$-H\"older continuous paths defined on $[0,T]$ for $\alpha \in (0,1]$, with values in some finite-dimensional space. For $f\in \mathcal{C}^\alpha$, the usual seminorm is given by

$$\|f\|_{\alpha}:=\sup_{s,t\in [0,T]}\frac{|f_{t} - f_s|}{|t-s|^\alpha}.$$
The sup-norm on the space of continuous functions will be denoted by $\|\cdot\|_\infty$. For a two-parameter function $g$, we write $g\in \mathcal{C}^\beta_2$ if

$$\|g\|_{\mathcal{C}^\beta_2}:=\sup_{s,t\in [0,T]}\frac{|g_{s,t}|}{|t-s|^\beta} < \infty,$$
for $\beta >0$. The set $\mathcal{L}(\mathbb{R}^d,\mathbb{R}^n)$ denotes the space of all linear operators from $\mathbb{R}^d$ to $\mathbb{R}^n$ and \textit{tr} $Q$ denotes the trace of a matrix $Q$. We further write $a\lesssim b$ for two positive quantities to express an estimate of the form $a \le C b$, where $C$ is a generic constant which may differ from line to line. By convention, any continuous function $f$ defined on $[0,T]$ will be extended to the real line $\mathbb{R}$ as

$$
f(t):=\left\{
\begin{array}{rl}
f(0); & \hbox{if} \ t\le 0 \\
f(T);& \hbox{if} \ t \ge T.\\
\end{array}
\right.
$$

Throughout this article, we are given a reference continuous $\mathbb{R}^d$-valued stochastic process $X$ equipped with a second-order $\mathbb{R}^{d\times d}$-valued stochastic process $\mathbb{X}$ which satisfies the Chen's relation (see e.g. \cite{hairerbook})

\begin{equation}\label{chen}
\mathbb{X}_{s,t} - \mathbb{X}_{s,u} - \mathbb{X}_{u,t} = \big(X^i_{u} - X^i_s\big)\big( X^j_{t} - X^j_u\big) ; 1\le i,j\le d,
\end{equation}
for every $(s,u,t) \in [0,T]^3$. We then write $\mathbf{X} = (X,\mathbb{X})$. Let us consider

\begin{eqnarray*}
\mathbb{X}_{s,t} &=& \frac{1}{2}\Big(\mathbb{X}^{i,j}_{s,t}  +  \mathbb{X}^{j,i}_{s,t} \Big) + \frac{1}{2}\Big(\mathbb{X}^{i,j}_{s,t}  -  \mathbb{X}^{j,i}_{s,t} \Big); 1\le i,j\le d,\\
& &\\
&=:&\text{Sym}(\mathbb{X}_{s,t})+ \text{Anti}(\mathbb{X}_{s,t}).
\end{eqnarray*}
Throughout this paper, all stochastic processes are defined on a given probability space $\big(\Omega, \mathcal{F},\mathbb{P}\big)$.

\begin{definition}
We say that a pair $\mathbf{X} = (X,\mathbb{X})$ is a \textbf{geometric process} if
\begin{eqnarray}\label{geom}
\nonumber\text{Sym}(\mathbb{X}_{s,t}) &=& \frac{1}{2}[(X_t - X_s)\otimes (X_t-X_s)]\\
&:=&\frac{1}{2}(X^i_t- X^i_s) (X^j_t-X^j_s)~; 1\le i,j\le d,~s,t\in [0,T].
\end{eqnarray}
\end{definition}

\begin{definition}\label{controldef}

Given a reference process $X$, we say that an $\mathcal{L}(\mathbb{R}^d,\mathbb{R}^n)$-valued stochastic process $Y$ is \textbf{stochastically controlled} by $X$ if there exists an $\mathcal{L}(\mathbb{R}^d,\mathcal{L}(\mathbb{R}^d,\mathbb{R}^n))$-valued stochastic process $Y'$ so that the remainder term $R^Y$ given implicitly by the relation

\begin{equation}\label{lineardep}
Y_t - Y_s = Y'_s \big(X_{t} - X_s\big) + R^Y_{s ,t},
\end{equation}
is orthogonal to $X$, in the sense that

\begin{equation}\label{orth}
\lim_{\epsilon\rightarrow 0^+}\frac{1}{\epsilon}\int_0^t R^{Y}_{s,s+\epsilon}\big( X_{s+\epsilon} - X_s\big)ds=\mathbf{0},
\end{equation}
in probability for each $t \in [0,T]$.
\end{definition}
This defines the set $\mathcal{D}_X(\mathcal{L}(\mathbb{R}^d,\mathbb{R}^n))$ of all stochastically controlled processes $(Y,Y')$ satisfying (\ref{lineardep}) and (\ref{orth}). When $n=1$, we write $\mathcal{D}_X(\mathbb{R}^d):=\mathcal{D}_X(\mathcal{L}(\mathbb{R}^d,\mathbb{R}))$.

\begin{remark}
Clearly, the concept of stochastically controlled processes does not depend on a Gaussian structure for the driving noise. In fact, if $\mathbb{F}$ is a filtration and $X$ is a continuous $\mathbb{F}$-local martingale, then the class of stochastically controlled processes coincides with the class of continuous $\mathbb{F}$-weak Dirichlet processes, see Prop 3.7 in \cite{gor}.

\end{remark}



Inspired by \cite{gubinelli2004}, let us now give the definition of the integral in the sense of regularization.

\begin{definition}\label{roughstochDEF}
For a given pair $\mathbf{X}=(X,\mathbb{X})$, we say that $(Y,Y')\in \mathcal{D}_X(\mathbb{R}^d)$ is \textbf{rough stochastically integrable} if

\begin{equation}\label{stochroughint}
\int_0^t Y_s d\mathbf{X}_s:=\lim_{\epsilon\downarrow 0}\frac{1}{\epsilon}\int_0^t \Big( Y_s \big(X_{s+\epsilon} - X_s\big) + Y'_s\mathbb{X}_{s,s+\epsilon}\Big)ds
\end{equation}
exists in probability for each $t\in [0,T]$.
\end{definition}
We observe $Y'$ can be viewed as an $\mathcal{L}(\mathbb{R}^{d\times d}, \mathbb{R})$-valued process via the canonical injection $\mathcal{L}(\mathbb{R}^d,\mathcal{L}(\mathbb{R}^d,\mathbb{R}))\hookrightarrow \mathcal{L}(\mathbb{R}^{d\times d}, \mathbb{R})$. Moreover, we make an abuse of notation: we omit the dependence of the integral on $Y'$.

The next result is a simple consequence of the Sewing Lemma in the context of geometric rough paths (see e.g. \cite{gubinelli2004,hairerbook,lyons,friz}). For a proof of Lemma \ref{existenceGUB}, see Example 3.4 and Proposition 5.3 in \cite{gor}.

\begin{lemma}\label{existenceGUB}
Let $\mathbf{X} = (X,\mathbb{X})$ be a random $\gamma$-geometric rough path in the sense of \cite{gubinelli2004}, where $ X\in \mathcal{C}^\gamma$ and $\mathbb{X} \in \mathcal{C}^{2\gamma}_2$ a.s. with $\frac{1}{3} < \gamma < \frac{1}{2}$. Let $(Y,Y')$ be a controlled rough path in sense of \cite{gubinelli2004}, i.e., $Y$ is an $\mathbb{R}^d$-valued process with $\gamma$-H\"older continuous paths, $Y'$ is an $\mathcal{L}(\mathbb{R}^d, \mathcal{L}(\mathbb{R}^d,\mathbb{R}))$-valued process with $\gamma$-H\"older continuous paths so that the remainder term $R^Y$ given implicitly by relation

\begin{equation}\label{firstorder}
Y_t - Y_s = Y'_s\big(X_{t} - X_s\big) + R^Y_{s,t}
\end{equation}
satisfies $R^Y\in \mathcal{C}^{2\gamma}_2$ a.s. Then, $(Y,Y') \in \mathcal{D}_X(\mathbb{R}^d)$ and the limit

$$\lim_{\epsilon\rightarrow 0^+}\frac{1}{\epsilon}\int_0^\cdot \Big( Y_s (X_{s+\epsilon} - X_s) + Y'_s\mathbb{X}_{s,s+\epsilon}\Big)ds$$
exists almost surely and uniformly on $[0,T]$. Moreover, it coincides with the rough integral as described in \cite{gubinelli2004}.
\end{lemma}
\begin{remark}
We recall the Gubinelli's derivative introduced in \cite{gubinelli2004} may not be unique and uniqueness holds under the so-called true roughness
property for the driving noise (see e.g. Section 6.2 in \cite{hairerbook}). Therefore, in general, the process $Y'$ for a given $(Y,Y') \in \mathcal{D}_X(\mathbb{R}^d)$ is not unique. We postpone the investigation of the uniqueness of $Y'$ to a future project. In case $X$ is a continuous local-martingale, then $Y'$ is unique. See Propositions 3.7 and 3.9 in \cite{gor}.
\end{remark}
\section{The Gaussian space and some tools from Malliavin calculus}\label{gaussiansection}

Recall that any zero-mean continuous Gaussian process carries an abstract Wiener space structure which allows us to construct a Gross-Sobolev-Malliavin derivative and its associated adjoint. In general, this construction is abstract and a common strategy is to find a kernel representation for the covariance (see e.g. \cite{alos2001stochastic}) or make use of the two-dimensional $\rho$-variation with $1\le \rho<2$ (see e.g. \cite{cassfrizvic}) for the underlying covariance kernel. We follow the Gaussian analysis developed by \cite{kruk2007} and \cite{krukrusso} in terms of Schwartz distributions associated with the underlying covariance kernel. In an unpublished work, \cite{krukrusso} extend \cite{kruk2007} and they treat covariance structures admitting singularities on the diagonal of $[0,T]^2$. This section presents a brief account of \cite{krukrusso}.

Next, we describe the class of the Gaussian driving noises that we will consider in this article. In the sequel, $W$ is a (zero mean) real-valued Gaussian continuous process such that $W_0=0$ a.s. Let us denote

$$R(s_1,s_2):=\mathbb{E}[W_{s_1}W_{s_2}]; (s_1,s_2)\in [0,T]^2.$$

By recalling our convention that $W_t=W_T$ for $t\ge T$, we observe $R$ can be naturally extended to $\mathbb{R}^2_+$. A priori, $R$ is only continuous on $\mathbb{R}^2_+$ and hence $\partial^2 R:=\frac{\partial^2 R}{\partial s_1\partial s_2}$ will be interpreted in the sense of distributions. We denote

$$D :=\{(s_1,s_2)\in \mathbb{R}^2_+; s_1=s_2\}.$$
A priori, $\partial_u R(u,v), \partial_v R(u,v)$ and $\partial^2 R(u,v)$ are Schwartz distributions. We explore regularity of $R$ outside the diagonal $D$. Throughout the paper, the following assumptions will be in force.

\

\noindent\textbf{Assumption A}. For every $s\in [0,T]$, $R(dx, s) := \partial_xR(x,s)dx$ is a finite non-negative measure with compact support on $[0,T]$.



\

\noindent\textbf{Assumption B}. We suppose the product of the distribution $\partial^2 R$ with the smooth function $(s_1-s_2)$

$$\partial^2 R(s_1,s_2)(s_1-s_2)$$
is a regular distribution on $\mathbb{R}^2_+$ which is a real Radon measure that we denote by $\bar{\mu}$.

\

\noindent \textbf{Assumption $\textbf{C}$}.

\

\noindent (i) $\partial^2 R$ is a sigma-finite non-positive measure and absolutely continuous w.r.t. Lebesgue on $\mathbb{R}^2_+\setminus D$. With a slight abuse of notation, we denote it by $d\mu = \partial^2R dx$ on $\mathbb{R}^2_+\setminus D$. We assume that the Radon-Nikodym derivative satisfies

\begin{equation}\label{dagrowth}
\big|\partial^2 R(s_1,s_2)\big|\lesssim |s_1-s_2|^\alpha + \phi(s_1,s_2),
\end{equation}
for $(s_1,s_2) \in [0,T]^2\setminus D$, where $-\frac{3}{2} < \alpha < -1$ and there exists $L>1$ such that $\phi:[0,T]^2\setminus D \rightarrow \mathbb{R}_+$ is a symmetric $p$-integrable function over $[0,T]^2\setminus D$ for every $p  \in (1, L)$.

\

\noindent (ii) $\text{Var}\big(W_t-W_s\big)\lesssim |t-s|^{\alpha+2}$, for $s,t\in [0,T]$.

\

In addition, for the exponent $-\frac{3}{2}< \alpha < -1$ given in Assumption C (i), (ii), we assume:

\

\noindent (iii)

$$\big| R(v_1,T) - R(v_2,T) \big|\lesssim |v_1-v_2|^{\alpha+2}$$
for every $v_1,v_2 \in [0,T]^2\setminus D$.

\

\noindent (iv) There exists a non-increasing integrable function $\psi:[0,T]\rightarrow \mathbb{R}_+$ such that


\begin{enumerate}
  \item $\int_a^b|\phi(r_1,r_2)|dr_1\lesssim |b-a|^{\frac{\alpha+2}{2}}\psi(r_2)$
  \item $\int_{c}^d \psi(y)dy\lesssim |d-c|^{\frac{\alpha+2}{2}}$, for every $a,b,c,d$ in $[0,T]$.
  \item $s^{\frac{\alpha+2}{2}} \psi(s) \in L^1[0,T]$.



\end{enumerate}
Under Assumption B, one can check the total variation measure $|\mu|$ is absolutely continuous w.r.t. the total variation measure $|\bar{\mu}|$ with Radon-Nikodym derivative given by $\frac{1}{|y-x|}$. Of course, Assumption C(ii) and the Gaussian property imply that $W$ has $\gamma$-H\"older continuous paths for any $\frac{1}{4} < \gamma  < \frac{\alpha}{2} +1$. Assumption C (iii) and (iv) are technical hypotheses which will play a role in the proofs of Proposition \ref{epsilonlemma} and Theorem \ref{gaussiancase}.

\begin{example}\label{FBMexample}
Let $W$ be a fractional Brownian motion with exponent $0 < H < \frac{1}{2}$. Then,

$$R(s_1,s_2)= \frac{1}{2}\Big(\tilde{s}^{2H}_1 + \tilde{s}^{2H}_2 - |\tilde{s}_2 - \tilde{s}_1|^{2H}\Big); (s_1,s_2)\in \mathbb{R}^2_+,$$
where $\tilde{s}_i = s_i \wedge T$ for $i=1, 2$. Assumptions A, B and C are fulfilled. Indeed, $s_1\mapsto R(s_1, s)$ is absolutely continuous for each $s\in \mathbb{R}_+$, where

$$
\partial_{s_1} R(s_1,T)=\left\{
\begin{array}{rl}
H[s_1^{2H-1}+ (T-s_1)^{2H-1}]; & \quad \hbox{if} \ s_1 < T\\
0 \quad \quad \quad \quad \quad \quad \quad \quad \quad; & \quad \hbox{if} \ s_1 > T.
\end{array}
\right.
$$
Moreover,

$$\bar{\mu}(ds_1ds_2) = H(2H-1)|s_1-s_2|^{2H-1}\text{sgn}(s_1-s_2)\mathds{1}_{[0,T]^2 \setminus D}(s_1,s_2)ds_1ds_2$$
and

$$\partial^2 R(s_1,s_2) = H(2H-1)|s_1-s_2|^{2H-2}\mathds{1}_{[0,T]^2 \setminus D}(s_1,s_2),$$
for $(s_1,s_2)\in \mathbb{R}^2_+\setminus D$. Assumption C is fulfilled for $\alpha=2H-2$,  $\frac{1}{4} < H < \frac{1}{2}$ and $\phi=0$.

\end{example}

\begin{example}\label{biFBMex}
Let $W = B^{H,K}$ be a bifractional Brownian motion with parameters $H \in (0,1), K \in (0,1]$. It is known (see e.g. \cite{russo2006bifractional})

$$R(s_1,s_2) = 2^{-K} \big[ (\tilde{s}_1^{2H} + \tilde{s}_2^{2H})^K -  |\tilde{s}_1 - \tilde{s}_2|^{2HK}\big],$$
where $\tilde{s}_i = s_i\wedge T$. One can easily check

$$\partial_{s_1} R(s_1,s_2) = 2HK2^{-K}\Big[ (s^{2H}_1 + s^{2H}_2)^{K-1}s_1^{2H-1} - |s_1-s_2|^{2HK-1}\text{sign}(s_1-s_2)   \Big]$$
for $s_1,s_2 \in (0,T)$. Then,

$$\partial^2 R(s_1,s_2) = 2^{-K}\Big[ (4H^2K(K-1))(s^{2H}_1+s^{2H}_2)^{K-2}(s_1s_2)^{2H-1} + 2HK(2HK-1)|s_1-s_2|^{2HK-2}\Big],$$
for $(s_1,s_2)\in [0,T]^2\setminus D$,

$$
\partial_{s_1} R(s_1,\infty)=\left\{
\begin{array}{rl}
2HK2^{-K}\Big[(s^{2H}_1 + T^{2H})^{K-1}s_1^{2H-1} + (T-s_1)^{2HK-1}\Big]; & \hbox{if} \ s_1 \in (0,T) \\
0;& \hbox{if} \ s_1 > T, \\
\end{array}
\right.
$$
and

\begin{eqnarray*}
\bar{\mu}(ds_1ds_2) &=& \mathds{1}_{[0,T]^2}(s_1,s_2)2^{-K}\Big[ 4H^2K(K-1)(s^{2H}_1 +s^{2H}_2)^{K-2}(s_1s_2)^{2H-1}(s_1-s_2)^2    \\\
&+& 2HK(2HK-1)|s_1-s_2|^{2HK}\Big]ds_1ds_2.
\end{eqnarray*}
Since $2^{K-2}(s_1s_2)^{H(K-2)}\ge (s^{2H}_1+s^{2H}_2)^{K-2}$, we notice the existence of a positive constant $C(H,K,T)$ such that

\begin{equation}\label{bf1}
\partial_{s_1} R(s_1,T)\le C(H,K,T) \Big\{ s_1^{2H-1} + (T-s_1)^{2HK-1} \Big\}
\end{equation}
for every $s_1>0$,
\begin{equation}\label{bf2}
\big|\partial^2 R(s_1,s_2)\big| \le C(H,K,T) \Big\{ (s_1s_2)^{HK-1} + |s_1-s_2|^{2HK-2}\Big\},
\end{equation}
for every $(s_1,s_2)\in [0,T]^2\setminus D$. The function, $\phi(s_1,s_2) = (s_1s_2)^{HK-1}$ is $p$-integrable over $[0,T]^2\setminus D$ for every $1 < p < \frac{1}{1-HK}$. Therefore,  Assumptions A, B and C are fulfilled for $\frac{1}{4} < HK  < \frac{1}{2}$ and $\psi(t) = t^{HK-1}$. We observe bifractional Brownian motion does not have stationary increments for $K < 1$, it is $HK$-self similar with $\gamma$-H\"older continuous paths for $\gamma < HK$. See e.g. \cite{russo2006bifractional} for details.
\end{example}

\subsection{The ``reproducing kernel'' Hilbert space and related operators}\label{gsubsection}
In this section, we set the basic elements of the reproducing kernel Hilbert space associated with $R$ as introduced by \cite{krukrusso}. Throughout this paper, $X= (X^1, \ldots, X^d)$ is a $d$-dimensional centered process with iid components satisfying Assumptions A, B and C. In the sequel, let $C^1_0(\mathbb{R}_+,\mathbb{R}^d)$ be the space of $\mathbb{R}^d$-valued $C^1$-functions with compact support in $\mathbb{R}_+$ and $\langle \cdot, \cdot \rangle$ is the standard inner product on $\mathbb{R}^d$. We also denote $e_j; j=1,\ldots,d$ as the canonical basis of $\mathbb{R}^d$ and $\mathbf{1} = \sum_{\ell=1}^d e_\ell$.

Let $I:C^1_0(\mathbb{R}_+,\mathbb{R}^d)\rightarrow L^2(\mathbb{P})$ be the linear mapping defined by

$$I(f):= \int_0^\infty f_s dX_s:= \langle f(+\infty), X_\infty\rangle - \int_0^{+\infty} \langle X_s, d f(s)\rangle, $$
where $\langle f(+\infty), X_\infty\rangle:= \lim_{t\rightarrow +\infty}\langle f(t), X_T\rangle=0$.

Let $\tilde{L}_R(\mathbb{R}^d)$ be the linear space of all Borel functions $f:\mathbb{R}_+\rightarrow \mathbb{R}^d$ such that

\begin{description}
  \item[i] $\int_0^\infty |f|^2(s)|R|(ds,\infty) < \infty$,
  \item[ii] $\int_{\mathbb{R}^2_+\setminus D} |f(s_1) - f(s_2)|^2|\mu|(ds_1 ds_2) < \infty$.
\end{description}
For $f\in \tilde{L}_R(\mathbb{R}^d)$, we define

\begin{equation}\label{Vinner}
\|f\|^2_{L_R(\mathbb{R}^d)}:=\int_0^\infty |f(s)|^2R(ds,T) - \frac{1}{2}\int_{\mathbb{R}^2_+\setminus D} |f(s_1) - f(s_2)|^2 \mu(ds_1 ds_2).
\end{equation}
It is possible to show $\tilde{L}_R(\mathbb{R}^d)$ is a Hilbert space w.r.t. the inner-product associated with (\ref{Vinner}) and

\begin{equation}\label{iso1}
\mathbb{E}|I(f)|^2 = \|f\|^2_{L_R(\mathbb{R}^d)},
\end{equation}
for every $f\in C^1_0(\mathbb{R}_+,\mathbb{R}^d)$. Let $L_R(\mathbb{R}^d)$ be the closure of $C^1_0(\mathbb{R}_+,\mathbb{R}^d)$ w.r.t. $\|\cdot\|_{L_R(\mathbb{R}^d)}$ as a subset of $\tilde{L}_R(\mathbb{R}^d)$. If $d=1$, we will write $L_R = L_R(\mathbb{R})$. Then, $I:C^1_0(\mathbb{R}_+,\mathbb{R}^d)\rightarrow L^2(\mathbb{P})$ can be uniquely extended to a linear isometry

\begin{equation}\label{paleyop}
I:L_R(\mathbb{R}^d)\rightarrow L^2(\mathbb{P}).
\end{equation}

One can check $L_R(\mathbb{R}^d)$ is a real separable Hilbert space and

\begin{equation}\label{BVcase}
\int_0^\infty \varphi dX  = -\int_0^\infty \langle X, d\varphi\rangle,
\end{equation}
for every bounded variation function $\varphi$ with compact support. This implies that

\begin{equation}\label{cov1}
R(s,t) = \langle \mathds{1}_{[0,t]}, \mathds{1}_{[0,s]}\rangle_{L_R}; s,t \in [0,T].
\end{equation}

See Propositions 6.18, 6.14, 6.22, 6.32 and 6.33 in \cite{krukrusso} for the proof of these results. The Paley - Wiener integral associated with $X$ is given by

$$I(f):= \int_0^\infty fdX; f \in L_R(\mathbb{R}^d).$$

With the Paley-Wiener integral (\ref{paleyop}) at hand, one can construct a Malliavin calculus based on the Gaussian Hilbert space $L_R(\mathbb{R}^d)$ (see e.g. \cite{nualart2006}). Let $\mathcal{S}$ be the set of cylindrical real-valued random variables of the form

\begin{equation}\label{cylinder}
F = f \Bigg( \int_0^\infty \phi_1 dX, \ldots, \int_0^\infty \phi_m dX    \Bigg),
\end{equation}
where $f\in C^\infty_b(\mathbb{R}^m)$ (here $f$ is a smooth real-valued function on $\mathbb{R}^m$, where $f$ and all its partial derivatives are bounded), $\phi_1, \ldots, \phi_m \in C^1_0(\mathbb{R}_+,\mathbb{R}^d)$ and $m\ge 1$. For a cylinder random variable $F$ of the form (\ref{cylinder}), we then define

$$\mathbf{D}_t F = \sum_{i=1}^m\partial _i f \Bigg( \int_0^\infty \phi_1 dX, \ldots, \int_0^\infty \phi_mdX \Bigg)\phi_i(t); t\ge 0.$$
 We recall $\mathbf{D}:\mathcal{S}\rightarrow L^2(\Omega,L_R(\mathbb{R}^d))$ is a densely defined and closable operator satisfying the classical properties of the Gross-Sobolev-Malliavin derivative on the Gaussian space $\Big((\Omega, \mathcal{F}, \mathbb{P}); L_R (\mathbb{R}^d)\Big)$. For details, we refer the reader to \cite{nualart2006}.



In this article, we will frequently work with Hilbert space-valued smooth random elements in the sense of Malliavin calculus. Let $V$ be a real separable Hilbert space with a norm $\|\cdot \|_V$. Let $\mathcal{S}_{V}$ be the set of smooth $V$-valued stochastic processes of the form

$$F = \sum_{j=1}^n F_j v_j, v_j\in V, F_j \in \mathcal{S}.$$

We recall $\mathbf{D}$ can also be viewed a closable operator from $\mathcal{S}_{V}\subset L^p(\Omega;V)$ into $L^p(\Omega; V\otimes L_{R}(\mathbb{R}^{d})$, where $V\otimes L_R(\mathbb{R}^d)$ is the Hilbert tensor product of the pair $(V, L_R(\mathbb{R}^d))$ equipped with its standard norm $\|\cdot\|_{V\otimes L_{R}(\mathbb{R}^{d})}$. Let $\mathbb{D}^{1,p}(V)$ be the completion of $\mathcal{S}_{V}$ w.r.t.

$$\|F\|_{1,p,V}:=\Bigg[ \mathbb{E}\|F\|^p_{V} + \mathbb{E}\|\mathbf{D}F\|^p_{V\otimes L_{R}(\mathbb{R}^{d})}\Bigg]^{\frac{1}{p}},$$
for $p\ge 1$. To keep notation simple, we simply write $\mathbb{D}^{1,p}  = \mathbb{D}^{1,p}(\mathbb{R})$.

Throughout this article, $L_{2,R}(\mathbb{R}^{d\times d}):=L_R(\mathbb{R}^d) \otimes L_R(\mathbb{R}^d)$ is the Hilbert tensor product of $L_R(\mathbb{R}^d)$. The space $L_{2,R}(\mathbb{R}^{d\times d})$ can be identified as the closure of the algebraic tensor product $L_R(\mathbb{R}^d) \otimes^a L_R(\mathbb{R}^d)$ w.r.t. the norm

$$\|g\|^2_{2,R}:=\int_0^\infty \| g(t,\cdot)\|^2_{L_R(\mathbb{R}^d)}R(dt,T) -\frac{1}{2}\int_{\mathbb{R}^2_+\setminus D}\| g(t,\cdot) - g(s,\cdot)  \|^2_{L_R(\mathbb{R}^d)}\partial^2R(s,t)dsdt,$$
for an elementary tensor product $g = g_1\otimes g_2$, where $g_1,g_2 \in L_{R}(\mathbb{R}^{d})$.

The Gross-Sobolev-Malliavin derivative operator $\big(\mathbf{D},\mathbb{D}^{1,2}\big)$ admits an adjoint which is a densely defined closable linear operator $\big(\boldsymbol{\delta}, \text{dom}~\boldsymbol{\delta}\big)$, where $\mathbb{D}^{1,2}(L_R(\mathbb{R}^d))\subset \text{dom}~\boldsymbol{\delta}$. We recall the classical inequalities
\begin{equation}\label{cs1}
\|\boldsymbol{\delta}(u)\|_{L^2(\mathbb{P})}\lesssim \big\| \mathbb{E}[u] \big\|_{L_R(\mathbb{R}^d)} + \big\| \mathbf{D}u \big\|_{{L^2(\Omega; L_{2,R}(\mathbb{R}^{d\times d}))}}
\end{equation}
and
\begin{equation}\label{cs2}
\|\boldsymbol{\delta}(u)\|_{L^2(\mathbb{P})}\lesssim \| u\|_{\mathbb{D}^{1,2}(L_R(\mathbb{R}^d))},
\end{equation}
for $u \in \mathbb{D}^{1,2}(L_R(\mathbb{R}^d))$. See e.g. Prop. 1.5.8 in \cite{nualart2006}. We observe that $\mathbf{D}u$ is identified as a two-parameter matrix-valued process for $u \in \mathbb{D}^{1,2}(L_R(\mathbb{R}^d))$. We also make use of the well-known multiplication rule of smooth random variables with Skorohod integrals: Let $u \in \text{dom} ~\boldsymbol{\delta}$, $F \in \mathbb{D}^{1,2}$ such that $F\int_0^\infty u_s\boldsymbol{\delta}X_s \in L^2(\mathbb{P})$. Then, $Fu \in \text{dom} ~\boldsymbol{\delta}$ and

\begin{equation}\label{multfor}
\int_0^\infty Fu_s\boldsymbol{\delta}X_s  = F\int_0^\infty u_s\boldsymbol{\delta}X_s  - \langle \mathbf{D}F, u\rangle_{L_R(\mathbb{R}^d)}.
\end{equation}


\begin{definition}
If $u \mathds{1}_{[0,t]} \in \text{dom}~\boldsymbol{\delta}$ for every $t\ge 0$, then we define

$$\int_0^t u_s\boldsymbol{\delta} X_s: = \boldsymbol{\delta}(u \mathds{1}_{[0,t]}); t\ge 0.$$
\end{definition}
Of course, if $u \in L_R(\mathbb{R}^d)$, then $\int_0^t u_s\boldsymbol{\delta} X_s = I(u\mathds{1}_{[0,t]})$ for every $t\ge 0$.

Next, we recall the concept of the symmetric stochastic integral via regularization in the sense of \cite{russo1993forward}.
\begin{definition}
Let $Y$ be an $\mathbb{R}^d$-valued process with locally integrable paths. Let

$$I^0(\epsilon, Y,dX)(t) := \frac{1}{2\epsilon}\int_0^t \langle Y_s, X_{s+\epsilon} - X_{s-\epsilon}\rangle ds; 0\le t\le T.$$
We set

$$\int_0^t Yd^0 X := \lim_{\epsilon\downarrow 0}I^0(\epsilon, Y,dX)(t)\quad (\mathbb{P}-\text{probability});~0\le t\le T,$$
when it exists. The random variable $\int_0^t Yd^0X$ is called the \textbf{symmetric-Stratonovich integral} of $Y$ w.r.t. $X$.
\end{definition}
\begin{remark}\label{russorem}
We observe the symmetric-Stratonovich integral (if it exists) is the limit in probability of

$$
\int_0^t Yd^0 X = \lim_{\epsilon\downarrow 0} \frac{1}{2\epsilon}\int_0^t \langle Y_{u+\epsilon}  + Y_u, X_{u+\epsilon } - X_{u}\rangle du;~0\le t\le T.
$$
See Remark 3.2 in \cite{er2}.
\end{remark}



Next, we present two technical lemmas which will play a key role in this work.


\begin{lemma}\label{XL}
Under Assumptions A, B, C (i,ii), we have $X \in \mathbb{D}^{1,2}(L_R(\mathbb{R}^d))$.
\end{lemma}

\begin{lemma}\label{funcINT}
Let $\rho$ be a finite Borel measure on $\mathbb{R}_+, a:\mathbb{R}^2_+\rightarrow \mathbb{R}$ be a Borel function and $Y$ be a $\mathbb{R}^d$-valued stochastic process. We suppose the following.

\begin{enumerate}
  \item $a(s, \cdot)\in L_R$ for a.e.~$s$ w.r.t. $\rho$.
  \item $\int_0^\infty \|a(s,\cdot)\|^2_{L_R} \rho(ds) < \infty$
  \item $t\mapsto Y_t \in \mathbb{D}^{1,2}(\mathbb{R}^d)$ is continuous and bounded on $\text{supp}~\rho$.
\end{enumerate}
Then, the process

$$Z_t = \int_0^\infty a(s,t)Y_s \rho(ds)$$
belongs to $\mathbb{D}^{1,2}(L_R(\mathbb{R}^d))$ and

$$\mathbf{D}_\tau Z_t =  \int_0^\infty a(t,s)\mathbf{D}_\tau Y_s\rho(ds),~\tau \ge 0.$$
\end{lemma}
The proof of Lemma \ref{XL} is given in the Supplementary Material. The proof of Lemma \ref{funcINT} (when $d=1$) is given in Prop. 9.14 in \cite{krukrusso}. The same arguments apply to the multidimensional case. 

Let us now present two assumptions which will be essential in Theorem \ref{mainTH}. 
 
\

\noindent \textbf{Assumption S1:} There exists $\gamma \in (0,1]$ such that

\begin{equation}\label{incrP1}
\|Y_t-Y_s\|^2_{\mathbb{D}^{1,2}(\mathbb{R}^d)}\lesssim |t-s|^{2\gamma},
\end{equation}
where $2\gamma + \alpha +1>0$ and $\alpha \in (-\frac{3}{2},-1)$ is the exponent of Assumption C.

\

\noindent \textbf{Assumption S2}: Let $\alpha \in (-\frac{3}{2},-1)$ be the exponent in Assumption C. Assume there exists $\eta>0$ such that $\eta + \alpha+1 > 0$ and

\begin{equation}\label{COND}
\mathbb{E}\big|\textit{tr}[\mathbf{D}_{r_1}Y_s - \mathbf{D}_{r_2}Y_s]\big|^2\lesssim |r_2-r_1|^{2\eta},
\end{equation}
for every $0\le r_1 < r_2\le s \le T$.

\

In the sequel, we present a technical lemma which play a key role in the approximation scheme for Skorohod integrals. The proof of Lemma \ref{skolemma} is given in Section \ref{appendixsec}. 

\begin{lemma}\label{skolemma}
Let $X= (X_1,\ldots,X_d)$ be a $d$-dimensional Gaussian process satisfying Assumptions A, B and C (i,ii). Let $\alpha \in (-\frac{3}{2},-1)$ be the exponent of Assumption C. Assume $Y\in \mathbb{D}^{1,2}(L_R(\mathbb{R}^d))$ satisfies Assumption S1 with $2\gamma + \alpha +1 >0$. Then, $Y$ satisfies

$$\| Y_{\cdot + r} - Y_\cdot\|^2_{\mathbb{D}^{1,2}(L_R(\mathbb{R}^d))}\lesssim |r|^{2\gamma+ \alpha+1},$$
for every $|r| \in (0,1)$.
\end{lemma}

For a given $Y \in \mathbb{D}^{1,2}(L_R(\mathbb{R}^d))$, we denote

\begin{equation}\label{Ybar}
\bar{Y}^\epsilon_u : = \frac{1}{2\epsilon}\int_\epsilon^{T-\epsilon} Y_s \mathds{1}_{[u-\epsilon, u+\epsilon]}(s)ds,
\end{equation}
for $0\le u\le T$ and $2\epsilon < T$.


\begin{proposition}\label{epsilonlemma}
Let $X$ be a Gaussian process satisfying Assumption A, B, C(i,ii,iii) with $-\frac{3}{2} < \alpha < -1$. Assume $Y\in \mathbb{D}^{1,2}\big(L_R(\mathbb{R}^d)\big)$ satisfies Assumptions S1 with $2\gamma + \alpha +1>0$ and $\gamma \le \frac{\alpha}{2} + 1$. Let $\bar{Y}^\epsilon$ be the process defined in (\ref{Ybar}). Then,

$$\mathbb{E}\Bigg|\int_{0}^{T} \big(\bar{Y}^\epsilon_s - Y_s\big)\boldsymbol{\delta} X_u\Bigg|^2\lesssim \epsilon^{2\gamma + \alpha+1},$$
for every $\epsilon < \frac{T}{4}\wedge 1$.
\end{proposition}
\begin{proof}
Let $-\frac{3}{2} < \alpha < -1$ be the parameter in Assumption C. Assume $Y\in \mathbb{D}^{1,2}\big(L_R(\mathbb{R}^d)\big)$ satisfies Assumptions S1 with $2\gamma + \alpha +1>0$ and $\gamma \le \frac{\alpha}{2} + 1 < \frac{1}{2}$. At first, we notice $t\mapsto Y_t \in \mathbb{D}^{1,2}(\mathbb{R}^d)$ is continuous and hence Lemma \ref{funcINT} allows us to state $\bar{Y}^\epsilon \in \mathbb{D}^{1,2}(L_R(\mathbb{R}^d))$. We may assume $\epsilon < \frac{T}{4} \wedge 1$, where $\epsilon \downarrow 0$. Let us denote

$$A_1(\epsilon):=[2\epsilon, T-2\epsilon],~A_2(\epsilon):=[0,2\epsilon)~\text{and}~A_3(\epsilon):=(T-2\epsilon, T]. $$
By (\ref{cs2}), we have

\begin{eqnarray*}
\mathbb{E}\Bigg|\int_{0}^{T} \big(\bar{Y}^\epsilon_s - Y_s\big)\boldsymbol{\delta} X_u\Bigg|^2&\lesssim&  \| \bar{Y}^\epsilon - Y\|^2_{\mathbb{D}^{1,2}(L_R(\mathbb{R}^d))}\\
&=& \mathbb{E}\| \bar{Y}^\epsilon - Y\|^2_{L_R(\mathbb{R}^d)} + \mathbb{E}\big\| \mathbf{D}\big(\bar{Y}^\epsilon - Y\big)\big\|^2_{L_{2,R}(\mathbb{R}^{d\times d})},
\end{eqnarray*}
for every $\epsilon < \frac{T}{4}\wedge 1$. In order to shorten notation, let us denote
$$f^\epsilon_t = \frac{1}{2\epsilon}\int_{-\epsilon}^\epsilon \big[Y_{t+r} - Y_t\big]dr; 0\le t\le T.$$
We observe $f^\epsilon_t = \bar{Y}^\epsilon_t - Y_t$ for $t \in A_1(\epsilon)$ and by applying Lemma \ref{funcINT}, we have $f^\epsilon \in \mathbb{D}^{1,2}(L_R(\mathbb{R}^d))$ for every $\epsilon  < \frac{T}{4}\wedge 1$. We observe

\begin{equation}\label{div}
\bar{Y}^\epsilon_t = \frac{1}{2\epsilon} \int_\epsilon^{t+\epsilon} Y_r dr;~t\in A_2(\epsilon)\quad \text{and}\quad\bar{Y}^\epsilon_t = \frac{1}{2\epsilon} \int_{t-\epsilon}^{T-\epsilon} Y_r dr;~t\in A_3(\epsilon).
\end{equation}
Clearly,

$$\int_{A^2_1(\epsilon)\setminus D}\mathbb{E} | (\bar{Y}^\epsilon_t - Y_t) - (\bar{Y}^\epsilon_s - Y_s) |^2 |\partial^2 R(s,t)|dsdt$$
$$\le \int_{[0,T]^2\setminus D} \mathbb{E}|f^\epsilon_t - f^\epsilon_s |^2 |\partial^2 R(s,t)|dsdt \lesssim \mathbb{E}\| f^\epsilon\|^2_{L_R(\mathbb{R}^d)}.$$
By using Jensen's inequality on the Bochner integral (see e.g. \cite{perlman1974}) and Lemma \ref{skolemma}, we get

\begin{eqnarray*}
\mathbb{E}\| f^\epsilon\|^2_{L_R(\mathbb{R}^d)} &=& \mathbb{E}\Bigg\| \frac{1}{2\epsilon} \int_{-\epsilon}^\epsilon[Y_{\cdot+r} - Y_\cdot]dr\Bigg\|^2_{L_R(\mathbb{R}^d)}\\
&\le& \frac{1}{2\epsilon}\mathbb{E}\int_{-\epsilon}^\epsilon \|Y_{\cdot+r}  - Y_\cdot\|^2_{L_R(\mathbb{R}^d)} dr \lesssim\epsilon^{2\gamma + \alpha+1}.
\end{eqnarray*}
This shows

\begin{equation}\label{est1}
\int_{A^2_1(\epsilon)\setminus D}\mathbb{E} | (\bar{Y}^\epsilon_t - Y_t) - (\bar{Y}^\epsilon_s - Y_s) |^2 |\partial^2 R(s,t)|dsdt\lesssim \epsilon^{2\gamma + \alpha+1},
\end{equation}
for every $\epsilon < \frac{T}{4}\wedge 1$. Next, we observe

$$
\sup_{\epsilon < \frac{T}{4}\wedge 1}\sup_{(s,t)  \in A^c_1(\epsilon)\times A_1(\epsilon)\setminus D}\mathbb{E} | (\bar{Y}^\epsilon_t - Y_t) - (\bar{Y}^\epsilon_s - Y_s) |^2\lesssim \sup_{0\le r\le T}\mathbb{E}|Y_r|^2 < \infty,
$$
where

$$
\int_0^{2\epsilon} \int_{2\epsilon}^{T-2\epsilon}(t-s)^\alpha dtds + \int_{T-2\epsilon}^T  \int_{2\epsilon}^{T-2\epsilon}(s-t)^\alpha dtds \lesssim \epsilon^{\alpha+2}.
$$
Therefore,

\begin{equation}\label{est2}
\int_{A_1(\epsilon)\times A^c_1(\epsilon)\setminus D}\mathbb{E} | (\bar{Y}^\epsilon_t - Y_t) - (\bar{Y}^\epsilon_s - Y_s) |^2 |\partial^2 R(s,t)|dsdt\lesssim \epsilon^{\alpha+2},
\end{equation}
for every $\epsilon < \frac{T}{4}\wedge 1$. By applying Jensen's inequality, using (\ref{div}) and Assumption S1, we get

\begin{eqnarray*}
\mathbb{E}\big|(\bar{Y}^\epsilon_t - Y_t) - (\bar{Y}^\epsilon_s - Y_s) \big|^2 &\le& (t-s) \mathbb{E} \int_{s+\epsilon}^{t+\epsilon} \Big| Y_r \frac{1}{2\epsilon} - \frac{(Y_t-Y_s)}{t-s}\Big|^2dr\\
&\lesssim& \frac{(t-s)^2}{4\epsilon^2}\sup_{0\le r\le T}\mathbb{E}|Y_r|^2 + (t-s)^{2\gamma}; 0\le s < t < 2\epsilon.
\end{eqnarray*}
Therefore,

\begin{eqnarray*}
\int_{A^2_2(\epsilon)\setminus D}\mathbb{E} | (\bar{Y}^\epsilon_t - Y_t) - (\bar{Y}^\epsilon_s - Y_s) |^2 |\partial^2 R(s,t)|dsdt&\lesssim& \frac{1}{\epsilon^2}\int_0^{2\epsilon}\int_{0}^t (t-s)^{\alpha+2}dsdt\\
&+& \int_0^{2\epsilon}\int_{0}^t (t-s)^{2\gamma + \alpha}dsdt\\
&\lesssim& \epsilon^{\alpha+2},
\end{eqnarray*}
for every $\epsilon < \frac{T}{4}\wedge 1$. Similarly, by applying Jensen's inequality, using (\ref{div}) and Assumption S1, we get

\begin{eqnarray*}
\mathbb{E}\big|(\bar{Y}^\epsilon_t - Y_t) - (\bar{Y}^\epsilon_s - Y_s) \big|^2 &\le& (t-s) \mathbb{E} \int_{s-\epsilon}^{t-\epsilon} \Big| Y_r \frac{1}{2\epsilon} + \frac{(Y_t-Y_s)}{t-s}\Big|^2dr\\
&\lesssim& \frac{(t-s)^2}{4\epsilon^2}\sup_{0\le r\le T}\mathbb{E}|Y_r|^2 + (t-s)^{2\gamma}; T-2\epsilon < s < t \le T.
\end{eqnarray*}
Therefore,

\begin{eqnarray}\label{est3}
\nonumber\int_{A^2_3(\epsilon)\setminus D}\mathbb{E} | (\bar{Y}^\epsilon_t - Y_t) - (\bar{Y}^\epsilon_s - Y_s) |^2 |\partial^2 R(s,t)|dsdt&\lesssim& \frac{1}{\epsilon^2}\int_{T-2\epsilon}^T\int_{T-2\epsilon}^t (t-s)^{\alpha+2}dsdt\\
\nonumber&+& \int_{T-2\epsilon}^T\int_{T-2\epsilon}^t (t-s)^{2\gamma + \alpha}dsdt\\
&\lesssim& \epsilon^{\alpha+2},
\end{eqnarray}
for every $\epsilon < \frac{T}{4}\wedge 1$. By using (\ref{div}), one can easily check

$$\sup_{t\in A_2(\epsilon)\cup A_3(\epsilon)}\mathbb{E}|\bar{Y}^\epsilon_t|^2 \le \sup_{0\le r\le T}\mathbb{E}|Y_r|^2 < \infty.$$
Therefore, by using Assumption C(iii) and Jensen's inequality on the Bochner integral, we have

\begin{eqnarray}
\nonumber\mathbb{E}\int_0^T |\bar{Y}^\epsilon_t - Y_t|^2\partial_t R(t,T)dt &=& \sum_{i=1}^3\mathbb{E}\int_{A_i(\epsilon)} |\bar{Y}^\epsilon_t - Y_t|^2\partial_t R(t,T)dt \\
\nonumber&\lesssim& \mathbb{E}\|f^\epsilon\|^2_{L_R(\mathbb{R}^d)}+ \epsilon^{\alpha+2}\\
\nonumber&\lesssim& \frac{1}{2\epsilon}\int_{-\epsilon}^\epsilon \big\| Y_{\cdot+r} - Y_\cdot\big\|^2_{\mathbb{D}^{1,2}(L_R(\mathbb{R}^d))} dr + \epsilon^{\alpha+2}\\
\label{est4}&\lesssim& \epsilon^{2\gamma+\alpha+1},
\end{eqnarray}
for every $\epsilon < \frac{T}{4}\wedge 1$. Summing up (\ref{est1}), (\ref{est2}), (\ref{est3}) and (\ref{est4}), we get

$$\mathbb{E}\| \bar{Y}^\epsilon - Y\|^2_{L_R(\mathbb{R}^d)}\lesssim \epsilon^{2\gamma + \alpha+2},$$
for every $\epsilon < \frac{T}{4}\wedge 1$. Next, we investigate
\begin{eqnarray*}
\mathbb{E}\big\| \mathbf{D}\big(\bar{Y}^\epsilon - Y\big)\big\|^2_{L_{2,R}(\mathbb{R}^{d\times d})}&=& \mathbb{E}\int_0^T \big\| \mathbf{D}\big(\bar{Y}^\epsilon_t - Y_t\big)\big\|^2_{L_R(\mathbb{R}^d)}\partial_t R(t,T)dt\\
&+&\frac{1}{2} \mathbb{E}\int_{[0,T]^2\setminus D} \big\| \mathbf{D}\big(\bar{Y}^\epsilon_t - Y_t\big) - \mathbf{D}\big(\bar{Y}^\epsilon_s - Y_s\big) \big\|^2_{L_R(\mathbb{R}^d)}|\partial^2 R(s,t)|dsdt.
\end{eqnarray*}
The analysis is similar to the first part so we omit some details. Indeed, by using (\ref{div}) jointly with Jensen's inequality on the Bochner integral and Lemma \ref{skolemma}, we get

$$\mathbb{E}\int_{A^2_1(\epsilon)\setminus D} \big\| \mathbf{D}\big(\bar{Y}^\epsilon_t - Y_t\big) - \mathbf{D}\big(\bar{Y}^\epsilon_s - Y_s\big) \big\|^2_{L_R(\mathbb{R}^d)}|\partial^2 R(s,t)|dsdt
$$
$$\le \| f^\epsilon\|^2_{\mathbb{D}^{1,2}(L_R(\mathbb{R}^d))} = \Bigg\| \frac{1}{2\epsilon} \int_{-\epsilon}^\epsilon[Y_{\cdot+r} - Y_\cdot]dr \Bigg\|^2_{\mathbb{D}^{1,2}(L_R(\mathbb{R}^d))}$$
\begin{equation}\label{est5}
\le \frac{1}{2\epsilon} \int_{-\epsilon}^\epsilon \|Y_{\cdot+r} - Y_\cdot\|^2_{\mathbb{D}^{1,2}(L_R(\mathbb{R}^d))} dr\lesssim \epsilon^{2\gamma+\alpha+1},
\end{equation}
for every $\epsilon < \frac{T}{4}\wedge 1$. Moreover,

\begin{equation}\label{est6}
\sup_{\epsilon < \frac{T}{4}\wedge 1} \sup_{0\le t\le T}\mathbb{E}\|\mathbf{D}\bar{Y}^\epsilon_t\|^2_{L_R(\mathbb{R}^d)} + \sup_{0\le t\le T}\mathbb{E}\|\mathbf{D} Y_t\|^2_{L_R(\mathbb{R}^d)}\lesssim T^{2\gamma} + \|Y_0\|^2_{\mathbb{D}^{1,2}(\mathbb{R}^d)},
\end{equation}

\begin{equation}\label{est7}
\mathbb{E}\big\| \mathbf{D}\big(\bar{Y}^\epsilon_t - Y_t\big) - \mathbf{D}\big(\bar{Y}^\epsilon_s - Y_s\big) \big\|^2_{L_R(\mathbb{R}^d)}\lesssim \frac{(t-s)^2}{\epsilon^2}+ (t-s)^{2\gamma},
\end{equation}
for $0\le s < t < 2\epsilon$ or $T-2\epsilon < s < t\le T$. The estimates (\ref{est5}), (\ref{est6}), (\ref{est7}) and Assumption S1 yield

$$\mathbb{E}\big\| \mathbf{D}\big(\bar{Y}^\epsilon - Y\big)\big\|^2_{L_{2,R}(\mathbb{R}^{d\times d})}\lesssim \epsilon^{2\gamma + \alpha+1},$$
for every $\epsilon < \frac{T}{4}\wedge 1$. This concludes the proof.
\end{proof}

Next, we construct the second-order process which will allow us to connect (stochastic) rough integral with the symmetric-Stratonovich integral.

\begin{proposition}\label{liftingTH}
Assume $X$ is a $d$-dimensional Gaussian process (with iid components), where Assumptions A, B and C (i,ii,iii) are fulfilled. Then, the $\mathbb{R}^{d\times d}$-valued two-parameter process


\begin{equation*}
\mathbb{X}^{ij}_{s,t}=\left\{
\begin{array}{rl}
\int_s^t (X^i_r-X^i_s) d^0 X^j_r; & \hbox{if} \ i\neq j\\
\frac{1}{2}(X^i_t-X^i_s)^2  ;& \hbox{if} \ i=j
\end{array}
\right.
\end{equation*}
is geometric, it satisfies the Chen's relation (\ref{chen}) and we have the following representation outside the diagonal

\begin{equation}\label{skrepX}
\mathbb{X}^{ij}_{s,t}= \boldsymbol{\delta} \Big( (X^i  - X^i_{s})\mathds{1}_{[s,t]}e_j \Big),
\end{equation}
for $i\neq j$.
\end{proposition}
\begin{proof}
Fix $i\neq j$ and $0\le s  <t$. Let us consider $Z = (X^i-X^i_s)\mathds{1}_{[s,t]}e_j$. By Lemma \ref{XL}, $Z \in \mathds{D}^{1,2}(L_R(\mathbb{R}^d))$.  Clearly, $Z_a =0$ for $a < s$ or $a>t$ and one can easily check

\begin{equation}\label{za}
\| Z_a-Z_b\|^2_{\mathbb{D}^{1,2}(\mathbb{R}^d)}\lesssim\left\{
\begin{array}{rl}
\text{Var}(X^i_b - X^i_a); & \hbox{if} \ s\le a,b\le t \\
\text{Var}(X^i_{a\wedge b} - X^i_s)  ;& \hbox{if} s\le a\wedge b\le t < a\vee b \\
\text{Var}(X^i_{a\vee b} - X^i_s)  ;& \hbox{if}  a\wedge b < s\le a\vee b \le t. \
\end{array}
\right.
\end{equation}

The estimate (\ref{za}) shows that $a\mapsto Z_a \in \mathbb{D}^{1,2}(\mathbb{R}^d)$ is continuous except at $t$. Since $i\neq j$, we have $\langle  \mathbf{D} (X^i_u-X^i_s)\mathds{1}_{[s,t]}(u), \mathds{1}_{[u-\epsilon,u+\epsilon]}
e_j \rangle_{L_R(\mathbb{R}^d)} =0; 0\le u\le T$. Then, (\ref{multfor}), Lemma \ref{funcINT} and Fubini's theorem for Skorohod integral (see Prop 10.3 in \cite{krukrusso}) yield

\begin{equation}\label{z1}
\frac{1}{2\epsilon}\int_0^\infty \langle Z_r, X_{r+\epsilon} - X_{r-\epsilon}\rangle dr = \int_0^\infty Z^\epsilon_a\boldsymbol{\delta}X_a,
\end{equation}
where $Z^\epsilon_a = (2\epsilon)^{-1}\int_{a-\epsilon}^{a+\epsilon} (X^i_u-X^i_s)\mathds{1}_{[s,t]}(u)du e_j$. By (\ref{cs2}), we have

\begin{equation}\label{z2}
\mathbb{E}\Bigg|\int_0^\infty \big(Z^\epsilon - Z\big)\boldsymbol{\delta}X \Bigg|^2\lesssim \big\|  Z^\epsilon - Z \big\|^2_{\mathbb{D}^{1,2}(L_R(\mathbb{R}^d))}.
\end{equation}
By definition,
\begin{eqnarray}\label{z3}
\| Z^\epsilon- Z\|^2_{\mathbb{D}^{1,2}(L_R(\mathbb{R}^d))} &=& \int_0^T \| Z^\epsilon_a -Z_a\|^2_{\mathbb{D}^{1,2}(\mathbb{R}^d)}\partial_a R(a,T)da\\
\nonumber&+&\frac{1}{2}\int_{[0,T]^2\setminus D}\|(Z^{\epsilon}-Z)_a - (Z^{\epsilon}-Z)_b\|^2_{\mathbb{D}^{1,2}(\mathbb{R}^d)} |\partial^2 R(a,b)|dadb.
\end{eqnarray}
By using the Jensen's inequality on the Bochner integral (see e.g. \cite{perlman1974}) and the Lebesgue almost everywhere continuity of $a\mapsto Z_a \in \mathbb{D}^{1,2}(\mathbb{R}^d)$, we get pointwise convergence

\begin{equation}\label{z4}
\| Z^\epsilon_a -Z_a\|^2_{\mathbb{D}^{1,2}(\mathbb{R}^d)} = \Bigg\| \frac{1}{2\epsilon}\int_{-\epsilon}^\epsilon [Z_{m+a} - Z_a ]dm\Bigg\|^2_{\mathbb{D}^{1,2}(\mathbb{R}^d)}\le \frac{1}{2\epsilon}\int_{-\epsilon}^\epsilon \| Z_{m+a} - Z_a\|^2_{\mathbb{D}^{1,2}(\mathbb{R}^d)}dm \rightarrow 0,
\end{equation}
for each $a\neq t$ and

\begin{equation}\label{z5}
\| (Z^\epsilon -Z)_a - (Z^\epsilon -Z)_b \|^2_{\mathbb{D}^{1,2}(\mathbb{R}^d)}\le \frac{1}{2\epsilon}\int_{-\epsilon}^\epsilon \| Z_{m+a} - Z_a + Z_b-Z_{b+m}\|^2_{\mathbb{D}^{1,2}(\mathbb{R}^d)}dm \rightarrow 0,
\end{equation}
for each $(a,b) \in \mathbb{R}_+\setminus \{ t\}\times \mathbb{R}_+\setminus \{t\}$, as $\epsilon \downarrow 0$. Assumption C (i,ii), (\ref{za}), (\ref{z4}) and bounded convergence theorem yield

$$\int_0^T \| Z^\epsilon_a -Z_a\|^2_{\mathbb{D}^{1,2}(\mathbb{R}^d)}\partial_a R(a,T)da\rightarrow 0,$$
as $\epsilon \downarrow 0$. For the second term, by using the inequality in (\ref{z5}), Assumption C (i,ii) and (\ref{za}), one can easily check there exists $p >1$ such that

\begin{equation}\label{z6}
\sup_{0 < \epsilon < 1}\int_{[0,T]^2\setminus D} \| (Z^\epsilon -Z)_a - (Z^\epsilon -Z)_b \|^{2p}_{\mathbb{D}^{1,2}(\mathbb{R}^d)}|\partial^2R(a,b)|^pdadb < \infty.
\end{equation}
Then (\ref{z3}), (\ref{z4}), (\ref{z5}) and (\ref{z6}) yield $\|  Z^\epsilon - Z\|^2_{\mathbb{D}^{1,2}(L_R(\mathbb{R}^d))}\rightarrow 0$ as $\epsilon \downarrow 0$. By using (\ref{z1}) and (\ref{z2}), we conclude (\ref{skrepX}). The Chen's relation is obvious because the Stratonovich integral is constructed by regularization via limits of Riemann's integrals. A simple integration by parts argument on the symmetric-Stratonovich integrals yields $\mathbb{X}$ is geometric. This completes the proof.

\end{proof}

\begin{remark} \label{Rforward}

\begin{enumerate}
\item In Proposition \ref{liftingTH} we define
  the two-parameter rough path process  $\mathbb{X}$ in terms
  of symmetric-Stratonovich integrals.
  In general, we remark that those integrals cannot be replaced
  by forward (or backward) integrals via regularization.
  In fact, it is well-known, see  e.g. Lemma 6.1 of \cite{Russo_Vallois_Book},
  given a real process $X$, the forward integral
  $\int_0^\cdot X d^-X$ exists if and only if $X$
  is a finite quadratic variation process.
  If $X$ is such a Gaussian process, then this
  generally happens when the covariance $R_X$
  of the process is associated with a finite measure on $[0,T]^2$,
  see e.g. Proposition 3.1 of \cite{kruk2007}.
\item  On the other hand, the theory expanded in this
  paper could be adjusted to the ''regular'' case,
  i.e. for the case when $R_X$ is associated to a finite measure.
  In that case the definition of the rough path process
  $\mathbb{X}^{ij}$ could be built making use of forward integrals
  $\int_0^\cdot X_i d^- X^j$. In that case $\mathbb X$ still fulfills
  the Chen's relation (\ref{chen}) but it is not a geometric rough path.
  \end{enumerate}

\end{remark}

In order to integrate a controlled rough path in the sense of \cite{gubinelli2004} against $\mathbf{X} = (X,\mathbb{X})$, one has to check $\mathbb{X} \in \mathcal{C}^{2\gamma}_2$ a.s. Next, we give a class of examples in this direction, including a Gaussian process with non-stationary increments. In the sequel, if $g$ is a two-parameter continuous function, $\beta >0$ and $p\ge 1$, we write

$$U_{\beta,p}(g):=\Bigg[\int_0^T \int_0^T \frac{|g_{s,t}|^{p}}{|t-s|^{\beta p +2}}dsdt \Bigg]^{\frac{1}{p}}.$$
By using (\ref{cs1}), one can easily check the following example.


\begin{example}
If $X$ is a bifractional Brownian motion with parameter $\frac{1}{4} < HK< \frac{1}{2}$ with $H \in (0,1)$ and $K \in (0,1]$, then
$\mathbb{X}$ given in Proposition \ref{liftingTH} satisfies $$\mathbb{E}[U^p_{2\gamma,p}(\mathbb{X})]\lesssim \int_{[0,T]^2}\frac{|t-s|^{2pKH}}{|t-s|^{2\gamma p +2}}dsdt < \infty,$$
whenever $0 < \gamma < HK$ and $p > \frac{1}{2HK-2\gamma}$. By Corollary 4 in \cite{gubinelli2004}, this implies $\mathbb{X} \in \mathcal{C}^{2\gamma}_2$ a.s. for every $\gamma < HK$. Other examples of symmetric-Stratonovich-type second-order processes for Gaussian processes can be similarly treated by looking at $U_{\beta,p}$ and using a Skorohod-type representation of the form (\ref{skrepX}).
\end{example}

\section{Main results}\label{resultsection}
This section presents the main results of this paper. The proofs of Theorems \ref{gaussiancase} and \ref{mainTH} are given in sections \ref{proofTh1} and \ref{proofTh2}, respectively.

\begin{theorem}\label{gaussiancase}
Let $X$ be a Gaussian process satisfying assumptions A, B and C with $-\frac{4}{3} < \alpha < -1$. Let $\mathbf{X} = (X,\mathbb{X})$ be the geometric process given by (\ref{skrepX}). Assume that $(Y,Y')\in \mathcal{D}_X (\mathbb{R}^d)$, where $Y'$ satisfies the properties below:

\begin{enumerate}
\item $s\mapsto \mathbf{D}_vY'_s$ is continuous a.s. on $(0,T)\setminus\{v\}$ for Lebesgue almost all $v$.

\item There exists $p >2$ such that $t\mapsto Y^{'}_t$ is a $\mathbb{D}^{1,p}(\mathbb{R}^{d\times d})$-valued continuous function and

\begin{equation}\label{intassump}
\sup_{0\le t\le T}\mathbb{E}|Y^{'}_t|^{p} + \sup_{0\le t,r\le T}\mathbb{E}|\mathbf{D}_tY'_r|^{p} < \infty.
\end{equation}

\item There exists $q>2$ such that 
 \begin{equation}\label{mod1}
 \int_0^T \int_{v_2}^T \sup_{s\ge v_1~\text{or}~s < v_2}\| \mathbf{D}_{v_1}Y^{'}_s - \mathbf{D}_{v_2}Y^{'}_s\|^{q}_{L^q(\mathbb{P})}|\partial^2 R(v_1,v_2)\big|^{\frac{q}{2}} dv_1dv_2 < \infty.
 \end{equation}

\end{enumerate}

Then, $(Y,Y')\in \mathcal{D}_X (\mathbb{R}^d)$ is rough (stochastically) integrable if and only if $Y$ is symmetric-Stratonovich integrable
and, in this case, both integrals coincide

\begin{equation}\label{concl}
\int_0^t Y_s d\mathbf{X}_s = \int_0^t Y_s d^0 X_s; 0\le t \le T.
\end{equation}
\end{theorem}

\begin{remark}
The reader should be aware of the restriction $-\frac{4}{3} < \alpha < -1$ in Theorem \ref{gaussiancase}. For instance, in case of fractional Brownian motion,   $\alpha=2H-2$ and hence $\frac{1}{3} < H < 1$. Under assumptions 1, 2 and 3 for a pair $(Y,Y')\in \mathcal{D}_X (\mathbb{R}^d)$ in Theorem \ref{gaussiancase}, the symmetric-Stratonovich integral behaves like a stochastic rough integral driven by a reduced geometric process $\mathbb{X}  = (X,\text{Sym}(\mathbb{X}))$. See (\ref{I1}) and (\ref{redu}) for details.
\end{remark}

\begin{example}\label{c2bex}
If $f:\mathbb{R}^d\rightarrow \mathbb{R}^d \in C^2_b$ and $X$ is a Gaussian process satisfying assumptions A, B and C with $-\frac{4}{3} < \alpha < -1$. Then, $(f(X),\nabla f(X)) \in \mathcal{D}_X(\mathbb{R}^d)$ satisfies the assumptions in Theorem \ref{gaussiancase}.
\end{example}

\begin{proposition}\label{rdeex}
Assume that $V\in C_b^3(\mathbb{R}^d, \mathcal{L}(\mathbb{R}^d, \mathbb{R}^d))$, $\xi \in \mathbb{R}^d$ and let $X$ be a Gaussian process satisfying Assumptions A, B and C with $-\frac{4}{3} < \alpha < -1$. In addition, we assume the second order process (\ref{skrepX}) satisfies $\mathbb{X} \in \mathcal{C}^{2\gamma}_2$ a.s. ($\frac{1}{3} < \gamma < \frac{\alpha}{2} +1$) and $R$ has finite two-dimensional $\rho$-variation for $1\le \rho < \frac{3}{2}$ (see e.g. Def. 5.50 in \cite{friz}). Let $Y$ be the solution of the rough differential equation

\begin{equation}\label{rdeEQ}
Y_t = \xi + \int_0^t V(Y_s)d\mathbf{X}_s; 0\le t\le T.
\end{equation}
Then, $Y$ satisfies the assumptions in Theorem \ref{gaussiancase}. In particular, $Y$ is a solution to the Stratonovich differential equation interpreted in the sense of \cite{russo1993forward}


\begin{equation}\label{stratEQ}
Y_t = \xi + \int_0^t V(Y_s)d^0X_s; 0\le t\le T.
\end{equation}
\end{proposition}
\begin{proof}
Let $V = (V^1, \ldots, V^d)$ where $V^i:\mathbb{R}^d\rightarrow \mathbb{R}^d$ are $C^3_b(\mathbb{R}^d;\mathbb{R}^d)$ vector fields. It is known that $Y'_t = V(Y_t)$ (see e.g. Prop 8.3 in \cite{hairerbook}) and hence chain rule yields $\mathbf{D}V(Y_t) = \big( \mathbf{D}V^1(Y_t), \ldots, \mathbf{D}V^d(Y_t)\big) = \big(\nabla V^1(Y_t)\circ \mathbf{D}Y_t, \ldots, \nabla V^d(Y_t)\circ \mathbf{D}Y_t\big)$. It is well-known (see e.g. \cite{cass2015sm}) that $\mathbf{D}_sY_t = J_t\circ J^{-1}_s\circ V(Y_s)\mathds{1}_{[0,t]}(s)$, where $J_t$ denotes the Jacobian of the solution $Y_t$ where $Y_0 = \xi$. Here, $J^{-1}_s$ is the inverse of the matrix-valued Jacobian $J_s$. We fix $\frac{1}{3} < \gamma < \frac{\alpha}{2} +1$. Then,

\begin{eqnarray*}
\sup_{s\ge \max{(v_1, v_2)}~\text{or}~s < \min{(v_1, v_2)}}| \mathbf{D}_{v_1}Y^{'}_s - \mathbf{D}_{v_2}Y^{'}_s|&\le& \max_{1\le i\le d}\|\nabla V^i(Y)\|_\infty\|J_\cdot\|_\infty \|J^{-1}\|_{\gamma}\| V(Y)\|_\infty\\ 
&\times& |v_1-v_2|^\gamma\\
&+& \max_{1\le i\le d}\|\nabla V^i(Y)\|_\infty\|J_\cdot\|_\infty \|J^{-1}_\cdot\|_\infty \|V(Y)\|_{\gamma}\\
&\times& |v_1-v_2|^\gamma,
\end{eqnarray*}
for $\frac{1}{3} < \gamma < \frac{1}{2}$. By invoking \cite{cass2013int}, we know that

$$\big(\|J\|_{p-var}, \|J^{-1}\|_{p-var}\big) \in \bigcap_{q\ge 1}L^q(\mathbb{P})$$
for $2 < p < 3$. Here, $\|\cdot \|_{p-var}$ denotes the $p$-variation norm. If $R$ has finite two-dimensional $\rho$-variation, it is actually possible to prove (see e.g. Remark 7.3 in \cite{cass2015sm})

$$
\Big\{\|J\|_\infty, \|J^{-1}\|_\infty,\|J\|_{\frac{1}{p}}, \|J^{-1}\|_{\frac{1}{p}}\Big\} \subset \bigcap_{q\ge 1}L^q(\mathbb{P}).
$$
Under Assumption C, assumptions (\ref{mod1}) and (\ref{intassump}) hold true. Indeed, since $-\frac{4}{3} < \alpha < -1$ and $\frac{1}{3} < \gamma < \frac{1}{2}$, then $\frac{-1}{\gamma + \frac{\alpha}{2}} > 3$ so that $p\gamma + \alpha \frac{p}{2} + 1 >0$ as long as $2 < p < \frac{-1}{\gamma + \frac{\alpha}{2}}$. The proof of (\ref{stratEQ}) follows by routine
arguments based on chain rule and application of Theorem \ref{gaussiancase} to $V(Y)$, so we omit details.

\end{proof}




Next, we discuss convergence rates of first-order Stratonovich approximation schemes. For simplicity of exposition, we present the results in the case $X$ is the fractional Brownian motion.

\

\begin{theorem}\label{mainTH}
Let $X$ be a $d$-dimensional fractional Brownian motion with $\frac{1}{4} < H < \frac{1}{2}$. Assume $Y\in \mathbb{D}^{1,2}(L_R(\mathbb{R}^d))$ is adapted w.r.t. $X$ and it satisfies the following regularity conditions.

\begin{itemize}
\item There exists $q>2$ such that $\sup_{0\le t\le T}\mathbb{E}|Y_t|^q < \infty.$
  \item $\text{tr}[\mathbf{D}_\cdot Y_s]$ has continuous paths on $[0,s]$ for every $s\le T$ and 
  $$\sup_{0\le t\le T}\mathbb{E}|\textit{tr}[\mathbf{D}_0 Y_t]|^2 < \infty.$$
  \item Assumption S1 is fulfilled for $0 < \gamma\le H$ such that $2\gamma + 2H -1 > 0$.
  \item Assumption S2 is fulfilled for $\eta>0$ such that $\eta + 2H -1 > 0$.

\end{itemize}
Then, $Y$ is symmetric-Stratonovich integrable w.r.t. $X$ and we have the representation

$$
\int_0^T Y_s d^0X_s = \int_0^T Y_s \boldsymbol{\delta}X_s + H \int_0^T \text{tr}[\mathbf{D}_{s-}Y_s]s^{2H-1}ds$$
\begin{equation}\label{skorohod_formula}
+\int_{0 \le r_1 < r_2 \le T}\text{tr}[\mathbf{D}_{r_1}Y_{r_2} - \mathbf{D}_{r_2-}Y_{r_2}]\partial^2R(r_1,r_2)dr_1dr_2.
\end{equation}
In addition, there exists a constant $C$ which depends on (\ref{incrP1}) and (\ref{COND}) such that

\begin{equation}\label{rateabs}
\mathbb{E}\Bigg| \int_0^T Y_sd^0X_s - I^0(\epsilon,Y,dX)(T)  \Bigg|^2\le C \{\epsilon^{2\gamma + 2H-1} + \epsilon^{2(\eta + 2H -1)}\},
\end{equation}
for every $\epsilon>0$ sufficiently small. In particular, when we restrict to the case $\frac{1}{3} < H < \frac{1}{2}$ and $(Y,Y') \in \mathcal{D}_X(\mathbb{R}^d)$ satisfies items 1, 2 and 3 in Theorem \ref{gaussiancase}, then $(Y,Y')\in \mathcal{D}_X(\mathbb{R}^d)$ is rough stochastically integrable, representation (\ref{skorohod_formula}) and the estimate (\ref{rateabs}) hold for the stochastic rough integral as well.
\end{theorem}

\begin{remark}\label{TindelCOMP}
The assumption that $Y$ is adapted w.r.t. $X$ in the representation (\ref{skorohod_formula}) in Theorem \ref{mainTH} is not essential. Indeed, in case $Y$ is not necessarily adapted, there will be two additional terms related to $\text{tr}[\mathbf{D}_{s+}Y_{s}]$ and $
\text{tr}[\mathbf{D}_{r_1}Y_{r_2} - \mathbf{D}_{r_2+}Y_{r_2}]\partial^2R(r_1,r_2)$ on $0\le r_2 < r_1\le t$. For simplicity of exposition, we only discuss in detail the case where $Y$ is adapted.
\end{remark}
We now present two classes of significant examples which illustrate Theorem \ref{mainTH} and its relation with Theorem \ref{gaussiancase}.
\subsection{The case $Y = f(X)$}\label{fxsec}

\begin{lemma}\label{TRPlemma}
Fix $\frac{1}{4} < H < \frac{1}{2}$ and let $f:\mathbb{R}^d\rightarrow \mathbb{R}^d$ be a continuously differentiable function such that $f$ and $\nabla f$ are $\theta$-H\"older continuous functions with $\frac{1}{2H}-1 < \theta \le 1$. Then, $f(X) \in \mathbb{D}^{1,2}(L_R(\mathbb{R}^d))$, Assumption S2 is fulfilled with exponent $\eta=1$ and

\begin{equation}\label{fskc}
\| f(X_t) - f(X_s)\|^2_{\mathbb{D}^{1,2}(\mathbb{R}^d)}\lesssim |t-s|^{2H\theta},
\end{equation}
for $s,t\ge 0$. Therefore, $f(X)$ satisfies the Assumptions of Theorem \ref{mainTH} and it is symmetric-Stratonovich integrable. In particular, when we restrict to the case $\frac{1}{3} < H < \frac{1}{2}$ and $\frac{1}{2H}-1 < \theta \le \frac{1}{H}-2$, then $\nabla f(X)$ is $\theta \gamma$-H\"older continuous for every $\gamma < H$ and
\begin{equation}\label{tayfe}
f(X_t) - f(X_s) - \nabla f(X_s)(X_t-X_s) = O(|t-s|^{(\theta +1) \gamma}),
\end{equation}
where $(\theta +1)\gamma + \gamma < 1$ for every $\gamma < H$. In particular, the classical Sewing lemma fails.
\end{lemma}

The proof of Lemma \ref{TRPlemma} is given in Section \ref{appendixsec}. Next, we illustrate Theorem \ref{mainTH} with the almost sure convergence rate.

\begin{corollary}\label{thintr1}
Let $X= (X_1,\ldots,X_d)$ be a $d$-dimensional fractional Brownian motion with parameter $\frac{1}{3} < H < \frac{1}{2}$. Let $\mathbf{X} = (X,\mathbb{X})$ be the geometric rough path given in Proposition \ref{liftingTH}. Assume that $f:\mathbb{R}^d\rightarrow \mathbb{R}^d \in C^2_b$ and fix $\rho >0$ such that $0< \rho < 2H-\frac{1}{2}$. There exists a square-integrable random variable $C$ such that

\begin{equation}\label{asfx}
\Bigg| \int_0^T f(X_s)d\mathbf{X}_s - I^0(2^{-n},f(X),dX)(T)\Bigg|\le C 2^{-n\rho}\rightarrow 0,
\end{equation}
almost surely, as $n\rightarrow +\infty$.
\end{corollary}
\begin{proof}
Since
$$\mathbf{D}_{r_1}f(X_s) - \mathbf{D}_{r_2}f(X_s) = 0,$$
for every $0\le r_1 < r_2 \le s\le T$, we can take any $\eta = 1$ in Assumption S2. A direct application of Lemma \ref{TRPlemma}, Theorems \ref{gaussiancase} and \ref{mainTH} and Example \ref{c2bex} yields

\begin{equation}\label{asfx1}
\mathbb{E}\Bigg| \int_0^T f(X_s)d\mathbf{X}_s - I^0(2^{-n},f(X),dX)(T)\Bigg|^2 \lesssim \max\Big\{\|\nabla f\|^2_1, \|f\|_1^2, |\nabla f(0)|^2\Big\}2^{-n(4H-1)},
\end{equation}
for every $n \ge 1$ sufficiently large. Let us define

$$C = \Bigg(\sum_{m\ge 1} 2^{m\rho} \Big| \int_0^T f(X_s)d\mathbf{X}_s - I^0(2^{-n},f(X),dX)(T)\Big|^2 \Bigg)^{\frac{1}{2}}.$$
Then, (\ref{asfx1}) implies $\mathbb{E}|C|^2<\infty$. Therefore, (\ref{asfx}) holds true.

\end{proof}





\subsection{The case of rough differential equations}\label{sdesec}
In this section, we apply Theorem \ref{mainTH} to the class of rough differential equations of the form (\ref{rdeEQ}), where $\mathbf{X} = (X,\mathbb{X})$ is a $\gamma$-H\"older geometric rough path lift for the fractional Brownian motion $X$ with parameter $\frac{1}{3} < H < \frac{1}{2}$ and $\gamma < H$. Just like in the proof of Proposition \ref{rdeex}, let $J_t$ be the Jacobian of the solution $Y_t$ where $Y_0 = \xi$ is an arbitrary initial condition and let $J^{-1}_t$ be the inverse of $J_t$. We recall the following fundamental result due to \cite{cass2013int} and \cite{cass2015sm}:
\begin{equation}\label{intJ2}
\|Y\|_{\gamma},\| J\|_{\gamma}, \|J^{-1}\|_{\gamma} \in \cap_{q\ge 1}L^q(\mathbb{P}),
\end{equation}
and
\begin{equation}\label{intJ3}
\|Y\|_\infty, \| J\|_{\infty}, \|J^{-1}\|_{\infty} \in \cap_{q\ge 1}L^q(\mathbb{P}).
\end{equation}
See also Remark 2.7 in \cite{cass2015sm}. It is convenient to work with the norms

$$\|f\|_{\infty,\kappa}:= \|f\|_\infty + \|f\|_\kappa,$$
for a one-parameter function $f$ and $0 < \kappa \le 1$.





The following result is an almost immediate consequence of the H\"older-type estimates (\ref{intJ2}) and (\ref{intJ3}). The proof of Lemma \ref{etarde} is given in Section \ref{appendixsec}. 

\begin{lemma}\label{etarde}
For a given $\frac{1}{3} < \gamma < H < \frac{1}{2}$, there exists a constant $C$ which depends on the moments of $\|J\|_{\infty,\gamma}, \|J^{-1}\|_{\infty,\gamma}, \|Y\|_{\infty,\gamma}, \|\nabla V\|_\infty, H$ and $T$ such that

\begin{equation}
\|Y_t - Y_s\|^2_{\mathbb{D}^{1,2}(\mathbb{R}^d)}\le C |t-s|^{2\gamma},
\end{equation}
for every $s,t \ge 0$.

\end{lemma}
Next, we illustrate Theorem \ref{mainTH} with the almost sure convergence rate.
\begin{corollary}\label{thintr2}
Let $X= (X_1,\ldots,X_d)$ be a $d$-dimensional fractional Brownian motion with parameter $\frac{1}{3} < H< \frac{1}{2}$. Let $\mathbf{X} = (X,\mathbb{X})$ be the geometric rough path given in Proposition \ref{liftingTH} and $V \in C^3_b \big(\mathbb{R}^d,\mathcal{L}(\mathbb{R}^d,\mathbb{R}^d)\big)$. Let $Y$ be the solution of the rough differential equation

$$
Y_t = Y_0 + \int_0^t V(Y_s)d\mathbf{X}_s; 0\le t\le T.
$$
Fix $\frac{1}{3} < \eta < H$ and $\rho$ such that $0 < \rho <  \eta + 2H-1$. Then, there exists a square-integrable random variable $C$ such that

\begin{equation}\label{asrde}
\Bigg| \int_0^T Y_sd\mathbf{X}_s - I^0(2^{-n},Y,dX)(T)\Bigg|\le C 2^{-n\rho} \rightarrow 0,
\end{equation}
almost surely, as $n\rightarrow +\infty$.
\end{corollary}
\begin{proof}
First, we observe the solution of the rough differential equation (\ref{rdeEQ}) belongs to $\mathbb{D}^{1,2}(L_R(\mathbb{R}^d))$. Indeed, the proof follows the same lines of Lemma \ref{TRPlemma}, (\ref{intJ2}), (\ref{intJ3}) and the well-known facts $Y_t \in \mathbb{D}^{1,2}(\mathbb{R}^d)$ for every $t\ge 0$, $\mathbf{D}_sY_t = J_t \circ J^{-1}_s \circ V(Y_s) \mathds{1}_{[0,t]}(s)$. Therefore, we omit the details. Moreover,  (\ref{intJ2}) and (\ref{intJ3}) imply

$$\mathbb{E}|\mathbf{D}_{r_1}Y_s - \mathbf{D}_{r_2} Y_s|^2 \lesssim |r_1-r_2|^{2\eta},$$
on $0\le r_1 < r_2\le s\le T$, for any $\eta$ such that $\frac{1}{3} < \eta < H < \frac{1}{2}$. Then, Lemma \ref{etarde} yields Assumption S2 and S1 are fulfilled. By applying Proposition \ref{rdeex}, Theorems \ref{gaussiancase}, \ref{mainTH} and noticing the leading term in the right-hand side of (\ref{rateabs}) is $2^{-n(2(\eta + 2H-1))}$, we get

\begin{equation}\label{l2rde}
\mathbb{E}\Bigg| \int_0^T Y_sd\mathbf{X}_s - I^0(2^{-n},Y,dX)(T)\Bigg|^2 \lesssim 2^{-n\big\{2(\eta + 2H-1)\big\}},
\end{equation}
for every $n \ge 1$ sufficiently large. Let us define

$$C = \Bigg(\sum_{m\ge 1} 2^{m\rho} \Big| \int_0^T Y_sd\mathbf{X}_s - I^0(2^{-n},Y,dX)(T)\Big|^2 \Bigg)^{\frac{1}{2}}.$$
Then, (\ref{l2rde}) implies $\mathbb{E}|C|^2<\infty$. Therefore, (\ref{asrde}) holds true.

\end{proof}

\section{Proof of Theorem \ref{gaussiancase}}\label{proofTh1}
In this section, in order to keep notation simple, we write $f_{s,t}:= f_t - f_s$ for a one-parameter function $f$ defined over $\mathbb{R}_+$.  Before we present the proof of Theorem \ref{gaussiancase}, it is convenient to summarize the main idea. Under the assumptions of Theorem \ref{gaussiancase}, it is enough to prove that

\begin{equation}\label{I1}
\lim_{\epsilon \rightarrow 0^+}\frac{1}{\epsilon} \int_0^t\Big\langle Y'_s, \text{Anti}(\mathbb{X}_{s,s+\epsilon})\Big\rangle_{\mathbf{F}} ds=0
\end{equation}
in probability, where $\langle \cdot, \cdot \rangle_{\mathbf{F}}$ denotes the Frobenius inner product on the space of $d\times d$-matrices.
 Indeed, if $(Y,Y') \in \mathcal{D}_X (\mathbb{R}^d)$, then we can take advantage of decomposition (\ref{lineardep}) and the geometric property of $\mathbb{X}$ to write

\begin{eqnarray}
\nonumber\frac{1}{\epsilon}\Big\langle \frac{Y_s + Y_{s+\epsilon}}{2}, X_{s,s+\epsilon}\Big\rangle &=& \frac{1}{\epsilon}\Big\langle Y_s, X_{s,s+\epsilon}    \Big\rangle + \frac{1}{2\epsilon} \big\langle Y'_s X_{s,s+\epsilon},   X_{s,s+\epsilon}\big\rangle +  \frac{1}{2\epsilon} \big\langle R^Y_{s,s+\epsilon},   X_{s,s+\epsilon}\big\rangle\\
\nonumber&=& \frac{1}{\epsilon}\Big\langle Y_s, X_{s,s+\epsilon}    \Big\rangle+ \frac{1}{2\epsilon} \big\langle Y'_s, X_{s,s+\epsilon} \otimes X_{s,s+\epsilon}\big\rangle_{\mathbf{F}}  +  o_\mathbb{P}(1)\\
&=& \label{redu}\frac{1}{\epsilon}\Big\langle Y_s, X_{s,s+\epsilon}    \Big\rangle+\frac{1}{\epsilon} \big\langle Y'_s, \text{Sym}(\mathbb{X}_{s,s+\epsilon})\big\rangle_{\mathbf{F}}  +  o_\mathbb{P}(1)\\
\nonumber&=& \frac{1}{\epsilon}\Big\langle Y_s, X_{s,s+\epsilon}    \Big\rangle+ \frac{1}{\epsilon} \big\langle Y'_s, \mathbb{X}_{s,s+\epsilon}\big\rangle_{\mathbf{F}}\\
\nonumber&-& \frac{1}{\epsilon} \big\langle Y'_s, \text{Anti}(\mathbb{X}_{s,s+\epsilon})\big\rangle_{\mathbf{F}}  +  o_\mathbb{P}(1),
\end{eqnarray}
where $\langle \cdot, \cdot \rangle$ above denotes the standard inner product on $\mathbb{R}^d$. We will analyze

$$
\epsilon^{-1}\int_0^t \big\langle Y^{'}_{s}, \text{Anti}(\mathbb{X}_{s,s+\epsilon})\rangle_{\mathbf{F}} ds = \epsilon^{-1} \int_0^t  \textit{tr}\Big[  (Y^{'}_{s})^\top \text{Anti}(\mathbb{X}_{s,s+\epsilon})\Big]ds,
$$
where $\top$ denotes the transpose operation. By using Proposition \ref{liftingTH}, (\ref{multfor}) and Fubini's theorem for Skorohod integrals (see Prop. 10.3 in \cite{krukrusso}), we observe the $(i,j)$-th element of the matrix $\epsilon^{-1}\int_0^t (Y^{'}_{s})^\top \text{Anti}(\mathbb{X}_{s,s+\epsilon})ds$ is given by

\begin{eqnarray}
\nonumber\sum_{\ell=1}^d\frac{1}{\epsilon}\int_0^t Y^{',i\ell}_{s} \big(\text{Anti}(\mathbb{X}_{s,s+\epsilon})\big)_{\ell j} ds &=& \frac{1}{2\epsilon}\sum_{\ell=1}^d\int_0^t \Bigg(\int_{r-\epsilon}^r Y^{',i\ell}_{s}\big\{ X^j_{s,r}e_\ell  - X^\ell_{s,r}e_j\big\} ds\Bigg)\boldsymbol{\delta}X_r\\
\label{alSYM}& &\\
\nonumber&+& \frac{1}{2\epsilon}\sum_{\ell=1}^d\int_0^t \Big\langle \mathbf{D}_\cdot Y^{',i \ell}_{s}, [e_\ell X^j_{s,\cdot} - e_j X^\ell_{s,\cdot}  ] \mathds{1}_{[s,s+\epsilon]}(\cdot)\Big \rangle_{L_R(\mathbb{R}^d)}ds,
\end{eqnarray}
for every $t\in [0,T]$, $\epsilon>0$. In the sequel, we are going to fix $i,\ell,j\in \{1, \ldots, d\}$ and $t\in [0,T]$ and prove that the second component in the right-hand side of (\ref{alSYM}) vanishes in $L^1(\mathbb{P})$ as $\epsilon \downarrow 0$.

Let us write $\mathbf{D}_rY^{',i\ell}_{s} = (\mathbf{D}^1_rY^{',i\ell}_{s} , \ldots, \mathbf{D}^d_rY^{',i\ell}_{s} )$ in $L_{R}(\mathbb{R}^d)$. Then, we have

\begin{small}
$$\Big\langle \mathbf{D}_\cdot Y^{',i \ell}_{s}, [e_\ell X^j_{s,\cdot} - e_j X^\ell_{s,\cdot}  ] \mathds{1}_{[s,s+\epsilon]}(\cdot)\Big \rangle_{L_R(\mathbb{R}^d)}
$$
$$ = \int_s^{s+\epsilon} \mathbf{D}^\ell_r Y^{',i\ell}_{s} X^j_{s,r} \partial_r R(r,T)dr$$
$$ - \int_s^{s+\epsilon} \mathbf{D}^j_r Y^{',i\ell}_{s} X^\ell_{s,r} \partial_r R (r,T)dr $$
$$+\frac{1}{2}\int_{[0,T]^2 \setminus D} \Big(\mathbf{D}^\ell_{r_1} Y^{',i\ell}_{s} - \mathbf{D}^\ell_{r_2} Y^{',i\ell}_{s}\Big)  \Big( X^j_{s,r_1}\mathds{1}_{[s,s+\epsilon]}(r_1) - X^j_{s,r_2}\mathds{1}_{[s,s+\epsilon]}(r_2) \Big) |\mu|(dr_1dr_2)$$

$$-\frac{1}{2}\int_{[0,T]^2 \setminus D} \Big(\mathbf{D}^j_{r_1} Y^{',i\ell}_{s} - \mathbf{D}^j_{r_2} Y^{',i\ell}_{s}\Big) \Big(X^\ell_{s,r_1}\mathds{1}_{[s,s+\epsilon]}(r_1) - X^\ell_{s,r_2} \mathds{1}_{[s,s+\epsilon]}(r_2) \Big) |\mu|(dr_1dr_2)$$
$$
 =: I^1_s + I^2_s + I^3_s+ I^4_s \quad a.s.$$
\end{small}
The components $I^1$ and $I^2$ can be estimated as follows. In order to keep notation simple, we set $\beta = \frac{\alpha}{2} + 1 \in \big(\frac{1}{3},\frac{1}{2}\big)$. By using (\ref{intassump}), Assumption C(iii) and H\"older's inequality, we get

$$\mathbb{E}\Big|\frac{1}{\epsilon}\int_0^t I^1_sds\Big|\lesssim \epsilon^{\beta} \int_0^t \Bigg(\frac{1}{\epsilon}\int_{s}^{s+\epsilon}\partial_r R(r,T)dr\Bigg)ds\rightarrow 0$$
as $\epsilon \rightarrow 0^+$. The term $I^2$ is similar. By symmetry, the analysis of the term $I^3$ is similar to $I^4$. Again, by using (\ref{intassump}), Assumption C and H\"older's inequality, we get

$$\mathbb{E}\Big|\frac{1}{\epsilon}\int_0^t I^3_sds\Big|\lesssim \frac{1}{\epsilon}\int_0^t \int_{\{s < r_2 < r_1 < s+\epsilon\}}\Big\{(r_1-r_2)^{\beta + \alpha} + (r_1-r_2)^\beta\phi(r_1,r_2)\Big\}dr_1dr_2ds$$
$$+\frac{1}{\epsilon}\int_0^t \int_{\{ r_2 \le s< r_1 < s+\epsilon\}}\Big\{(r_1-s)^{\beta}(r_1-r_2)^\alpha + (r_1-s)^\beta\phi(r_1,r_2)\Big\}dr_1dr_2ds$$
$$+\frac{1}{\epsilon}\int_0^t \int_{\{ s< r_2 < s+\epsilon\le r_1\}}\Big\{(r_2-s)^{\beta}(r_1-r_2)^\alpha + (r_2-s)^\beta\phi(r_1,r_2)\Big\}dr_1dr_2ds,$$
for $\epsilon >0$. By invoking Assumption C(iv), we observe

\begin{equation}\label{phiep1}
\epsilon^{-1}\int_{\{s < r_2 < r_1 < s+\epsilon\}}(r_1-r_2)^{\beta} \phi(r_1,r_2)dr_1dr_2 \lesssim \epsilon^{\alpha+1 + \beta},
\end{equation}
for every $s \in [0,t]$. Moreover,
\begin{eqnarray}
\nonumber\epsilon^{-1}\int_{\{s < r_2  < s+\epsilon\le r_1\}}(r_2-s)^{\beta} \phi(r_1,r_2)dr_1dr_2 &= &  \epsilon^{- 1} \int_{s}^{s+\epsilon}\int_{s+\epsilon}^T(r_2-s)^\beta\phi(r_1,r_2)dr_1dr_2\\
\nonumber&\lesssim&  \epsilon^{-1}[T^{\frac{\alpha+2}{2}} - (s+\epsilon)^{\frac{\alpha+2}{2}}]\psi(s)\\
\nonumber&\times& \int_s^{s+\epsilon}(r_2-s)^{\beta}dr_2\\
\label{phiep2}&\lesssim&  \epsilon^{\beta}T^{\frac{\alpha+2}{2}}\psi(s),
\end{eqnarray}
and

\begin{eqnarray}
\nonumber\epsilon^{-1}\int_{\{r_2\le s < r_1 < s+\epsilon\}}(r_1-s)^{\beta} \phi(r_1,r_2)dr_1dr_2 &= &  \epsilon^{- 1} \int_{s}^{s+\epsilon}\int_{0}^s(r_1-s)^\beta\phi(r_1,r_2)dr_2dr_1\\
\nonumber&\lesssim&  \epsilon^{-1}\int_s^{s+\epsilon}(r_1-s)^{\beta}s^{\frac{\alpha+2}{2}}\psi(r_1)dr_1\\
\label{phiep3}&\lesssim& s^{\frac{\alpha+2}{2}}\psi(s) \epsilon^{\beta},
\end{eqnarray}
for each $s \in [0,t]$. Moreover,

\begin{equation}\label{syI4Bis}
\int_{\{s < r_2 < r_1 < s+\epsilon\}}(r_1-r_2)^{\alpha + \beta}dr_1dr_2 \lesssim   \epsilon^{\alpha + 2 +\beta},\quad \int_{\{s < r_2 < s+\epsilon\le r_1\}}(r_1-r_2)^{\alpha}dr_1dr_2 \lesssim \epsilon^{\alpha+2},
\end{equation}
and

\begin{equation}\label{syI4.1}
\int_{\{ r_2\le s < r_1< s+\epsilon\}}(r_1-r_2)^{\alpha}dr_1dr_2 \lesssim \big| s^{\alpha+2} + \epsilon^{\alpha+2} - (s+\epsilon)^{\alpha+2} \big|\lesssim \epsilon^{\alpha+2},
\end{equation}
for every $s \in [0,t]$. Then, (\ref{phiep1}), (\ref{phiep2}), (\ref{phiep3}), (\ref{syI4Bis}) and (\ref{syI4.1}), allow us to conclude

$$\mathbb{E}\frac{1}{\epsilon}\Bigg| \int_0^t I^3_sds\Bigg|\lesssim \epsilon^{\beta +\alpha+1}\rightarrow 0,$$
as $\epsilon \downarrow 0$, because $\alpha + 1 + \beta >0 $. This shows that the second part of (\ref{alSYM}) vanishes.

\subsection{Estimating the Skorohod integral in (\ref{alSYM})}\label{skorohodsection}

Let us now devote our attention to the first component in the right-hand side of (\ref{alSYM}), namely the Skorohod integral.
In the sequel, we are going to fix $i,\ell,j\in \{1, \ldots, d\}$ and $t\in [0,T]$ and prove that the first part in the right-hand side of
(\ref{alSYM}) vanishes in $L^1(\mathbb{P})$ as $\epsilon \downarrow 0$.

In the sequel, to keep notation simple, we set

$$u^{i\ell,j}_{r-\epsilon,r}:=\int_{r-\epsilon}^r Y^{',i\ell}_{s}X^j_{s,r}ds = \int_0^\infty X^j_{s,r} \mathds{1}_{(r-\epsilon, r)}(s)Y^{',i\ell}_{s}ds.$$

The following technical result is an almost immediate consequence of the assumptions in Theorem \ref{gaussiancase}. Indeed, it is an application of Lemma \ref{funcINT}.
\begin{lemma}\label{rm}
Suppose that the assumptions of Theorem \ref{gaussiancase} hold true. Then, for every $i, \ell,j\in \{1,\ldots, d\}$ and $\epsilon>0$, we have 

\begin{equation}\label{Daug}
\big( u^{i\ell,j}_{\cdot-\epsilon, \cdot}e_\ell - u^{i\ell,\ell}_{\cdot-\epsilon, \cdot}e_j  \big)\in \mathbb{D}^{1,2}(L_R(\mathbb{R}^d)).
\end{equation}

In particular, the (only) non-null $\ell$-th column of $\mathbf{D}_v u^{i\ell,j}_{r-\epsilon, r}e_\ell$ equals to
\begin{equation}\label{prod1}
\int_{r-\epsilon}^r \Big\{ X^j_{s,r} \mathbf{D}_vY^{',i\ell}_{s} + Y^{',i\ell}_{s}\mathds{1}_{[s,r]}(v)e_j\Big\}
ds
\end{equation}
and the (only) non-null $j$-th column of $\mathbf{D}_v u^{i\ell,\ell}_{r-\epsilon, r}e_j$ equals to
\begin{equation}\label{prod2}
 \int_{r-\epsilon}^r \Big\{ X^\ell_{s,r} \mathbf{D}_vY^{',i\ell}_{s}  + Y^{',i\ell}_{s} \mathds{1}_{[s,r]}(v)e_\ell\Big\}
ds
\end{equation}
a.s. for every $v,r\in [0,T]$ and $\epsilon >0$.
\end{lemma}
By (\ref{Daug}) and (\ref{cs1}),

$$\frac{1}{\epsilon}\int_0^t \Bigg(\int_{r-\epsilon}^r Y^{',i\ell}_{s}\big\{ X^j_{s,r} e_\ell  - X^\ell_{s,r} e_j\big\} ds\Bigg)\boldsymbol{\delta}X_r= \boldsymbol{\delta} \Big( \frac{1}{\epsilon} \big(u^{i\ell,j}_{\cdot-\epsilon,\cdot}e_\ell - u^{i\ell,\ell}_{\cdot-\epsilon,\cdot}e_j\big)\mathds{1}_{[0,t]} \Big),$$
where

$$ \Bigg\|  \boldsymbol{\delta} \Big(\frac{1}{\epsilon} \big(u^{i\ell,j}_{\cdot-\epsilon,\cdot}e_\ell - u^{i\ell,\ell}_{\cdot-\epsilon,\cdot}e_j \big)\mathds{1}_{[0,t]}\Big) \Bigg\|_{L^2(\mathbb{P})}\lesssim \Bigg(\Big\| \mathbb{E} \Big[ \frac{1}{\epsilon}\big( u^{i\ell,j}_{\cdot-\epsilon,\cdot}e_\ell - u^{i\ell,\ell}_{\cdot-\epsilon,\cdot}e_j\big)\mathds{1}_{[0,t]}\Big]   \Big\|_{L_R(\mathbb{R}^d)}$$
\begin{equation}\label{skoest}
+ \Big\| \mathbf{D}_\cdot \Big[\frac{1}{\epsilon}\big( u^{i\ell,j}_{\cdot-\epsilon,\cdot}\mathds{1}_{[0,t]}e_\ell - u^{i\ell,\ell}_{\cdot-\epsilon,\cdot}\mathds{1}_{[0,t]}e_j   \big)\Big]   \Big\|_{L^2(\Omega; L_{2,R}(\mathbb{R}^{d\times d}))}\Bigg) =: J_1(\epsilon,t)+J_2(\epsilon,t),
\end{equation}
for every $t\in [0,T]$ and $\epsilon>0$.

\subsection{Analysis of $J_1(\epsilon,t)$}
In the sequel, we set $\beta = \frac{\alpha}{2} +1$, where $-\frac{4}{3} < \alpha < -1$. To shorten notation, we set

$$U^{i\ell,j,\epsilon}_{s_1} := \mathbb{E}[u^{i\ell,j}_{s_1-\epsilon, s_1}] =  \int_{s_1-\epsilon}^{s_1} \mathbb{E}[Y^{',i\ell}_{s}X^j_{s,s_1} ]ds$$ and

$$\Delta_{(\mathbf{s}; t)} U^{i\ell,j,\epsilon} :=  U^{i\ell,j,\epsilon}_{s_1}\mathds{1}_{[0,t]}(s_1)-U^{i\ell,j,\epsilon}_{s_2}\mathds{1}_{[0,t]}(s_2) $$
for $\mathbf{s}=(s_1,s_2)\in [0,T]^2\setminus D$. Then, for $\ell\neq j$, we have

$$\mathbb{E}\Big[\frac{1}{\epsilon} \big(u^{i\ell,j}_{s_1-\epsilon, s_1}e_\ell - u^{i\ell,\ell}_{s_1-\epsilon, s_1}e_j\big) \Big]\mathds{1}_{[0,t]}(s_1) = \frac{1}{\epsilon} \Big(U^{i\ell,j,\epsilon}_{s_1}e_\ell - U^{i\ell,\ell,\epsilon}_{s_1}e_j\Big) \mathds{1}_{[0,t]}(s_1)$$
and

\begin{eqnarray}
\nonumber\Bigg\|\mathbb{E}\Big[ \frac{1}{\epsilon} \big( u^{i\ell,j}e_\ell - u^{i\ell,\ell}e_j\big)\mathds{1}_{[0,t]}\Big]\Bigg\|^2_{L_R(\mathbb{R}^d)}&\lesssim& \int_0^t \Big|\frac{1}{\epsilon} \int_{r-\epsilon}^r \mathbb{E}[Y^{',i\ell}_{s}X^j_{s,r}]ds \Big|^2\big|\partial_r R(r,T)\big|dr\\
\nonumber& &\\
\nonumber& +& \int_0^t \Big| \frac{1}{\epsilon} \int_{r-\epsilon}^r \mathbb{E}[Y^{',i\ell}_{s} X^\ell_{s,r}]ds \Big|^2\Big|\partial_r R(r,T)\big|dr\\
\nonumber& &\\
\nonumber&+& \int_{[0,T]^2\setminus D}\Big|\frac{1}{\epsilon} \Delta_{(\mathbf{s}; t)} U^{i\ell,j,\epsilon} \Big|^2 |\mu|(ds_1ds_2)   \\
\nonumber& &\\
\label{al1}&+&  \int_{[0,T]^2\setminus D}\Big|\frac{1}{\epsilon} \Delta_{(\mathbf{s}; t)} U^{i\ell,\ell,\epsilon} \Big|^2 |\mu|(ds_1ds_2) .
\end{eqnarray}
By H\"older's inequality, assumption (\ref{intassump}) and Assumption C (ii), we have

\begin{eqnarray}
\nonumber\int_0^t \Big| \frac{1}{\epsilon}\int_{r-\epsilon}^r \mathbb{E}[Y^{',i\ell}_{s}X^j_{s,r}]ds \Big|^2\big|\partial_r R (r,T)\big|dr&\le& \epsilon^{-2}\int_0^t \Big(\int_{r-\epsilon}^r(r-s)^\beta ds\Big)^2\big|\partial_r R (r,T)\big|dr\\
\label{al2}& &\\
\nonumber&\lesssim& \epsilon^{-2}\epsilon^{2(\beta+1)}\int_0^T\big|\partial_r R (r,T)\big|dr.
\end{eqnarray}
By symmetry, the estimate (\ref{al2}) also holds for the second term in the right-hand side of (\ref{al1}). Now, we split

$$\int_{[0,T]^2\setminus D}\Big|\frac{1}{\epsilon} \Delta_{(\mathbf{s}; t)} U^{i\ell,j,\epsilon} \Big|^2 |\mu|(ds_1ds_2)  = 2\int_{0 < s_1 < t < s_2 \le T}\Big|\frac{1}{\epsilon} \Delta_{(\mathbf{s}; t)} U^{i\ell,j,\epsilon} \Big|^2 |\mu|(ds_1ds_2)$$
\begin{equation}\label{split1}
+ \int_{[0,t]^2 \setminus D}\Big|\frac{1}{\epsilon} \Delta_{(\mathbf{s}; t)} U^{i\ell,j,\epsilon} \Big|^2 |\mu|(ds_1ds_2).
\end{equation}
In the sequel, we will take advantage of assumption C (i). In case, $s_1 < t < s_2$, mean value theorem, assumption (\ref{intassump}), H\"older's inequality and Assumption C (ii) yield

$$
\Big|\frac{1}{\epsilon}\Delta_{(\mathbf{s}; t)} U^{i\ell,j,\epsilon} \Big|^2 = \Big|\frac{1}{\epsilon}\int_{s_1-\epsilon}^{s_1} \mathbb{E}[Y^{',i\ell}_{s} X^j_{s,s_1}]ds\Big|^2 =  \Big| \mathbb{E}[Y^{',i\ell}_{r_1}X^j_{r_1,s_1}]  \Big|^2
\lesssim   (s_1-r_1)^{2\beta}\le \epsilon^{2\beta},
$$
for some $r_1$ satisfying $s_1-\epsilon < r_1 < s_1 < t < s_2$. Then,

\begin{eqnarray}
\nonumber\int_{0 < s_1 < t < s_2 \le T}\Big|\frac{1}{\epsilon} \Delta_{(\mathbf{s}; t)} U^{i\ell,j,\epsilon} \Big|^2|s_1-s_2|^\alpha ds_1ds_2 &\lesssim&  \epsilon^{2\beta}\int_t^T \int_0^t(s_2-s_1)^{\alpha}ds_1ds_2\\
\nonumber&\lesssim& \epsilon^{2\beta}\int_t^T \{(s_2-t)^{\alpha+1} - s_2^{\alpha+1}\}ds_2\\
\label{bett}&\rightarrow& 0,
\end{eqnarray}
as $\epsilon \downarrow 0$.
In addition,

$$\int_{0 < s_1 < t < s_2 \le T}\Big|\frac{1}{\epsilon} \Delta_{(\mathbf{s}; t)} U^{i\ell,j,\epsilon} \Big|^2 \phi(s_1,s_2) ds_1ds_2\lesssim \epsilon^{2\beta}\int_{0 < s_1 < t < s_2 \le T}\phi(s_1,s_2) ds_1ds_2 \rightarrow 0,$$
as $\epsilon \downarrow 0$.

The case $s_1 < t$ and $s_2 < t$ is trickier.  At first, we observe $a \mapsto \mathbb{E}[Y^{',i\ell}_{a} X^j_{a,b}]$ is continuous for every $b$. Hence,

\begin{equation}\label{fatou1}
\lim_{\epsilon \downarrow 0}\Big|\frac{1}{\epsilon} \Delta_{(\mathbf{s}; t)} U^{i\ell,j,\epsilon} \Big|^2=0,
\end{equation}
for each $\mathbf{s}=(s_1,s_2)\in [0,t]^2\setminus D$. If $s_2 < s_1 < t$, then we shall write

$$\frac{1}{\epsilon}\int_{s_1-\epsilon}^{s_1} \mathbb{E}[Y^{',i\ell}_{s} X^j_{s,s_1}]ds = \frac{1}{\epsilon}\int_{s_2-\epsilon}^{s_2} \mathbb{E}[Y^{',i\ell}_{s}X^j_{s,s_1}]ds$$
$$+\frac{1}{\epsilon}\int_{s_2}^{s_1} \mathbb{E}[Y^{',i\ell}_{s}X^j_{s,s_1}]ds-\frac{1}{\epsilon}\int_{s_2-\epsilon}^{s_1-\epsilon} \mathbb{E}[Y^{',i\ell}_{s}X^j_{s,s_1}]ds,$$
and we arrive at

$$\Big|\frac{1}{\epsilon} \Delta_{(\mathbf{s}; t)} U^{i\ell,j,\epsilon} \Big|^2 = \Bigg|\frac{1}{\epsilon}\int_{s_1-\epsilon}^{s_1} \mathbb{E}[Y^{',i\ell}_{s}X^j_{s,s_1}]ds - \frac{1}{\epsilon}\int_{s_2-\epsilon}^{s_2} \mathbb{E}[Y^{',i\ell}_{s}X^j_{s,s_2}]ds\Bigg|^2
$$
 \begin{equation}\label{al3}
 = \Bigg|\frac{1}{\epsilon}\int_{s_2-\epsilon}^{s_2} \mathbb{E}[Y^{',i\ell}_{s}X^j_{s_2,s_1}]ds + \frac{1}{\epsilon}\int_{s_2}^{s_1} \mathbb{E}[Y^{',i\ell}_{s}X^j_{s,s_1}]ds- \frac{1}{\epsilon}\int_{s_2-\epsilon}^{s_1-\epsilon} \mathbb{E}[Y^{',i\ell}_{s}X^j_{s,s_1}]ds\Bigg|^2.
\end{equation}

Mean value theorem, assumption (\ref{intassump}), H\"older's inequality and Assumption C (ii) yield

\begin{equation}\label{al4}
\Big|\frac{1}{\epsilon}\int_{s_2-\epsilon}^{s_2} \mathbb{E}[Y^{',i\ell}_{s}X^j_{s_2,s_1}]ds\Big|^2\lesssim (s_1-s_2)^{2\beta},
\end{equation}
for every $s_2 < s_1 < t$. In addition, the same argument yields

\begin{equation}\label{al5}
\Big|\frac{1}{\epsilon}\int_{s_2}^{s_1} \mathbb{E}[Y^{',i\ell}_{s}X^j_{s,s_1}]ds\Big|^2\lesssim  \Bigg(\int_{s_2}^{s_1}(s_1-s)^\beta ds \epsilon^{-1}\Bigg)^2\lesssim  (s_1-s_2)^{2\beta},
\end{equation}
whenever $(s_1-s_2) < \epsilon$ and $s_2 < s_1 < t$. Similarly,

\begin{eqnarray}
\nonumber\Big|\frac{1}{\epsilon}\int_{s_2-\epsilon}^{s_1-\epsilon} \mathbb{E}[Y^{',i\ell}_{s}X^j_{s,s_1}]ds\Big|^2&\lesssim& \epsilon^{-2}\Bigg(\int_{s_2-\epsilon}^{s_1-\epsilon}(s_1-s)^\beta ds\Bigg)^2\\
\nonumber&\lesssim& \epsilon^{-2}(s_1-s_2+\epsilon)^{2(\beta+1)}\\
\label{al6}&\lesssim&  (s_1-s_2+\epsilon)^{2\beta},
\end{eqnarray}
whenever $(s_1-s_2) < \epsilon$ and $s_2 < s_1 < t$.

 We observe $|s_1-s_2|^{2\beta}$ is integrable w.r.t. the positive measures $|s_1-s_2|^\alpha ds_1ds_2$ and $\phi(s_1,s_2)ds_1ds_2$ (recall $2\beta+\alpha+1 >0$). Then, (\ref{fatou1}), the estimates (\ref{al3}), (\ref{al4}), (\ref{al5}), (\ref{al6}) and Assumption C(i) allow us to apply bounded convergence theorem to get

\begin{equation}\label{al7}
  \int_{\{\mathbf{s}; s_2 < s_1 < t, (s_1-s_2) < \epsilon\}} \Big|\frac{1}{\epsilon} \Delta_{(\mathbf{s};t)}U^{i\ell,j,\epsilon} \Big|^2 |\mu|(ds_1ds_2)\rightarrow 0,
\end{equation}
as $\epsilon \downarrow 0$. Now, Mean Value theorem yields

$$\Big|\frac{1}{\epsilon} \Delta_{(\mathbf{s}; t)} U^{i\ell,j,\epsilon} \Big|^2\mathds{1}_{\{\mathbf{s}; s_2 < s_1 < t, (s_1-s_2) \ge \epsilon\} }$$
\begin{equation}\label{al8}
= \Big| \mathbb{E}\big[Y^{',i\ell}_{\bar{s}_1(\epsilon)}X^j_{\bar{s}_1(\epsilon),s_1}\big]  - \mathbb{E}\big[Y^{',i\ell}_{\bar{s}_2(\epsilon)} X^j_{\bar{s}_2(\epsilon),s_2}\big]   \Big|^2\mathds{1}_{\{\mathbf{s}; s_2 < s_1 < t, (s_1-s_2) \ge \epsilon\} }
\end{equation}
for some $(\bar{s}_1(\epsilon),\bar{s}_2(\epsilon))$ satisfying $s_1-\epsilon < \bar{s}_1(\epsilon) < s_1$ and $s_2-\epsilon < \bar{s}_2(\epsilon) < s_2$. Jensen's inequality, (\ref{intassump}) and Assumption C (ii) yield

\begin{eqnarray}
\nonumber\Big|\mathbb{E}\big[Y^{',i\ell}_{\bar{s}_1(\epsilon)}X^j_{\bar{s}_1(\epsilon),s_1}\big]\Big|^2 \mathds{1}_{\{\mathbf{s}; s_2 < s_1 < t, (s_1-s_2)\ge \epsilon\} }&\lesssim& (s_1-\bar{s}_1(\epsilon))^{2\beta}\mathds{1}_{\{\mathbf{s}; s_2 < s_1 < t, (s_1-s_2)\ge \epsilon\} }\\
\label{al9}&\lesssim&  (s_1-s_2)^{2\beta}
\end{eqnarray}
and

\begin{eqnarray}
\nonumber\Big|\mathbb{E}\big[Y^{',i\ell}_{\bar{s}_2(\epsilon)}X^j_{\bar{s}_2(\epsilon),s_2}\big]\Big|^2 \mathds{1}_{\{\mathbf{s}; s_2 < s_1 < t, (s_1-s_2)\ge\epsilon\}}&\lesssim&  (s_1-\bar{s}_2(\epsilon))^{2\beta}\mathds{1}_{\{\mathbf{s}; s_2 < s_1 < t, (s_1-s_2)\ge \epsilon\}}\\
\nonumber& \lesssim &  (s_1-s_2+\epsilon)^{2\beta}\mathds{1}_{\{\mathbf{s}; (s_1-s_2)\ge \epsilon\}}\\
\label{al10}&\lesssim& (s_1-s_2)^{2\beta}.
\end{eqnarray}
Summing up (\ref{fatou1}), (\ref{al8}), (\ref{al9}), (\ref{al10}) and invoking bounded convergence theorem and (\ref{dagrowth}), we conclude

\begin{equation}\label{al11}
  \int_{\{\mathbf{s}; s_2 < s_1 < t, (s_1-s_2) \ge \epsilon\}} \Big|\frac{1}{\epsilon} \Delta_{(\mathbf{s};t)}U^{i\ell,j,\epsilon} \Big|^2 |\mu|(ds_1ds_2)\rightarrow 0,
\end{equation}
as $\epsilon \downarrow 0$. Summing up (\ref{al1}), (\ref{al2}), (\ref{split1}), (\ref{bett}), (\ref{al7} and (\ref{al11}) and using symmetry of the terms in (\ref{al1}), we conclude $\lim_{\epsilon \downarrow 0}J_1(\epsilon,t)=0$ in (\ref{skoest}) for each $t\in [0,T]$.

\subsection{Analysis of $J_2(\epsilon,t)$}
This section is devoted to the proof that

$$J_2(\epsilon,t) = \Big\| \mathbf{D}_\cdot \Big[\frac{1}{\epsilon}\big( u^{i\ell,j}_{\cdot-\epsilon,\cdot}\mathds{1}_{[0,t]}e_\ell - u^{i\ell,\ell}_{\cdot-\epsilon,\cdot}\mathds{1}_{[0,t]}e_j   \big)\Big]   \Big\|_{L^2(\Omega; L_{2,R}(\mathbb{R}^{d\times d}))}\rightarrow 0,$$
as $\epsilon\downarrow 0$, for a given $t \in [0,T]$ and $i,\ell,j \in \{1,\ldots, d\}$.

In the sequel, with a slight abuse of notation, when no confusion is possible, we write $|\cdot| = \|\cdot\|_{\mathbb{R}^{d\times d}}$. Let us fix $r\neq v$, $i,\ell,j \in \{1, \ldots, d\}$ with $\ell\neq j$ and $t \in [0,T]$. We recall the notation stated at the beginning of Section \ref{proofTh1}: We write $f_{a,b} = f_b-f_a$ for a one-parameter function $f$. We also recall that $\{e_m\}_{m=1}^d$ is the canonical basis of $\mathbb{R}^d$.


\begin{lemma}\label{pointwiseconv1}
If $Y^{'}$ satisfies the assumptions of Theorem \ref{gaussiancase}, then

\begin{equation}\label{point1}
\lim_{\epsilon\rightarrow 0^+}\frac{1}{\epsilon}\int_{r-\epsilon}^r \Big\{ X^j_{s,r}\mathbf{D}_vY^{',i\ell}_{s} + Y^{',i\ell}_{s}\mathds{1}_{[s,r]}(v)e_j\Big\}
ds=0
\end{equation}
almost surely, for Lebesgue almost all $(r,v)\in [0,T]^2\setminus D$.
\end{lemma}
\begin{proof}
If $r < v$, then $\mathds{1}_{[s,r]}(v)=0 $ whenever $r-\epsilon < s < r$. Then, for Lebesgue almost all $(r,v)$ with $r < v$, we have

\begin{equation}\label{punctual}
\Big|\frac{1}{\epsilon}\int_{r-\epsilon}^r \Big\{ X^j_{s,r}\mathbf{D}_vY^{',i\ell}_{s} + Y^{',i\ell}_{s}\mathds{1}_{[s,r]}(v)e_j\Big\}
ds\Big|= \Big|\frac{1}{\epsilon}\int_{r-\epsilon}^r X^j_{s,r}\mathbf{D}_vY^{',i\ell}_{s}ds\Big|\rightarrow 0
\end{equation}
almost surely as $\epsilon \downarrow 0$. In case $v < r$, we observe $v< r-\epsilon < r$ for every $\epsilon$ sufficiently small and $\mathds{1}_{[s,r]}(v) = 0$ whenever $v < r-\epsilon < s < r$. Then, for each $(r,v)$ with $v < r$, one can take $\epsilon = \epsilon(r,v)$ sufficiently small such that the estimate (\ref{punctual}) holds true as well. Then, we do have the almost sure convergence (\ref{point1}) pointwise in $[0,T]^2\setminus D$.
\end{proof}

We shall write

$$J_2(\epsilon,t)=\mathbb{E}\| h_\epsilon \|^2_{L_{2,R}(\mathbb{R}^{d\times d}))},$$
where $h_\epsilon$ is given by

$$h_\epsilon(v,r)=\mathbf{D}_v \Big[\frac{1}{\epsilon}\big( u^{i\ell,j}_{r-\epsilon,r}\mathds{1}_{[0,t]}(r)e_\ell - u^{i\ell,\ell}_{r-\epsilon,r}\mathds{1}_{[0,t]}(r)e_j   \big)\Big],$$

\begin{eqnarray*}
\| h_\epsilon\|^2_{L_{2,R}(\mathbb{R}^{d\times d})}&=&\int_0^\infty \int_0^\infty |h_\epsilon(v,r)|^2|\partial_v R(v,T)|ds|\partial_r R(r,T)|dr\\
&+& \frac{1}{2}\int_0^\infty \int_{\mathbb{R}^2_+\setminus D}|h_\epsilon(v_1,r) - h_\epsilon(v_2,r)|^2|\mu|(dv_1dv_2)|\partial_r R(r,T)|dr\\
&+&\frac{1}{2}\int_{\mathbb{R}^2_+\setminus D}\int_0^\infty |h_\epsilon(v,r_1) - h_\epsilon(v,r_2)|^2|\partial_v R(v,T)|dv|\mu|(dr_1dr_2)\\
&+& \frac{1}{4}\int_{\mathbb{R}^2_+\setminus D}\int_{\mathbb{R}^2_+\setminus D}|\Delta \Delta h_\epsilon(\mathbf{v},\mathbf{r})|^2|\mu|(dv_1dv_2)|\mu|(dr_1dr_2)\\
&=:& L_1(\epsilon) + L_2(\epsilon) + L_3(\epsilon) +L_4(\epsilon),
\end{eqnarray*}
and
$$\Delta \Delta h_\epsilon(\mathbf{v},\mathbf{r}) = h_\epsilon(v_1,r_1) - h_\epsilon(v_1,r_2) - h_\epsilon(v_2,r_1) + h_\epsilon(v_2,r_2)$$
for $\mathbf{v} = (v_1,v_2),\mathbf{r} = (r_1,r_2) \in \mathbb{R}^2_+\setminus D$.

\

In the sequel, we will analyze each element $L_1(\epsilon)$, $L_2(\epsilon)$,  $L_3(\epsilon)$ and $L_4(\epsilon)$.

\

\textbf{Analysis of $L_1(\epsilon)$}. By using Jensen's inequality, Lemma \ref{rm}, Gaussian moments of $X$, Assumptions A and (\ref{intassump}), one can easily check there exists $p>1$ such that

$$\sup_{0 < \epsilon < 1}\mathbb{E}\int_0^T\int_0^T |h_\epsilon(r,v)|^{2p}|\partial_r R(r,T)\partial_v R(v,T)|drdv < \infty.$$
Lemma \ref{pointwiseconv1} and Vitali convergence theorem allow us to conclude $\mathbb{E}[L_1(\epsilon)]\rightarrow 0$ as $\epsilon \downarrow 0$.

\

\textbf{Analysis of $L_2(\epsilon)$}. Next, we analyze

\begin{small}
\begin{equation}\label{al14}
\mathbb{E}\int_0^t\int_{[0,T]^2\setminus D}\Big| \mathbf{D}_{v_1} \big[\frac{1}{\epsilon}\big( u^{i\ell,j}_{r-\epsilon,r}e_\ell - u^{i\ell,\ell}_{r-\epsilon,r}e_j   \big)\big]  - \mathbf{D}_{v_2} \big[\frac{1}{\epsilon}\big( u^{i\ell,j}_{r-\epsilon,r}e_\ell - u^{i\ell,\ell}_{r-\epsilon,r}e_j   \big)\big]  \Big|^2|\mu|(dv_1dv_2) |R(dr,T)|.
\end{equation}
\end{small}
For this purpose, by symmetry and Lemma \ref{rm}, it is sufficient to bound

\begin{equation}\label{same}
\Big|\frac{1}{\epsilon}\int_{r-\epsilon}^r X^j_{s,r}\Big(\mathbf{D}^m_{v_1}Y^{',i\ell}_{s} - \mathbf{D}^m_{v_2} Y^{',i\ell}_{s}\Big) ds\Big|^2
\end{equation}
for $m\neq j$ and

\begin{equation}\label{Aterm}
\Big|\frac{1}{\epsilon}\int_{r-\epsilon}^r \Big\{ X^j_{s,r} \Big(\mathbf{D}^j_{v_1}Y^{',i\ell}_{s} - \mathbf{D}^j_{v_2} Y^{',i\ell}_{s}\Big) + Y^{',i\ell}_{s}\big[\mathds{1}_{[s,r]}(v_1) - \mathds{1}_{[s,r]}(v_2)\big] \Big\}ds\Big|^2.
\end{equation}
Clearly, we only need to check (\ref{Aterm}) because the term (\ref{same}) is totally analogous. In the sequel, to shorten notation, we denote $A_\epsilon(r,v_1,v_2)$ as the square root of (\ref{Aterm}). By using the same argument given in the proof of Lemma \ref{pointwiseconv1}, we can safely state that

\begin{equation}\label{tens1}
\lim_{\epsilon\rightarrow 0^+} A_\epsilon(r,v_1,v_2) = 0~a.s,
\end{equation}
for each $v_1 \neq v_2$ and $r\in [0,T]$. In the sequel, let us write

$$A_\epsilon(r,v_1,v_2) =\sum_{i=1}^6 A_\epsilon(r,v_1,v_2)\mathds{1}_{E_i(\epsilon)}$$
for $v_1 < v_2$ (without any loss of generality), where

\begin{itemize}
\item $E_1(\epsilon) = \{(r,v_1,v_2); v_1 <  v_2 < r-\epsilon\}$
\item $E_2(\epsilon) = \{(r,v_1,v_2); r < v_1 < v_2 \}$
\item $E_3(\epsilon) = \{(r,v_1,v_2); v_1 < r -\epsilon< v_2 <r \}$
\item $E_4(\epsilon) = \{(r,v_1,v_2); r -\epsilon< v_1 < v_2 < r \}$
\item $E_5(\epsilon) = \{(r,v_1,v_2); r -\epsilon< v_1 < r <v_2 \}$
\item $E_6(\epsilon) = \{(r,v_1,v_2); v_1 < r -\epsilon< r < v_2 \}$.
\end{itemize}
Here, for each positive small $\epsilon$, $\{E_i(\epsilon); 1\le i\le 6\}$ constitutes a partition of $[0,T]\times \{(v_1,v_2) \in [0,T]^2\setminus D; v_1 < v_2\}$. By using Jensen, H\"older's inequalities and Assumption A, C(ii) and (\ref{mod1}), there exists $q>1$ such that

$$
\mathbb{E}\int_{E_1(\epsilon)}|A_\epsilon(r,v_1,v_2)|^2 |\mu|(dv_1dv_2) |R(dr,T)|$$
$$
 \lesssim \epsilon^{\alpha+2}\int_0^T \int_{v_1} \sup_{s\ge v_2} \|\mathbf{D}_{v_1}Y^{'}_{s} - \mathbf{D}_{v_2} Y^{'}_{s}\|^2_{L^{2q}(\mathbb{P})} |\mu|(dv_1dv_2) \rightarrow 0,
$$
as $\epsilon \downarrow 0$. Similarly, there exists $q>1$ such that

$$
\mathbb{E}\int_{E_2(\epsilon)}|A_\epsilon(r,v_1,v_2)|^2 |\mu|(dv_1dv_2) |R(dr,T)|$$
$$
 \lesssim \epsilon^{\alpha+2}\int_0^T \int_{v_1} \sup_{s < v_1} \|\mathbf{D}_{v_1}Y^{'}_{s} - \mathbf{D}_{v_2} Y^{'}_{s}\|^2_{L^{2q}(\mathbb{P})} |\mu|(dv_1dv_2) \rightarrow 0,
$$
as $\epsilon \downarrow 0$. Similar analysis can be made for $E_i(\epsilon)$ for $3\le i\le 6$. Indeed, one can show that for each $i=3,4,5,6$,

$$\{|A_\epsilon|^2\mathds{1}_{E_i(\epsilon)} |\partial^2R|; 0 < \epsilon < 1\}$$
is uniformly integrable w.r.t $\mathbb{P}\times |R(\cdot,T)|\times Leb$. Vitali convergence theorem combined with (\ref{tens1}) yield $\mathbb{E}[L_2(\epsilon)]\rightarrow 0$ as $\epsilon \downarrow 0$.

\

\noindent \textbf{Analysis of $L_3(\epsilon)$ and $L_4(\epsilon)$.} In order to shorten notation, we now set

$$\Xi^{i\ell,j,\epsilon}_{r,v,t}:=\mathbf{D}_v \Big[\frac{1}{\epsilon}\big( u^{i\ell,j}_{r-\epsilon,r}\mathds{1}_{[0,t]}(r)e_\ell - u^{i\ell,\ell}_{r-\epsilon,r}\mathds{1}_{[0,t]}(r)e_j   \big)\Big],$$

\begin{eqnarray}
\nonumber\Delta_\mathbf{r} \Xi^{i\ell,j,\epsilon}(\mathbf{r},v,t)&:=&\Xi^{i\ell,j,\epsilon}_{r_1,v,t} - \Xi^{i\ell,j,\epsilon}_{r_2,v,t} = \frac{1}{\epsilon}\mathbf{D}_v \Big[\big( u^{i\ell,j}_{r_1-\epsilon,r_1}\mathds{1}_{[0,t]}(r_1)  - u^{i\ell,j}_{r_2-\epsilon,r_2}\mathds{1}_{[0,t]}(r_2)\big)   e_\ell \\
\label{dDeltaR_1}& &\\
\nonumber&-& \big(u^{i\ell,\ell}_{r_1-\epsilon,r_1}\mathds{1}_{[0,t]}(r_1) - u^{i\ell,\ell}_{r_2-\epsilon,r_2}\mathds{1}_{[0,t]}(r_2)\big) e_j   \big)\Big],
\end{eqnarray}

\begin{equation}\label{dDeltaR_2}
\Delta_\mathbf{v} \Delta_\mathbf{r} \Xi^{i\ell,j,\epsilon}(\mathbf{r},\mathbf{v},t):= \Delta_\mathbf{r} \Xi^{i\ell,j,\epsilon}(\mathbf{r},v_1,t) - \Delta_\mathbf{r} \Xi^{i\ell,j,\epsilon}(\mathbf{r},v_2,t),
\end{equation}
for $\mathbf{v} = (v_1,v_2),\mathbf{r} = (r_1,r_2) \in \mathbb{R}^2_+\setminus D$.

Of course, we recall that the above multi-parameter processes take values on the space of $d\times d$-matrices. It remains to estimate

$$ \mathbb{E}\int_{[0,T]^2\setminus D}\Big\| \mathbf{D}_\cdot \Big[\frac{1}{\epsilon}\big( u^{i\ell,j}_{r_1-\epsilon,r_1}\mathds{1}_{[0,t]}(r_1)e_\ell - u^{i\ell,\ell}_{r_1-\epsilon,r_1}\mathds{1}_{[0,t]}(r_1)e_j   \big)\Big]  $$
$$ - \mathbf{D}_\cdot \Big[\frac{1}{\epsilon}\big( u^{i\ell,j}_{r_2-\epsilon,r_2}\mathds{1}_{[0,t]}(r_2)e_\ell - u^{i\ell,\ell}_{r_2-\epsilon,r_2}\mathds{1}_{[0,t]}(r_2)e_j   \big)\Big] \Big\|^2_{L_{R}(\mathbb{R}^{d\times d})}|\mu|(dr_1dr_2)
$$
$$ = \mathbb{E}\int_{[0,T]^2\setminus D}\int_0^T |\Delta_\mathbf{r} \Xi^{i\ell,j,\epsilon}(\mathbf{r},v,t)   |^2 |R(dv,T)| |\mu|(dr_1dr_2)  $$
$$+\frac{1}{2}\mathbb{E}\int_{[0,T]^2\setminus D}\int_{[0,T]^2\setminus D} |  \Delta_\mathbf{v} \Delta_{\mathbf{r}} \Xi^{i\ell,j,\epsilon}(\mathbf{r},\mathbf{v},t) |^2 |\mu|(dv_1dv_2) |\mu|(dr_1dr_2) = L_3(\epsilon) + L_4(\epsilon).$$

\

\noindent \textbf{Analysis of $L_3(\epsilon)$}. Since $\ell\neq j$, by symmetry, Lemma \ref{rm} and the definition of (\ref{dDeltaR_1}), we only need to check convergence to zero in $L^2(\mathbb{P}\times |R(dv, T)|\times d|\mu|)$ of the $\ell$-th column (the only non-null column) of $\frac{1}{\epsilon}\mathbf{D}_v \Big[\big( u^{i\ell,j}_{r_1-\epsilon,r_1}  - u^{i\ell,j}_{r_2-\epsilon,r_2}\big)   e_\ell\Big]$.

\begin{lemma}\label{punctual1}
Assume that $Y^{'}$ satisfies the assumptions in Theorem \ref{gaussiancase}. Then, for each $\ell\neq j$ and $t\in(0,T]$,
$$\lim_{\epsilon\rightarrow 0^+}
\frac{1}{\epsilon}\mathbf{D}_v \Big[\big( u^{i\ell,j}_{r_1-\epsilon,r_1}\mathds{1}_{[0,t]}(r_1)  - u^{i\ell,j}_{r_2-\epsilon,r_2}\mathds{1}_{[0,t]}(r_2)\big)   e_\ell\Big]=0~a.s
$$
for almost all $(v,r_1,r_2)\in [0,T]\times [0,T]^2\setminus D$ w.r.t the product measure $|R(dv,T)|\times d|\mu|$.
\end{lemma}
\begin{proof}
The (only) non-null $\ell$-th column of
$$\epsilon^{-1}\mathbf{D}_v u^{i\ell,j}_{r_1-\epsilon, r_1}e_\ell\mathds{1}_{[0,t]}(r_1) - \epsilon^{-1}\mathbf{D}_v u^{i\ell,j}_{r_2-\epsilon, r_2}e_\ell\mathds{1}_{[0,t]}(r_2)$$
equals to
\begin{equation}\label{pj1}
\frac{1}{\epsilon}\int_{r_1-\epsilon}^{r_1} \Big\{ X^j_{s,r_1}\mathbf{D}_vY^{',i\ell}_{s} + Y^{',i\ell}_{s}\mathds{1}_{[s,r_1]}(v)e_j\Big\}
ds\mathds{1}_{[0,t]}(r_1)
\end{equation}
$$
-\frac{1}{\epsilon}\int_{r_2-\epsilon}^{r_2} \Big\{ X^j_{s,r_2} \mathbf{D}_vY^{',i\ell}_{s} + Y^{',i\ell}_{s}\mathds{1}_{[s,r_2]}(v)e_j\Big\}
ds\mathds{1}_{[0,t]}(r_2)
$$
a.s for Lebesgue almost all $v,r_1, r_2\in [0,T]$ and $\epsilon >0$. Then, the argument is the same as the one applied in the proof of Lemma \ref{pointwiseconv1}.
\end{proof}

We need to investigate convergence to zero of (\ref{pj1}) in $L^2(\mathbb{P}\times |R(dv, T)|\times d|\mu|)$. Again, the idea is to explore almost sure convergence stated in Lemma \ref{punctual1} and uniform integrability. By symmetry, we may restrict $r_2 < r_1\le t$. The case $r_2 \le t < r_1\le T$ is trivial because no singularity appears in $\partial^2 R(r_1,r_2)$. We split $[0,T]\times \{(r_1,r_2); r_2 < r_1 \le t\}$ into three cases

$$F_1 = \{(v,r_1,r_2); 0\le v < r_2 < r_1\le t\},\quad F_2 = \{(v,r_1,r_2); 0\le r_2 < v < r_1\le t\}$$
$$F_3 = \{(v,r_1,r_2); 0\le r_2 < r_1 \le v\le T\}.$$
We will check that

$$\Bigg|\frac{1}{\epsilon}\mathbf{D}_v \Big[\big( u^{i\ell,j}_{r_1-\epsilon,r_1}\mathds{1}_{[0,t]}(r_1)  - u^{i\ell,j}_{r_2-\epsilon,r_2}\mathds{1}_{[0,t]}(r_2)\big)   e_\ell\Big]\Bigg|^2|\partial^2 R(r_1,r_2)|\mathds{1}_{F_z}$$
is uniformly integrable (along the parameter $\epsilon \in (0,1)$) over the measure space $\mathbb{P}\times |R(dv, T)|\times \text{Leb}$, for each $z=1,2,3$.

The process (\ref{pj1}) at the region $F_2$ can be easily estimated by using (\ref{intassump}) and (\ref{dagrowth}), assumption A and the fact that none singularity appears in $\partial^2 R$. Indeed, there exists $p>1$ such that

$$
\mathbb{E}\int_{F_2}\Bigg|\frac{1}{\epsilon}\int_{r_m-\epsilon}^{r_m} \Big\{ X^j_{s,r_m}\mathbf{D}_vY^{',i\ell}_{s} + Y^{',i\ell}_{s}\mathds{1}_{[s,r_m]}(v)e_j\Big\}
ds\Bigg|^{2p} |\partial^2R(r_1,r_2)|^p|\partial_v R(v,T)|d\mathbf{r}dv
$$
$$\lesssim \int_{r_2 < v < r_1\le t}\{(r_1-r_2)^{\alpha p} + \phi(r_1,r_2)^p\}|\partial_v R(v,T)|d\mathbf{r}dv < \infty,$$
for every $\epsilon  \in (0,1)$ and $m=1,2$.
At the region $F_3$ (we may suppose $r_1 < v$), (\ref{pj1}) reduces to

\begin{equation}\label{rclose}
\frac{1}{\epsilon}\int_{r_1-\epsilon}^{r_1}X^j_{s,r_1}\mathbf{D}_vY^{',i\ell}_sds - \frac{1}{\epsilon}\int_{r_2-\epsilon}^{r_2}X^j_{s,r_2}\mathbf{D}_vY^{',i\ell}_sds.
\end{equation}
We split $\{(v,r_1,r_2); 0\le r_2 < r_1 < v\} = \{(v,r_1,r_2); 0\le r_2 < r_1 < v, r_1-r_2 < \epsilon\}\cup \{(v,r_1,r_2); 0\le r_2 < r_1 < v, r_1-r_2\ge \epsilon\} =:K_1\cup K_2$. On $K_1$, we can write (\ref{rclose}) as

$$\frac{1}{\epsilon}\int_{r_1-\epsilon}^{r_2}X^j_{r_2,r_1}\mathbf{D}_vY^{',i\ell}_sds + \frac{1}{\epsilon}\int_{r_2}^{r_1}X^j_{s,r_1}\mathbf{D}_vY^{',i\ell}_sds - \frac{1}{\epsilon}\int_{r_2-\epsilon}^{r_1-\epsilon}X^j_{s,r_2}\mathbf{D}_vY^{',i\ell}_sds$$
and hence Assumption C yield

$$\mathbb{E}\int_{K_1}\Bigg| \frac{1}{\epsilon}\int_{r_2-\epsilon}^{r_1-\epsilon}X^j_{s,r_2}\mathbf{D}_vY^{',i\ell}_sds \Bigg|^2|\mu|(dr_1dr_2)|\partial_v R(v,T)|dv$$
$$+ \mathbb{E}\int_{K_1}\Bigg| \frac{1}{\epsilon}\int_{r_2} ^{r_1}X^j_{s,r_1}\mathbf{D}_vY^{',i\ell}_sds \Bigg|^2|\mu|(dr_1dr_2)|\partial_v R(v,T)|dv$$

$$+\mathbb{E}\int_{K_1}\Bigg| \frac{1}{\epsilon}\int_{r_1-\epsilon} ^{r_2}X^j_{r_2,r_1}\mathbf{D}_vY^{',i\ell}_sds \Bigg|^2|\mu|(dr_1dr_2)|\partial_v R(v,T)|dv$$
$$\lesssim \int_0^T \int_{r_1-\epsilon}^{r_1}(r_1-r_2)^{2\alpha+2}dr_2dr_1\lesssim \epsilon^{2\alpha+3}\rightarrow 0,$$
as $\epsilon\downarrow 0$, because $2\alpha+3>0$. On $K_2$, we estimate (\ref{rclose}) as follows: We take $1 < p < \frac{1}{-2\alpha-2}$ and again by Assumption C, we have

$$\mathbb{E}\int_{K_2}\Bigg|\frac{1}{\epsilon}\int_{r_1-\epsilon}^{r_1}X^j_{s,r_1}\mathbf{D}_vY^{',i\ell}_sds\Bigg|^{2p}|\partial^2 R(r_1,r_2)|^p dr_1dr_2|\partial_v R(v,T)|dv$$
$$\lesssim \int_0^T \int_{r_2}^T (r_1-r_2)^{p(2\alpha+2)}dr_1dr_2 < \infty,$$
for every $\epsilon \in (0,1)$.

For the analysis on $F_1$, we write $F_1 = \cup_{i=1}^7 F_{1,i}$, where

$$F_{1,1} = \{v < r_2-\epsilon <  r_1-\epsilon < r_2 < r_1\},~ F_{1,2} = \{v < r_2-\epsilon < r_2 <  r_1-\epsilon < r_1\},$$
$$F_{1,3} = \{ r_2-\epsilon < v <r_1-\epsilon < r_2 < r_1\},~ F_{1,4} = \{ r_2-\epsilon < v < r_2 < r_1-\epsilon < r_1\},$$
$$F_{1,5} = \{ r_2-\epsilon < r_1-\epsilon < v <  r_2 < r_1\},~F_{1,6}=\{ r_2-\epsilon < v< r_2 < r_1-\epsilon < r_1\},$$
$$F_{1,7}=\{ r_2-\epsilon < v< r_1-\epsilon < r_2 < r_1\}.$$
We observe (\ref{intassump}), Assumption C, Jensen and H\"older'inequality allow us to choose $1 < q < \frac{\alpha+3}{-\alpha}$ such that

$$
\mathbb{E}\int_{F_{1,z}}\Bigg|\frac{1}{\epsilon}\int_{r_m-\epsilon}^{r_m} \Big\{ X^j_{s,r_m}\mathbf{D}_vY^{',i\ell}_{s} + Y^{',i\ell}_{s}\mathds{1}_{[s,r_m]}(v)e_j\Big\}
ds\Bigg|^{2q} |\partial^2R(r_1,r_2)|^q|\partial_v R(v,T)|d\mathbf{r}dv
$$
$$\lesssim \int_{r_2 < r_1}(r_1-r_2)^{\alpha+2+ q\alpha} d\mathbf{r} < \infty$$
for every $\epsilon \in (0,1)$, $m=1,2$ and $z=3,4,6,7$. Next, we analyze the set $F_{1,5}$. In this case, we may write (\ref{pj1}) equals to

$$
\frac{1}{\epsilon}\int_{r_2-\epsilon}^{r_1-\epsilon} Y^{',i\ell}_{s}ds e_j + \frac{1}{\epsilon}\int_v^{r_1} X^j_{s,r_1}\mathbf{D}_vY^{',i\ell}_{s}ds
- \frac{1}{\epsilon} \int_{v}^{r_2} X^j_{s,r_2} \mathbf{D}_vY^{',i\ell}_{s}
ds$$

$$+ \frac{1}{\epsilon}\int_{r_1-\epsilon}^v X^j_{r_2,r_1}\mathbf{D}_vY^{',i\ell}_{s}ds - \frac{1}{\epsilon}\int_{r_2-\epsilon}^{r_1-\epsilon} X^j_{s,r_2}\mathbf{D}_vY^{',i\ell}_{s}ds  $$
on $F_{1,5}$. At this point, we use Assumption C, (\ref{intassump}) and Fubini's theorem to get

$$\mathbb{E}\int_{F_{1,5}}\Bigg|\frac{1}{\epsilon}\int_{r_2-\epsilon}^{r_1-\epsilon} Y^{',i\ell}_{s}ds e_j\Bigg|^2|\partial_v R(v,T)|dv|\partial R(r_1,r_2)|dr_1dr_2$$
$$\lesssim \epsilon^{-2}\int_0^T \int_{r_1-\epsilon}^{r_1} \big(r_2-(r_1-\epsilon)\big)^{\alpha+2}(r_1-r_2)^{\alpha+2} dr_2dr_1\lesssim \epsilon^{2\alpha+3}\rightarrow 0$$
as $\epsilon \downarrow 0$. We can write

\begin{eqnarray*}
\frac{1}{\epsilon}\int_v^{r_1} X^j_{s,r_1}\mathbf{D}_vY^{',i\ell}_{s}ds
- \frac{1}{\epsilon} \int_{v}^{r_2} X^j_{s,r_2} \mathbf{D}_vY^{',i\ell}_{s}
ds &=& \frac{1}{\epsilon}\int_v^{r_2}X^j_{r_2,r_1}\mathbf{D}_vY^{',i\ell}_{s}ds\\
 &+& \frac{1}{\epsilon}\int_{r_2}^{r_1}X^j_{s,r_1}\mathbf{D}_vY^{',i\ell}_{s}ds
\end{eqnarray*}
on $F_{1,5}$. Repeat the same argument used above to conclude

\begin{small}
$$\lim_{\epsilon \rightarrow 0^+}\mathbb{E}\int_{F_{1,5}}\Bigg|\frac{1}{\epsilon}\int_{r_2}^{r_1}X^j_{s,r_1}\mathbf{D}_vY^{',i\ell}_{s}ds\Bigg|^2|\partial_v R(v,T)|dv |\mu|(dr_1dr_2)=0,$$
$$\lim_{\epsilon \rightarrow 0^+}\mathbb{E}\int_{F_{1,5}}\Bigg|\frac{1}{\epsilon}\int_{v}^{r_2}X^j_{r_2,r_1}\mathbf{D}_vY^{',i\ell}_{s}ds\Bigg|^2|\partial_v R(v,T)|dv |\mu|(dr_1dr_2)=0,$$
\end{small}
\begin{small}
$$\lim_{\epsilon \rightarrow 0^+}\mathbb{E}\int_{F_{1,5}}\Bigg|\frac{1}{\epsilon}\int_{r_1-\epsilon}^v X^j_{r_2,r_1}\mathbf{D}_vY^{',i\ell}_{s}ds\Bigg|^2|\partial_v R(v,T)|dv |\mu|(dr_1dr_2)   =0,$$
$$\lim_{\epsilon \rightarrow 0^+}\mathbb{E}\int_{F_{1,5}}\Bigg|\frac{1}{\epsilon}\int_{r_2-\epsilon}^{r_1-\epsilon} X^j_{s,r_2}\mathbf{D}_vY^{',i\ell}_{s}ds\Bigg|^2 |\partial_v R(v,T)|dv |\mu|(dr_1dr_2)  =0.$$
\end{small}

By using Jensen's inequality, Assumptions A, C and (\ref{intassump}), we can repeat the same argument given in the analysis of (\ref{rclose}) to conclude

$$\lim_{\epsilon\rightarrow 0^+}\mathbb{E}\int_{F_{1,1}}\Big|\Delta_\mathbf{r} \Xi^{i\ell,j,\epsilon}(\mathbf{r},v,t) \Big|^2 |\partial_v R(v,T)|dv |\mu|(dr_1dr_2)=0$$
and there exists $p>2$ such that

$$\sup_{0 < \epsilon < 1}\mathbb{E}\int_{F_{1,2}}\Big|\Delta_\mathbf{r} \Xi^{i\ell,j,\epsilon}(\mathbf{r},v,t) \Big|^p |\partial_v R(v,T)|dv |\partial^2R(r_1,r_2)|^{\frac{p}{2}}d\mathbf{r}< \infty.$$
Vitali convergence theorem combined with Lemma \ref{punctual1} allow us to conclude $\mathbb{E}[L_3(t)]\rightarrow 0$ as $\epsilon\rightarrow 0^+$.

\

\noindent \textbf{Analysis of $L_4(\epsilon)$.} In the sequel, in view of assumption (\ref{dagrowth}), we may suppose that $\phi=0$, i.e.,

$$\big| \partial^2 R(r_1,r_2)\big|\lesssim |r_1-r_2|^\alpha; (r_1,r_2)\in [0,T]^2\setminus D.$$
The main difficulty lies on the singularity of the kernel $|r_1-r_2|^\alpha$ on $[0,T]^2\setminus D$. Indeed, by Assumption C, we recall there exists $L>1$ such that $\phi$ is $p$-integrable on $[0,T]^2\setminus D$ for every $p \in (1,L)$. Then, we may restrict the analysis to the case $\phi=0$.

\

Since $\ell\neq j$, by symmetry, Lemma \ref{rm} and the definition of (\ref{dDeltaR_2}), we only need to check convergence to zero in $L^2(\mathbb{P}\times |\mu|\times |\mu|)$ of the $\ell$-th column (the only non-null column) of

\begin{small}
\begin{equation}\label{al25}
\frac{1}{\epsilon}\Bigg\{ \mathbf{D}_{v_1} \Big[\big( u^{i\ell,j}_{r_1-\epsilon,r_1}\mathds{1}_{[0,t]}(r_1)  - u^{i\ell,j}_{r_2-\epsilon,r_2}\mathds{1}_{[0,t]}(r_2)\big)   e_\ell\Big] - \mathbf{D}_{v_2} \Big[\big( u^{i\ell,j}_{r_1-\epsilon,r_1}\mathds{1}_{[0,t]}(r_1)  - u^{i\ell,j}_{r_2-\epsilon,r_2}\mathds{1}_{[0,t]}(r_2)\big)   e_\ell\Big]\Bigg\}.
\end{equation}
\end{small}
Without any loss of generality, we may assume $0\le r_2 < r_1\le t$, $v_2 < v_1\le T$. We also observe the case $r_2 < t < r_1$ can be easily treated because, in this case, no singularity appears in $|r_1-r_2|^{\alpha}$. We can write (\ref{al25}) as

$$ \frac{1}{\epsilon}\int_{r_1-\epsilon}^{r_1}  X^j_{s,r_1}\Big(\mathbf{D}_{v_1}Y^{',i\ell}_s- \mathbf{D}_{v_2}Y^{',i\ell}_s\Big)ds$$
$$ - \frac{1}{\epsilon}\int_{r_2-\epsilon}^{r_2}  X^j_{s,r_2} \Big(\mathbf{D}_{v_1}Y^{',i\ell}_s- \mathbf{D}_{v_2}Y^{',i\ell}_s \Big) ds$$
$$+ \frac{1}{\epsilon}\int_{r_1-\epsilon}^{r_1} Y^{',i\ell}_s\big(\mathds{1}_{[s,r_1]}(v_1) - \mathds{1}_{[s,r_1]}(v_2)\big) e_j ds$$
$$ - \frac{1}{\epsilon}\int_{r_2-\epsilon}^{r_2} Y^{',i\ell}_s\big(\mathds{1}_{[s,r_2]}(v_1) - \mathds{1}_{[s,r_2]}(v_2)\big) e_j ds$$

$$=:a_1(\mathbf{r},\mathbf{v},\epsilon) - a_2(\mathbf{r},\mathbf{v},\epsilon)  + b_1(\mathbf{r},\mathbf{v},\epsilon) - b_2(\mathbf{r},\mathbf{v},\epsilon).$$
To shorten notation, we denote $a(\mathbf{r},\mathbf{v},\epsilon) = a_1(\mathbf{r},\mathbf{v},\epsilon) - a_2(\mathbf{r},\mathbf{v},\epsilon),  b(\mathbf{r},\mathbf{v},\epsilon)=b_1(\mathbf{r},\mathbf{v},\epsilon) - b_2(\mathbf{r},\mathbf{v},\epsilon)$.

\begin{lemma}\label{punctual2}
We have $\lim_{\epsilon\downarrow 0}a_i(\mathbf{r},\mathbf{v},\epsilon) = \lim_{\epsilon\downarrow 0}b_i(\mathbf{r},\mathbf{v},\epsilon) =\mathbf{0}$ a.s for Lebesgue almost all $(\mathbf{r},\mathbf{v})\in [0,T]^2\setminus D\times [0,T]^2\setminus D$, for each $i=1,2$.
\end{lemma}
\begin{proof}
The same argument given in Lemmas \ref{pointwiseconv1} and \ref{punctual1} applies here.
\end{proof}

In the sequel, we will check that

$$
|b(\mathbf{r},\mathbf{v},\epsilon)|^2|\partial^2R(\mathbf{r})|\partial^2R(\mathbf{v})|
$$
is an uniformly integrable family (in $0 < \epsilon < 1$) w.r.t the measure $\mathbb{P}\times \text{Leb}$ and hence Vitali convergence theorem combined with Lemma \ref{punctual2} will imply

$$\lim_{\epsilon \rightarrow 0^+}\mathbb{E}\int_{v_1>v_2, t \ge r_1 > r_2}|b(\mathbf{r},\mathbf{v},\epsilon)|^2|\partial^2R(\mathbf{r})||\partial^2R(\mathbf{v})|d\mathbf{r}d\mathbf{v}=0.$$
We observe $b = \mathbf{0}$ on $\{r_2 < r_1 < v_2 < v_1\}$ so that we only need to analyze $b$ on $r_1 > v_2$. We split $\{(\mathbf{r},\mathbf{v}); 0 \le r_2 < r_1\le t, 0\le v_2 < v_1\le T, r_1 > v_2\}$ in terms of the partition

$$G_1 = \{v_2 < v_1 < r_2 < r_1 \},~G_2 = \{r_2 < v_2 < v_1 < r_1\}$$
$$G_3 =\{v_2 < r_2 < v_1 < r_1\},~G_4 = \{v_2 < r_2 < r_1 < v_1\},~G_5 =\{r_2 < v_2 < r_1 < v_1\}$$
The most delicate cases are $G_1$ and $G_2$. We split $G_1$ in terms of the partition

$$G_{11} = \{r_2-\epsilon < r_1-\epsilon  < v_2 < v_1 < r_2 < r_1\}, G_{12} = \{v_2 < v_1 < r_2-\epsilon < r_1-\epsilon  < r_2 < r_1\}$$
$$G_{13} = \{v_2 < v_1 < r_2-\epsilon < r_2 <  r_1-\epsilon  < r_1\}, G_{14} = \{r_2-\epsilon < v_2 < r_1-\epsilon  < v_1 <  r_2 < r_1\}$$
$$G_{15} = \{v_2 < r_2-\epsilon < v_1 < r_1-\epsilon  < r_2 < r_1\}, G_{16} = \{v_2 < r_2-\epsilon < v_1 < r_2 <  r_1-\epsilon  < r_1\}$$
$$G_{17} = \{r_2-\epsilon < v_2 < v_1 < r_1-\epsilon < r_2 < r_1\},~G_{18} = \{r_2-\epsilon < v_2 < v_1 < r_2 < r_1-\epsilon < r_1\}$$
$$G_{19} = \{v_2 < r_2-\epsilon < r-1-\epsilon < v_1 < r_2 < r_1\}.$$
We observe $b = \mathbf{0}$ on $\cup_{\ell=1}^3G_{1\ell}$. Jensen's inequality and assumption (\ref{intassump}) yield

$$\mathbb{E}\int_{\cup_{\ell=4}^6G_{1\ell}} |b(\mathbf{r},\mathbf{v},\epsilon)|^p |\partial^2 R (\mathbf{r})|^{\frac{p}{2}} |\partial^2 R (\mathbf{v})|^{\frac{p}{2}}d\mathbf{r}d\mathbf{v}\lesssim \int_{\cup_{\ell=4}^6G_{1\ell}} |\partial^2 R (\mathbf{r})|^{\frac{p}{2}} |\partial^2 R (\mathbf{v})|^{\frac{p}{2}}d\mathbf{r}d\mathbf{v}.$$
Next, by using Assumption C and choosing $2 < p < \frac{18}{4}$, we have

$$\int_{G_{14}} |\partial^2 R (\mathbf{r})|^{\frac{p}{2}} |\partial^2 R (\mathbf{v})|^{\frac{p}{2}}d\mathbf{r}d\mathbf{v}\lesssim \int_{r_2 < r_1} \int_{r_1-\epsilon}^{r_2}\int_{r_2-\epsilon}^{r_1-\epsilon}(r_1-r_2)^{\frac{\alpha p}{2}}(v_1-v_2)^{\frac{\alpha p}{2}}dv_2dv_1d\mathbf{r}$$
$$\lesssim \int_{r_2<r_1} (r_1-r_2)^{\frac{\alpha p}{2} + 2}d\mathbf{r} < \infty$$
for every $\epsilon \in (0,1)$. Similar analysis can be made on $G_{15}$ and $G_{16}$.

We can choose $0 < \beta < 1$ such that $0 < - (\alpha+1) < \frac{1}{3} < \beta < \frac{2}{3} < \alpha+2 < 1$. Then, $\epsilon^{-2}(r_1-r_2)^2\le \epsilon^{-\beta}(r_1-r_2)^\beta$ on $G_{19}$. Then,

$$\mathbb{E}\int_{G_{19}} |b(\mathbf{r},\mathbf{v},\epsilon)|^2 |\partial^2 R (\mathbf{r})| |\partial^2 R (\mathbf{v})|d\mathbf{r}d\mathbf{v}\lesssim \frac{1}{\epsilon^2}\int_{G_{19}}(r_1-r_2)^2(r_1-r_2)^\alpha (v_1-v_2)^\alpha d\mathbf{r}d\mathbf{v}.$$
$$\lesssim \epsilon^{-\beta} \int_{r_2 < r_1}\int_0^{r_2-\epsilon}\int_{r_1-\epsilon}^{r_2}(v_1-v_2)^{\alpha}(r_1-r_2)^{\alpha+\beta}d\mathbf{r}$$
$$\lesssim \epsilon^{\alpha + \beta-2}\int_{r_2 < r_1}(r_1-r_2)^{\alpha+\beta}d\mathbf{r} \rightarrow 0$$
as $\epsilon \downarrow 0$. The analysis of the sets $G_{17}$ and $G_{18}$ is easy, so we omit the details. Next, we split the set $G_2$ into

$$G_{21} = \{r_2-\epsilon < r_2 < v_2 < v_1 < r_1-\epsilon < r_1\},~G_{22} = \{r_2-\epsilon < r_2 < v_2 < r_1-\epsilon < v_1 < r_1\}$$
$$G_{23} = \{r_2-\epsilon < r_2 < r_1-\epsilon < v_2 < v_1 < r_1\},~G_{24} = \{r_2-\epsilon < r_1-\epsilon < r_2 < v_2 < v_1 < r_1\}$$
We observe $b = \mathbf{0}$ on $G_{21}$ and, for each $i=2,3,4$, one can easily check we can take $2 < p < \frac{-3}{\alpha}$ such that

$$\mathbb{E}\int_{G_{2i}} |b(\mathbf{r},\mathbf{v},\epsilon)|^p |\partial^2 R (\mathbf{r})|^{\frac{p}{2}} |\partial^2 R (\mathbf{v})|^{\frac{p}{2}}d\mathbf{r}d\mathbf{v}\lesssim \int_{r_2 < r_1}(r_1-r_2)^{\alpha p+2} d\mathbf{r}< \infty,$$
for every $\epsilon \in (0,1)$. The analysis over $G_4$ is similar to $G_2$. The analysis of $G_3$ and $G_5$ is straightforward. By symmetry, we conclude
$$\lim_{\epsilon \rightarrow 0^+}\mathbb{E}\int_{[0,T]^2\setminus D}|b(\mathbf{r},\mathbf{v},\epsilon)|^2  |\partial^2R(\mathbf{r})|\partial^2R(\mathbf{v})|d\mathbf{r}d\mathbf{v}=0.
$$
The analysis of the term $a(\mathbf{r},\mathbf{v},\epsilon)$ is similar to $b$, so we may omit the details. Indeed, we need to combine assumptions C and (\ref{mod1}) to check uniform integrability of
$$|a(\mathbf{r},\mathbf{v},\epsilon)|^2|\partial^2 R(\mathbf{r})| |\partial^2 R(\mathbf{v})|$$
just like we did for the term $b$. For the subset $\{(\mathbf{r},\mathbf{v}); 0\le r_2 < r_1\le t, 0\le v_2 < v_2 \le T, r_1 > v_2\}$, we make the analysis over the same partition $\cup_{z=1}^5 G_z$. For the subset $\{(\mathbf{r},\mathbf{v}); 0\le r_2 < r_1\le t, 0\le r_2 < r_1 < v_2 < v_1 \le T\}$, we decompose just like $G_1$ and use assumptions C and (\ref{mod1}). By using symmetry and Vitali convergence theorem, we conclude

$$\lim_{\epsilon \rightarrow 0^+}\mathbb{E}\int_{[0,T]^2\setminus D}|a(\mathbf{r},\mathbf{v},\epsilon)|^2  |\partial^2R(\mathbf{r})|\partial^2R(\mathbf{v})|d\mathbf{r}d\mathbf{v}=0.
$$
This concludes the proof that $J_2(\epsilon,t)\rightarrow 0$, as $\epsilon\downarrow 0$.

\section{Appendix}\label{appendixsec}
\subsection{Proof of Lemma \ref{XL}}
In this sufficient to check the one-dimensional case $d=1$. Let $(v_n)_{n\ge 0}$ be an orthonormal basis of $L_R$. By Corollary 6.49 in \cite{krukrusso}, we know that
$X \in L_R$ a.s and hence we can write
\[
X = \sum_{i=0}^{\infty} F_i v_i \textrm{ in } L_R \ a.s.,
\]
where
\[
F_i= \left<X,v_i\right>_{L_R}.
\]
By Prop. 9.6 in \cite{krukrusso}, there exists $\phi_i \in L_R$ such that $F_i = \int_0^\infty \phi_idX$ so that $F_i\in \mathbb{D}^{1,2}$ for each $i\ge 0$. This shows that $X^n := \sum_{i=0}^{n} F_i v_i \in \mathbb{D}^{1,2}(L_R)$ for each $n\ge 0$. We recall that
\begin{equation}\label{fr1}
\mathbb{E}\|X \|^2_{L_R}=  \int_0^\infty \text{Var}(X_s) R(ds,\infty)-\frac{1}{2} \int_{\mathbb{R}^2_+\setminus D} \text{Var}(X_{s_1}-X_{s_2})\partial^2R(s_1,s_2)ds_1ds_2 < \infty,
\end{equation}
which is finite due to Assumption C. The estimate (\ref{fr1}) implies $\sum_{n\ge 0}\mathbb{E}F^2_n < \infty$ so that

\begin{equation}\label{fr2}
\lim_{n\rightarrow +\infty}\mathbb{E}\|X^n-X\|^2_{L_R}=0.
\end{equation}
It remains to show that the sequence $\mathbf{D}(\sum_{i=0}^{n}F_i
v_i))_{n \geq 0}$ is Cauchy in $L^2\big(\Omega; L_{2,R}\big)$.

It is enough to show

\begin{equation}\label{E455}
\mathbb{E} \left\| \sum_{i=n}^{\infty} \mathbf{D}F_i \otimes v_i
\right\|^2_{L_{2,R}} \rightarrow 0,
\end{equation}
as $n\rightarrow \infty$. Indeed,

$$\mathbb{E} \left\| \sum_{i=n}^{\infty} \mathbf{D}F_i \otimes v_i
\right\|^2_{L_{2,R}}= \mathbb{E} \sum_{i=n}^{\infty} \big\|\mathbf{D}F_i\big\|^2_{L_R} = \sum_{i=n}^{\infty} \big\|\phi_i\big\|^2_{L_R} = \sum_{i=1}^\infty\mathbb{E}F^2_i\rightarrow 0,$$
as $n\rightarrow +\infty$. The estimates (\ref{fr1}), (\ref{fr2}) and (\ref{E455}) allow us to conclude the proof.

\subsection{Proof of Lemma \ref{skolemma}}
Fix $-\frac{3}{2} < \alpha < -1$. Recall that $ R(dt,T)$ is a finite non-negative measure whose support is $[0,T]$ and $|\mu|$ is a sigma-finite positive measure whose support is $[0,T]^2$. By definition, for a given $-1 < r < 1$, we have

\begin{eqnarray*}
\| Y_{\cdot+r}-Y_{\cdot}\|^2_{\mathbb{D}^{1,2}(L_R(\mathbb{R}^d))} &=& \int_0^\infty \| Y_{t+r}-Y_{t}\|^2_{\mathbb{D}^{1,2}(\mathbb{R}^d)}\partial_t R(t,T)dt\\
&+&\frac{1}{2}\int_{\mathbb{R}^2_+\setminus D}\|(Y_{t+r}-Y_{t}) - (Y_{s+r}-Y_{s})\|^2_{\mathbb{D}^{1,2}(\mathbb{R}^d)} |\partial^2 R(s,t)|dsdt\\
&\lesssim& r^{2\gamma} + \frac{1}{2}\int_{\mathbb{R}^2_+\setminus D}\|(Y_{t+r}-Y_{t}) - (Y_{s+r}-Y_{s})\|^2_{\mathbb{D}^{1,2}(\mathbb{R}^d)} |\partial^2 R(s,t)|dsdt,
\end{eqnarray*}
where $|\mu|(dv_1dv_2) = |\partial^2 R (v_1,v_2)|dv_1dv_2$ and $R(dt,T) = R(dt,\infty) = \partial_t R(t,T)dt$. At first, one can easily check Assumption S1 yields

\begin{equation}\label{incrP2}
\| Y_{t+r}-Y_{t} - Y_{s+r} + Y_s\|^2_{\mathbb{D}^{1,2}(\mathbb{R}^d)}\lesssim \min\Big\{|t-s|^{2\gamma}, |r|^{2\gamma}\Big\},
\end{equation}
for $0\le s < t\le T$ and $|r|\in (0,1)$. Having said that, the idea is to split the region

$$
\{(s,t) \in \mathbb{R}^2_+; 0\le s < t < \infty\} = \{(s,t); 0\le s < t < s+|r|\}\cup \{(s,t); 0\le s < s+|r| \le t\}.
$$
By symmetry and using again Assumption S1 and (\ref{incrP2}), we shall write
\begin{equation}\label{splitF}
\int_{\mathbb{R}^2_+\setminus D}\|(Y_{t+r}-Y_{t}) - (Y_{s+r}-Y_{s})\|^2_{\mathbb{D}^{1,2}(\mathbb{R}^d)} |\partial^2 R(s,t)|dsdt
\end{equation}
$$\lesssim \int_{0\le s < t < s+|r|} |t-s|^{2\gamma} |\partial^2 R(s,t)|dsdt $$
$$+ |r|^{2\gamma} \int_{0\le s < s+|r| \le t} |\partial^2 R(s,t)|dsdt.$$
By assumption C,
$$|\partial^2 R(s,t)|\lesssim |t-s|^\alpha + \phi(s,t);~(s,t) \in [0,T]^2\setminus D,$$
where $\phi$ is integrable over $[0,T]^2\setminus D$. For this reason, without any loss of generality, we may assume $\phi=0$. A direct computation yields

\begin{equation}\label{c4}
\int_{0\le s < t < s+|r|} |t-s|^{2\gamma}|\partial^2 R(s,t)|dsdt\lesssim |r|^{2\gamma + \alpha+1},
\end{equation}
for every $|r| \in (0,1)$. We also have,

\begin{equation}\label{c5}
\int_{0\le s < s+|r|\le t} |\partial^2 R(s,t)|dsdt \lesssim \int_0^{T-|r|}\int_{s+|r|}^T(t-s)^{\alpha}dtds\lesssim |r|^{\alpha+1},
\end{equation}
for every $|r| \in (0,1)$. Summing up, (\ref{splitF}), (\ref{c4}) and (\ref{c5}), we have

$$\| Y_{\cdot + r} - Y\|^2_{\mathbb{D}^{1,2}(L_R(\mathbb{R}^d))}\lesssim |r|^{2\gamma+ \alpha+1},$$
for every $|r| \in (0,1)$ and we conclude the proof.

\subsection{Proof of Lemma \ref{TRPlemma}}

The proof follows from routine arguments as summarized here. Let us fix $\frac{1}{4} < H < \frac{1}{2}$ and let $f:\mathbb{R}^d\rightarrow \mathbb{R}^d$ be a  a continuously differentiable function such that $f$ and $\nabla f$ are $\theta$-H\"{o}lder continuous functions with $\frac{1}{2H}-1< \theta\le 1$. Choose an orthonormal basis $\{v_n; n\ge 1\}$ of $L_R(\mathbb{R}^d)$ of continuous functions (see Prop 6.2 in \cite{krukrusso}). The conditions imposed on $(f,\nabla f)$ yields $f(X)\in L_R(\mathbb{R}^d)~a.s$ and we can define

$$F_n :=\sum_{\ell=1}^n \langle f(X), v_\ell \rangle_{L_R(\mathbb{R}^d)}v_\ell; n\ge 1,$$
in such way that $F_n \rightarrow f(X)$ in $L^2(\Omega,L_R(\mathbb{R}^d))$ as $n\rightarrow +\infty$. By Prop 8.12-8.14 in \cite{krukrusso}, the subexponential behavior of $\nabla f$ and the assumption $f \in C^1$ imply $f(X_s) \in \mathbb{D}^{1,2}(\mathbb{R}^d)$ and $\mathbf{D}f(X_s) = \nabla f (X_s) \mathds{1}_{[0,s]}$ for every $s \in [0,T]$. Moreover, by using Lemma 9.13 in \cite{krukrusso}, one can easily check $\langle f(X), v_n \rangle_{L_R(\mathbb{R}^d)} \in \mathbb{D}^{1,2}$ and hence $F_n \in \mathbb{D}^{1,2}\big(L_R(\mathbb{R}^d)\big)
$ for every $n\ge 1$. By using the $\theta$-H\"older regularity of $\nabla f$, we can check
$$
\sup_{n\ge 1}\mathbb{E}\|\mathbf{D}F_n\|^2_{L_{2,R}(\mathbb{R}^{d\times d})} < \infty.
$$
This shows that $f(X) \in \mathbb{D}^{1,2}(L_R(\mathbb{R}^d))$. Clearly,

$$\mathbb{E}|f(X_t) - f(X_s)|^2\lesssim \|f\|^2_\theta|t-s|^{2\theta H},$$
for every $0\le s,t < \infty$,
and
$$|\nabla f(X_a)|\le \|\nabla f\|_\theta |X_a|^\theta + |\nabla f(0)|;~a\ge 0.$$
Then,


$$
\sup_{s\ge 0}\mathbb{E}\big|\nabla f(X_s)\big|^2\lesssim \max\{\|\nabla f\|^2_\theta, |\nabla f(0)|^2\} < \infty.
$$
By definition,
\begin{eqnarray*}
\|f(X_t) - f(X_s)\|^2_{\mathbb{D}^{1,2}(\mathbb{R}^d)} &=& \mathbb{E}|f(X_t) - f(X_s)|^2\\
 &+& \mathbb{E}\| \nabla f(X_t)\mathds{1}_{[0,t]} - \nabla f(X_s)\mathds{1}_{[0,s]}\|^2_{L_R(\mathbb{R}^{d\times d})},
\end{eqnarray*}
and triangle inequality yields

\begin{eqnarray*}
\mathbb{E}\| \nabla f(X_t)\mathds{1}_{[0,t]} - \nabla f(X_s)\mathds{1}_{[0,s]}\|^2_{L_R(\mathbb{R}^{d\times d})}&\lesssim& \mathbb{E}|\nabla f(X_t) - \nabla f(X_s)|^2\|\mathds{1}_{[0,t]}\|^2_{L_R}\\
&+&  \mathbb{E}|\nabla f(X_s)|^2\|\mathds{1}_{[0,t]} -  \mathds{1}_{[0,s]}\|^2_{L_R}\\
&\lesssim& \|\nabla f\|^2_\theta |t-s|^{2\theta H} + |t-s|^{2H}.
\end{eqnarray*}

Therefore,

$$
\|f(X_t) - f(X_s)\|^2_{\mathbb{D}^{1,2}(\mathbb{R}^d)}\lesssim \big\{|t-s|^{2\theta H} + |t-s|^{2H}\big\},
$$
for every $0\le s,t < \infty$. Then, $f(X)$ satisfies the assumptions of Theorem \ref{mainTH}.
In particular, assumptions S1 and S2 hold with exponents $\theta H$ and $1$, respectively. Now, we set $\frac{1}{3} < H < \frac{1}{2}$, $\frac{1}{2H}-1 < \theta\le \frac{1}{H}-2$ and $0 < \gamma < H$. It is known that (see e.g Exercise 13.2 in \cite{hairerbook})

\begin{equation}\label{tayf}
f(y) = f(x) + \nabla f(x) (y-x) + O(|y-x|^{\theta + 1});~y,x \in \mathbb{R}^d.
\end{equation}
Expansion (\ref{tayf}) immediately implies that $\nabla f (X)$ is $\theta \gamma$-H\"older continuous. This concludes the proof.

\subsection{Proof of Lemma \ref{etarde}}
Next, we devote our attention to the proof of Lemma \ref{etarde} but at first, we need two technical lemmas. In the sequel, $a\lesssim_L b$ means $a\le C b$, where $C$ is a constant which depends on a parameter $L$.

\begin{lemma}\label{fin1}
For a given $\frac{1}{3} < \gamma < H < \frac{1}{2}$,

$$\Big\| J^{-1}_{\cdot}\circ V(Y)\mathds{1}_{[0,M]}\Big\|^2_{L_R(\mathbb{R}^d)}\lesssim_{T,H,\gamma} \max\Big\{ \|J^{-1}\|_{\infty,\gamma} \|\nabla V\|_\infty^2\|Y\|_\gamma^2; \|V(Y)\|_\infty \|J^{-1}\|^2_{\infty,\gamma}\Big\}~a.s,$$
for every $M >0$.
\end{lemma}
\begin{proof}
In order to alleviate notation, we write $\mathbf{v} = (v_1,v_2) \in \mathbb{R}^2_+\setminus D$. Fix an arbitrary initial condition $Y_0=x_0$ and $M>0$. By assumption $V \in C^3_b (\mathbb{R}^d; \mathbb{R}^{d\times d})$ so that

$$
\|V(Y)\|^p_\infty\lesssim \|\nabla V\|_\infty^p \big\{\|Y\|^p_\infty + |x_0|^p \big\} +|V(x_0)|^p,
$$
for every $1\le p < \infty $. Since $R(dt,T)$ is a positive finite measure on $[0,T]$, then

$$
\big\| J^{-1}\circ V(Y)\mathds{1}_{[0,M]}\big\|^2_{L_R(\mathbb{R}^d)}\lesssim \|J^{-1}\|^2_\infty \| V(Y)\|^2_\infty
$$
$$+ 2\int\int_{0\le v_1 < v_2} \Big| J^{-1}_{v_1}\circ V(Y_{v_1})\mathds{1}_{[0,M]}(v_1) - J^{-1}_{v_2}\circ V(Y_{v_2})\mathds{1}_{[0,M]}(v_2)\Big|^2|\partial^2R(\mathbf{v})|d\mathbf{v}$$
$$=: Q_1 + Q_2.$$
Let us decompose

$$Q_2 = 2\int\int_{0\le v_1 < v_2\le M} \Big| J^{-1}_{v_1}\circ V(Y_{v_1})\mathds{1}_{[0,M]}(v_1) - J^{-1}_{v_2}\circ V(Y_{v_2})\mathds{1}_{[0,M]}(v_2)\Big|^2|\partial^2R(\mathbf{v})|d\mathbf{v} $$
$$+ 2\int\int_{0\le v_1 < M < v_2} \Big| J^{-1}_{v_1}\circ V(Y_{v_1})\mathds{1}_{[0,M]}(v_1) - J^{-1}_{v_2}\circ V(Y_{v_2})\mathds{1}_{[0,M]}(v_2)\Big|^2|\partial^2R(\mathbf{v})|d\mathbf{v}$$
$$=:Q_{2,1} + Q_{2,2}.$$
Clearly,
$$
Q_{2,2}\lesssim \|J^{-1}\|_\infty^2 \|V(Y)\|_\infty^2 T^{2H}~a.s.
$$
We write
\begin{eqnarray}
\nonumber J^{-1}_{v_1}\circ V(Y_{v_1}) - J^{-1}_{v_2}\circ V(Y_{v_2})&=& J^{-1}_{v_1}\circ V(Y_{v_1}) - J^{-1}_{v_1}\circ V(Y_{v_2})\\
\label{fp1}&+& J^{-1}_{v_1}\circ V(Y_{v_2}) - J^{-1}_{v_2}\circ V(Y_{v_2}).
\end{eqnarray}
By using (\ref{fp1}), the $\gamma$-H\"older property of $(J^{-1},Y)$, the Lipschitz property $V$ and triangle inequality, we have

\begin{eqnarray*}
Q_{2,1}&\lesssim&  \max \Big\{  \|J^{-1}\|^2_\infty\|\nabla V\|_\infty^2 \|Y\|_\gamma^2; \|V(Y)\|_\infty^2 \|J^{-1}\|_\gamma^2\Big\}\\
&\times& \int_{0\le v_1 < v_2\le M}(v_2-v_1)^{2\gamma+2H-2}d\mathbf{v} < \infty.
\end{eqnarray*}
This concludes the proof.
\end{proof}

\begin{lemma}\label{fin2}
For a given $\frac{1}{3} < \gamma < H < \frac{1}{2}$, we have

\begin{eqnarray*}
\Big\| J^{-1}_{\cdot}\circ V(Y)\mathds{1}_{(N,M]}\Big\|^2_{L_R(\mathbb{R}^d)}&\lesssim_{T,H,\gamma}& \max\Big\{ \|J^{-1}\|^2_{\infty,\gamma} \|V(Y)\|^2_\infty; \|\nabla V\|_\infty^2\|Y\|_\gamma^2; \|V(Y)\|^2_\infty \Big\}\\
&\times& \Big\{ |T\wedge M - T\wedge N|^{2H} + |T\wedge M  - T\wedge N|^{2\gamma+2H}\Big\}~a.s,
\end{eqnarray*}
for every $N < M < \infty$.
\end{lemma}
\begin{proof}
In order to alleviate notation, we write $\mathbf{v} = (v_1,v_2) \in \mathbb{R}^2_+\setminus D$. Fix $N < M$. Triangle inequality yields


$$\Big\| J^{-1}_{\cdot}\circ V(Y)\mathds{1}_{(N,M]}\Big\|^2_{L_R(\mathbb{R}^d)}\lesssim \| J^{-1}\|_\infty^2 \| V(Y)\|_\infty^2 |T\wedge M  - N|^{2H}$$
$$+  \|J^{-1}\|^2_\infty \int_{\mathbb{R}^2_+\setminus D} | V(Y_{v_1})\mathds{1}_{(N,M]}(v_1) - V(Y_{v_2})\mathds{1}_{(N,M]}(v_2)|^2|\partial^2R(\mathbf{v})|d\mathbf{v}$$
$$+   \int_{\mathbb{R}^2_+\setminus D} \Big| \big[J^{-1}_{v_1}\circ V(Y_{v_2}) - J^{-1}_{v_2}\circ V(Y_{v_2})\big]\mathds{1}_{(N,M]}(v_2)\Big|^2|\partial^2R(\mathbf{v})|d\mathbf{v}.$$
We split

$$\int_{\mathbb{R}^2_+\setminus D} | V(Y_{v_1})\mathds{1}_{(N,M]}(v_1) - V(Y_{v_2})\mathds{1}_{(N,M]}(v_2)|^2|\partial^2R(\mathbf{v})|d\mathbf{v}$$
$$ = 2 \int_{0\le N < v_1 < v_2 \le M}|V(Y_{v_1}) - V(Y_{v_2})|^2 |\partial^2R(\mathbf{v})|d\mathbf{v}$$
$$+ 2 \int_{0\le N < v_1 < M < v_2}|V(Y_{v_1})|^2 |\partial^2R(\mathbf{v})|d\mathbf{v} $$
$$+ 2 \int_{0\le v_1 \le N < v_2\le M}|V(Y_{v_2})|^2 |\partial^2R(\mathbf{v})|d\mathbf{v}. $$
We observe
$$\int_{0\le N < v_1 < v_2 \le M}|V(Y_{v_1}) - V(Y_{v_2})|^2 |\partial^2R(\mathbf{v})|d\mathbf{v}
$$
$$\le \| \nabla V\|_\infty^2 \|Y\|_\gamma^2 \int_{N < v_1 < v_2\le M}(v_2-v_1)^{2\gamma}|\partial^2R(\mathbf{v})|d\mathbf{v},$$
where $\int_{N < v_1 < v_2\le M}(v_2-v_1)^{2\gamma}|\partial^2R(\mathbf{v})|d\mathbf{v} = 0$ if $T\le N < M$. Otherwise,

\begin{eqnarray*}
\int_{N < v_1 < v_2\le M}(v_2-v_1)^{2\gamma}|\partial^2R(\mathbf{v})|d\mathbf{v} &=& \int_N^{T\wedge M}\int_{v_1}^{M\wedge T}(v_2-v_1)^{2\gamma + 2H-2}d\mathbf{v}\\
&\lesssim&  |M\wedge T - N|^{2\gamma + 2H}.
\end{eqnarray*}
Therefore,

$$
\int_{0\le N < v_1 < v_2 \le M}|V(Y_{v_1}) - V(Y_{v_2})|^2 |\partial^2R(\mathbf{v})|d\mathbf{v}\lesssim \| \nabla V\|_\infty^2 \|Y\|_\gamma^2 |M\wedge T - N|^{2\gamma + 2H}.
$$
Next, we observe
$$
\int_{0\le N < v_1 < M < v_2}|V(Y_{v_1})|^2 |\partial^2R(\mathbf{v})|d\mathbf{v}\lesssim  \| V(Y)\|^2_\infty (M-N)^{2H},
$$
and

$$
\int_{0\le v_1 \le N  < v_2\le M}|V(Y_{v_1})|^2 |\partial^2R(\mathbf{v})|d\mathbf{v}\lesssim  \| V(Y)\|^2_\infty (M\wedge T-N)^{2H}.
$$
Lastly, we observe

$$\int_{\mathbb{R}^2_+\setminus D} \Big| \big[J^{-1}_{v_1}\circ V(Y_{v_2}) - J^{-1}_{v_2}\circ V(Y_{v_2})\big]\mathds{1}_{(N,M]}(v_2)\Big|^2|\partial^2R(\mathbf{v})|d\mathbf{v}$$
$$\le \|J^{-1}\|_\gamma^2 \| V(Y)\|^2_\infty\int_{\mathbb{R}^2_+ \setminus D}|v_1-v_2|^{2\gamma} \mathds{1}_{(N,M]}(v_2)|\partial^2 R(\mathbf{v})|d\mathbf{v}$$
$$ = 2\|J^{-1}\|_\gamma^2 \| V(Y)\|^2_\infty\int_{0\le v_1 < N < v_2\le M}|v_1-v_2|^{2\gamma} |\partial^2 R(\mathbf{v})|d\mathbf{v} $$
$$
= 2\|J^{-1}\|_\gamma^2 \| V(Y)\|^2_\infty\int_{0\le N\le v_1 < v_2\le M}|v_1-v_2|^{2\gamma} |\partial^2 R(\mathbf{v})|d\mathbf{v}
$$
$$
\lesssim  \|J^{-1}\|_\gamma^2 \| V(Y)\|^2_\infty \big\{|T\wedge M-N|^{2\gamma + 2H} + |T\wedge M-N|\big\}.
$$
This concludes the proof.
\end{proof}

\

\noindent \textbf{Proof of Lemma \ref{etarde}}. Fix $\frac{1}{3} < \gamma < H < \frac{1}{2}$ and $Y_0=x_0$. At first, it is well-known that $V\in C^3_b$ implies that $Y_t \in \mathbb{D}^{1,2}(\mathbb{R}^d)$ for every $t\ge 0$ and

$$\mathbf{D}_sY_t = J_t \circ J^{-1}_s\circ V(Y_s)\mathds{1}_{[0,t]}(s).$$
The H\"older regularity of $Y$ yields

$$\mathbb{E}|Y_t - Y_s|^2\le \mathbb{E}\|Y\|_\gamma^2|t-s|^{2\gamma}.$$
Moreover,

$$\mathbb{E}\|V(Y)\|^p_\infty\lesssim \{\|\nabla V\|^p_\infty\mathbb{E}\|Y\|_\gamma^p  + |V(x_0)|^p\}.$$
We know that 
$$
\mathbf{D} Y_t  - \mathbf{D}Y_s= J_t\circ J^{-1}\circ V(Y)\mathds{1}_{[0,t]} - J_s\circ J^{-1}\circ V(Y)\mathds{1}_{[0,s]}. 
$$
Then, by applying Lemmas \ref{fin1} and \ref{fin2} above, we get

\begin{eqnarray*}
\mathbb{E}\|\mathbf{D}Y_t - \mathbf{D}Y_s\|^2_{L_R(\mathbb{R}^{d\times d})}&\lesssim& \mathbb{E}\Big[ |J_t-J_s|^2\big\| J^{-1}\circ V(Y)\mathds{1}_{[0,t]}\big\|^2_{L_R(\mathbb{R}^{d\times d})} \Big]\\
&+&  \mathbb{E}\Big\| J_s\circ J^{-1}\circ V(Y)\{\mathds{1}_{[0,t]} - \mathds{1}_{[0,s]}\}\Big\|^2_{L_R(\mathbb{R}^{d\times d})}\\
&\lesssim&  |t-s|^{2\gamma} +  |t-s|^{2\gamma+2H}.
\end{eqnarray*}
This concludes the proof of Lemma \ref{etarde}.

\subsection{Proof of Theorem \ref{mainTH}}\label{proofTh2}
This section is devoted to the proof of Theorem \ref{mainTH}. Although it is possible to present a complete analysis under the general assumption C, in order to keep presentation simple, we restrict the analysis to the concrete case of the fractional Brownian motion $\frac{1}{4} < H < \frac{1}{2}$. In the sequel, we assume the hypotheses of Theorem \ref{mainTH} are in force. In this section, we fix $\frac{1}{4} < H < \frac{1}{2}$ and $Y \in \mathbb{D}^{1,2}(L_R(\mathbb{R}^d))$ is adapted.

First, by using (\ref{multfor}), we can write

\begin{eqnarray*}
Y^i_s \big( X^i_{s+\epsilon} - X^i_{s-\epsilon}\big) &=& Y^i_s \int_0^\infty \mathds{1}_{[s-\epsilon,s+\epsilon]}(r) e_i\boldsymbol{\delta} X_r\\
&=& \int_0^\infty Y^i_s \mathds{1}_{[s-\epsilon,s+\epsilon]}(r) e_i\boldsymbol{\delta} X_r + \big \langle \mathbf{D}Y^i_s, \mathds{1}_{[s-\epsilon,s+\epsilon]} e_i  \big \rangle_{L_R(\mathbb{R}^d)},
\end{eqnarray*}
for $1\le i\le d$, $s\ge 0$ and $\epsilon>0$. By applying Fubini's theorem for Skorohod integrals (see Prop. 10.3 in \cite{krukrusso}) and recalling (\ref{Ybar}), we have

\begin{eqnarray*}
\nonumber\frac{1}{2\epsilon}\int_\epsilon^{T-\epsilon} \langle Y_s, X_{s+\epsilon} - X_{s-\epsilon} \rangle ds &=& \int_0^T \bar{Y}^\epsilon_s\boldsymbol{\delta}X_s +\frac{1}{2\epsilon}\int_\epsilon^{T-\epsilon} \sum_{i=1}^d \langle \mathbf{D}Y^i_s, e_i\mathds{1}_{[s-\epsilon, s+\epsilon]} \rangle_{L_R(\mathbb{R}^d)}ds.
\end{eqnarray*}
If, in addition, there exists $q>2$ such that $\sup_{0\le t\le T}\mathbb{E}| Y_t|^q  < \infty$, then Assumption C(ii), Jensen and H\"older's inequality yield

\begin{eqnarray*}
\frac{1}{2\epsilon}\int_0^T \langle Y_s, X_{s+\epsilon} - X_{s-\epsilon} \rangle ds&=& \int_0^T \bar{Y}^\epsilon_u \boldsymbol{\delta} X_u\\ 
&+&  \frac{1}{2\epsilon}\int_\epsilon^{T-\epsilon} \sum_{i=1}^d \langle \mathbf{D}Y^i_s, e_i\mathds{1}_{[s-\epsilon, s+\epsilon]} \rangle_{L_R(\mathbb{R}^d)}ds + O_{L^2(\mathbb{P})}(\epsilon^{\alpha+2}).
\end{eqnarray*}

By Proposition \ref{epsilonlemma}, we know that

$$
\mathbb{E}\Bigg|\int_0^T \big(\bar{Y}^\epsilon_s  - Y_s\big)\boldsymbol{\delta} X_s \Bigg|^2\lesssim \epsilon^{2\gamma + 2H-1},
$$
for $0 < \gamma \le H$ and $2\gamma + 2H-1>0$. The next lemmas will present the precise asymptotic behavior of

\begin{equation*}
\textit{Tr}\big(\mathbf{D}Y\big)_{\epsilon}:=\frac{1}{2\epsilon}\int_\epsilon^{T-\epsilon} \sum_{i=1}^d \langle \mathbf{D}Y^i_s, e_i\mathds{1}_{[s-\epsilon, s+\epsilon]} \rangle_{L_R(\mathbb{R}^d)}ds,
\end{equation*}
as $\epsilon \downarrow 0$ which allows us to prove Theorem \ref{mainTH}.

We may decompose

\begin{eqnarray}
\nonumber \textit{Tr}\big(\mathbf{D}Y\big)_{\epsilon} &=& \frac{1}{2\epsilon}\int_\epsilon^{T-\epsilon} \textit{tr}[\mathbf{D}_{s-}Y_s] \big\langle \mathds{1}_{[0,s]}, \mathds{1}_{[s-\epsilon, s+\epsilon]} \big\rangle_{L_R} ds\\
\nonumber&+&\frac{1}{2\epsilon}\int_\epsilon^{T-\epsilon} \Big\langle \textit{tr}[(\mathbf{D} Y_s - \mathbf{D}_{s-}Y_s)]\mathds{1}_{[0,s]}, \mathds{1}_{[s-\epsilon, s+\epsilon]} \Big\rangle_{L_R}ds\\
\label{secs}&=:& \textit{Tr}_1\big(\mathbf{D}Y\big)_{\epsilon} + \textit{Tr}_2\big(\mathbf{D}Y\big)_{\epsilon}.
\end{eqnarray}

\begin{lemma}\label{Tr1lemma}
Assume that

\begin{equation}\label{trint}
\sup_{0\le s\le T}\mathbb{E}|\text{tr}[\mathbf{D}_{s-}Y_s]|^2 < \infty.
\end{equation}
Then,

\begin{equation}\label{tr1}
\mathbb{E}\Big| \text{Tr}_1 (\mathbf{D}Y)_\epsilon - \frac{1}{2}\int_0^T \text{tr}[\mathbf{D}_{s-}Y_s]d\boldsymbol{v}_s\Big|^2\lesssim  \epsilon^{\frac{6H-1}{2}}
\end{equation}
for every $\epsilon >0$ such that $\epsilon^{0.75} + 2\epsilon < T$, where $\boldsymbol{v}(s) = s^{2H}; s\ge 0$.
\end{lemma}
\begin{proof}
In the sequel, we denote $\boldsymbol{v}(s) := s^{2H}; s\ge 0$. The $n$-th derivative of a function $f$ will be denoted by $f^{(n)}$. We observe $\boldsymbol{v}$ satisfies the following properties: $s\mapsto\boldsymbol{v}(s) $ is a $C^3(0,T)$ non-decreasing map and $s\mapsto |\boldsymbol{v}^{(3)}(s)|$ is non-increasing. In addition, there exists $\beta \in (0,1)$ such that $|\boldsymbol{v}^{(3)}(\epsilon^{\beta})|\epsilon^2\rightarrow 0$ and $\epsilon^{\beta (2H+1)-1} \rightarrow 0$ as $\epsilon\rightarrow 0^+$. Indeed, $\boldsymbol{v}^{(3)}(s) = c_H s^{2H-3}$ for a positive constant $c_{H}$ and notice

$$\frac{1}{1+2H} < \frac{2}{3-2H}$$
for $H > \frac{1}{6}$. Therefore, we can take any $\beta$ realizing

\begin{equation}\label{betap}
0 < \frac{1}{1+2H} < \beta < \frac{2}{3-2H} < 1,
\end{equation}
and for any such choice, we have $\epsilon^{\beta (2H-3)+2}\rightarrow 0$ and $\epsilon^{\beta(2H+1)-1}\rightarrow 0,
$ as $\epsilon \rightarrow 0^+$. Having said that, by using (\ref{cov1}) and elementary computations for the covariance, we can write

$$
\frac{1}{2\epsilon}\int_\epsilon^{T-\epsilon} \Big\langle \textit{tr}[\mathbf{D}_{s-}Y_s]\mathds{1}_{[0,s]}, \mathds{1}_{[s-\epsilon, s+\epsilon]} \Big\rangle_{L_R} ds
=\int_\epsilon^{T-\epsilon} \textit{tr}\big[\mathbf{D}_{s-}Y_s\big] dF_{\epsilon}(s),
$$

where
\begin{eqnarray*}
F_{\epsilon}(x) &=& \frac{1}{2\epsilon}\int_0^x \langle \mathds{1}_{[0,r]}, \mathds{1}_{[r-\epsilon, r+\epsilon]} \rangle_{L_R}dr\\
&=&\frac{1}{2} \int_0^x \Big[\frac{\boldsymbol{v}(r+\epsilon) - \boldsymbol{v}(r-\epsilon)}{2\epsilon}\Big]dr; x\ge 0.
\end{eqnarray*}
We denote

$$V_\epsilon(s) =  \frac{\boldsymbol{v}(s+\epsilon) - \boldsymbol{v}(s-\epsilon)}{2\epsilon} - \boldsymbol{v}^{(1)}(s); \epsilon < s  < T-\epsilon.$$
Taylor formula and mean value theorem yield

\begin{equation}\label{Vepsiest1}
V_\epsilon(s) = \frac{\epsilon^2}{6}\boldsymbol{v}^{(3)}(a(s,\epsilon)),
\end{equation}
where $a(s,\epsilon) \in (s-\epsilon,s+\epsilon)$ and $\epsilon < s  < T-\epsilon$. Fix $0 < \beta < 1$ according to (\ref{betap}). We split

\begin{eqnarray*}
\int_\epsilon^{T-\epsilon} \textit{tr}\big[\mathbf{D}_{s-}Y_s\big]\Big[F^{(1)}_{\epsilon}(s) - \frac{1}{2}\boldsymbol{v}^{(1)}(s)\Big]ds &=&\frac{1}{2}\int_{\epsilon^\beta + \epsilon}^{T-\epsilon}\textit{tr}\big[\mathbf{D}_{s-}Y_s\big]V_\epsilon(s)ds\\
&+& \frac{1}{2}\int_\epsilon^{\epsilon^\beta + \epsilon}\textit{tr}\big[\mathbf{D}_{s-}Y_s\big]V_\epsilon(s)ds,
\end{eqnarray*}
where we may assume $\epsilon^\beta + 2\epsilon < T$. By (\ref{Vepsiest1}), we observe

\begin{equation}\label{Vepsiest2}
\big|V_\epsilon(s)\big|= \frac{\epsilon^2}{6}|\boldsymbol{v}^{(3)}(a(s,\epsilon))|\lesssim \epsilon^2|\boldsymbol{v}^{(3)}(\epsilon^\beta)|,
\end{equation}
for every $s \in (\epsilon^\beta + \epsilon,T-\epsilon)$. By applying Jensen's inequality, using assumption (\ref{trint}) and (\ref{Vepsiest2}), we get

\begin{equation}\label{f1trace1}
\mathbb{E}\Bigg|\frac{1}{2}\int_{\epsilon^\beta + \epsilon}^{T-\epsilon}\textit{tr}[\mathbf{D}_{s-}Y_s]V_\epsilon(s)ds\Bigg|^2\lesssim \big(\epsilon^2|\boldsymbol{v}^{(3)}(\epsilon^\beta)|\big)^2 \rightarrow 0,
\end{equation}
as $\epsilon \rightarrow 0^+$. Fubini's theorem and (\ref{trint}) yield

$$\mathbb{E}\Bigg|\int_\epsilon^{\epsilon^\beta + \epsilon}\textit{tr}[\mathbf{D}_{s-}Y_s]V_\epsilon(s)ds\Bigg|^2
$$
$$=\mathbb{E}\int_\epsilon^{\epsilon^\beta + \epsilon}\int_\epsilon^{\epsilon^\beta + \epsilon} \textit{tr}[\mathbf{D}_{s-}Y_s] \textit{tr}[\mathbf{D}_{t-}Y_t]V_\epsilon(s) V_\epsilon(t)dsdt$$

$$\lesssim \int_\epsilon^{\epsilon^\beta + \epsilon}\int_\epsilon^{\epsilon^\beta + \epsilon}|V_\epsilon(s)| |V_\epsilon(t)|dsdt
$$
\begin{equation}\label{f2trace1}
\lesssim  \epsilon^{2[\beta (2H +1)-1]},
\end{equation}
for every $\epsilon>0$ sufficiently small. Again, by (\ref{trint}), we have
\begin{equation}\label{f4trace1}
\mathbb{E}\Big|\int_{T-\epsilon}^{T} \textit{tr}\big[\mathbf{D}_{s-}Y_s\big]\boldsymbol{v}^{(1)}_sds\Big|^2 + \mathbb{E}\Big|\int_{0}^{\epsilon} \textit{tr}\big[\mathbf{D}_{s-}Y_s\big]\boldsymbol{v}^{(1)}_sds\Big|^2\lesssim \epsilon^{4H},
\end{equation}
for every $\epsilon$ sufficiently small.  Summing up the estimates (\ref{f1trace1}), (\ref{f2trace1}) and (\ref{f4trace1}), we obtain

\begin{equation}\label{fcon}
\mathbb{E}\Big| \textit{Tr}_1 (\mathbf{D}Y)_\epsilon - \frac{1}{2}\int_0^T \textit{tr}[\mathbf{D}_{s-}Y_s]\boldsymbol{v}^{(1)}_sds\Big|^2\lesssim  \Big\{\big(\epsilon^2 \boldsymbol{v}^{(3)}(\epsilon^\beta))^2 + \epsilon^{2[\beta (2H+1)-1]} \Big\} ,
\end{equation}
for every $\epsilon >0$ such that $\epsilon^\beta + 2\epsilon < T$. Now, we will optimize the right-hand side of (\ref{fcon}). Let us consider the following bound for the right-hand side of (\ref{fcon}):

$$\epsilon^{2\big(2\beta H + 2-3\beta\big)} + \epsilon^{2\big(2\beta H + \beta-1\big)}\le 2 \max \Big\{\epsilon^{2\big(2\beta H + 2-3\beta\big)},\epsilon^{2\big(2\beta H + \beta-1\big)}\Big\},$$
where $\beta \in (\frac{1}{1+2H},  \frac{2}{3-2H})$. Next, we aim to compute

\begin{equation}\label{optbeta}
\arg\min_{\beta \in \big(\frac{1}{1+2H},  \frac{2}{3-2H}\big)}\max \Big\{\epsilon^{\big(2\beta H + 2-3\beta\big)},\epsilon^{\big(2\beta H + \beta-1\big)}\Big\}.
\end{equation}
We observe

$$\frac{1}{2} < \frac{1}{1+2H} < \frac{2}{3} < 0.80 < \frac{2}{3-2H} < 1,$$
and

$$
2\beta H + 2-3\beta \ge 2\beta H + \beta-1,
$$
whenever $\frac{1}{1+2H} <  \beta \le 0.75 < \frac{2}{3-2H}$ and

$$
2\beta H + 2-3\beta < 2\beta H + \beta-1,
$$
whenever $ 0.75 < \beta <  \frac{2}{3-2H}$. Moreover,
$$2\beta H -3\beta +2 = 2\beta H +\beta -1\Longleftrightarrow\beta = 0.75.$$
The fact that $\beta \mapsto 2\beta H -3\beta +2$ is strictly decreasing and the constant which appears in the right-hand side of (\ref{fcon}) does not depend on $\beta$ allow us to choose $\beta^*=0.75$ and this is the optimal choice realizing (\ref{optbeta}). Therefore,

$$\epsilon^{2\big(2\beta H + 2-3\beta\big)} + \epsilon^{2\big(2\beta H + \beta-1\big)}\le 2 \epsilon^{2\big(2\times 0.75 H + 0.75-1\big)} = 2\epsilon^{\frac{6H-1}{2}}.$$
This concludes the proof.
\end{proof}

Next, we devote our attention to the component $\textit{Tr}_2\big(\mathbf{D}Y\big)_{\epsilon}$.
\begin{lemma}\label{ratepr2}
Assume that $\text{tr}[\mathbf{D}_\cdot Y_s]$ has continuous paths on $[0,s]$ for every $s\le T$ and Assumption S2 holds true with parameters $\alpha=2H-2, \frac{1}{4} < H < \frac{1}{2}$ and $\eta+\alpha+1>0$. Then,

$$
\mathbb{E}\Big|\text{Tr}_2(\mathbf{D}Y)_\epsilon - \int_{0\le r_1 < r_2\le T}\text{tr}[\mathbf{D}_{r_1}Y_{r_2} - \mathbf{D}_{r_2-}Y_{r_2}]\partial^2R(r_1,r_2)dr_1dr_2 \Big|^2\lesssim \epsilon^{2(\eta + \alpha+1)}\rightarrow 0,
$$
as $\epsilon \rightarrow 0^+$.
\end{lemma}

\begin{proof}
Let us denote $\Delta_1 =\{(a,b)\in \mathbb{R}^2_+\setminus D; a < b\}$ and we define

$$h_\epsilon(r_1,r_2;s) := \Big\{ \textit{tr}[\mathbf{D}_{r_1}Y_s - \mathbf{D}_{s-} Y_s]\mathds{1}_{[0,s]}(r_1) - \textit{tr}[\mathbf{D}_{r_2}Y_s - \mathbf{D}_{s-} Y_s]\mathds{1}_{[0,s]}(r_2)\Big\} $$
$$\times \Big\{ \mathds{1}_{[s-\epsilon,s+\epsilon]}(r_1)  - \mathds{1}_{[s-\epsilon,s+\epsilon]}(r_2)\Big\},$$
for $(s_1,r_2) \in \Delta_1$ and $\epsilon\le s < T-\epsilon$. By the very definition,

$$\textit{Tr}_2 (\mathbf{D}Y)_\epsilon = \frac{1}{2\epsilon}\int_\epsilon^{T-\epsilon}\int_0^\infty \textit{tr}[\mathbf{D}_rY_s - \mathbf{D}_{s-} Y_s]\mathds{1}_{[0,s]}(r) \mathds{1}_{[s-\epsilon, s+\epsilon]}(r) \partial_r R(r,T)drds $$
\begin{equation}\label{nsplit}
-\frac{1}{2\epsilon}\int_\epsilon^{T-\epsilon}\int_{\Delta_1}h_\epsilon(r_1,r_2;s)\partial^2 R(r_1,r_2)dr_1dr_2ds =: I_{1,\epsilon} + I_{2,\epsilon}.
\end{equation}
We can write


$$I_{1,\epsilon}  =\frac{1}{2\epsilon} \int_\epsilon^{T-\epsilon} \int_{s-\epsilon}^s \textit{tr}[\mathbf{D}_rY_s - \mathbf{D}_{s-} Y_s] \partial_r R(r,T)dr ds.$$
Jensen's inequality and Assumption S2 yield

\begin{eqnarray}
\nonumber\mathbb{E}\Bigg|\frac{1}{2\epsilon} \int_\epsilon^{T-\epsilon} \int_{s-\epsilon}^s \textit{tr}[\mathbf{D}_rY_s - \mathbf{D}_{s-} Y_s] \partial_r R(r,T)dr ds\Bigg|^2&\lesssim& \int_\epsilon^{T-\epsilon} \frac{1}{\epsilon}\int^s_{s-\epsilon}(s-r)^{2 \eta}|\partial_r R(r,T)|^2drds\\
\label{T1}&\lesssim& \epsilon^{2(\eta + \alpha+1)}\rightarrow 0,
\end{eqnarray}
as $\epsilon \rightarrow 0^{+}$.


Now, we deal with the second term in (\ref{nsplit}). In case $\epsilon \le s$, we observe



\begin{equation}\label{defsp1}
h_\epsilon(r_1,r_2;s)=\left\{
\begin{array}{rl}
-\textit{tr}\big[\mathbf{D}_{r_1}Y_s - \mathbf{D}_{r_2}Y_{s}\big];& \hbox{if} \ 0 < r_1 < s-\epsilon, s-\epsilon \le r_2 \le s\\
-\textit{tr}\big[\mathbf{D}_{r_1}Y_s - \mathbf{D}_{s-}Y_{s}\big];& \hbox{if} \ 0 < r_1 < s-\epsilon, s < r_2 \le s+\epsilon\\
\textit{tr}\big[\mathbf{D}_{r_1}Y_s - \mathbf{D}_{s-}Y_{s}\big];& \hbox{if} \ 0 < s+ \epsilon <r_2,s-\epsilon \le r_1 \le s.\\
\end{array}
\right.
\end{equation}
As a result, we can write $I_{2,\epsilon}$ as

$$\frac{-1}{2\epsilon}\int_\epsilon^{T-\epsilon} \int_{\Delta_1}h_\epsilon(r_1,r_2;s) \partial^2 R(r_1,r_2)dr_1dr_2ds $$

$$ =  \frac{1}{2\epsilon}\int_\epsilon^{T-\epsilon} \int_{0}^{s-\epsilon}\int_s^{s+\epsilon} \textit{tr}[ \mathbf{D}_{r_1}Y_s - \mathbf{D}_{s-}Y_s] \partial^2 R(r_1,r_2)dr_2dr_1ds $$
$$+ \frac{1}{2\epsilon}\int_\epsilon^{T-\epsilon}  \int_{0}^{s-\epsilon}\int_{s-\epsilon}^{s} \textit{tr}[\mathbf{D}_{r_1}Y_s - \mathbf{D}_{r_2}Y_s] \partial^2 R(r_1,r_2)dr_2dr_1ds$$
$$ -  \frac{1}{2\epsilon}\int_\epsilon^{T-\epsilon}  \int_{s+\epsilon}^\infty\int_{s-\epsilon}^s  \textit{tr}[\mathbf{D}_{r_1}Y_s - \mathbf{D}_{s-}Y_s] \partial^2 R(r_1,r_2)dr_1dr_2ds$$
$$=:I_{2,\epsilon,1} + I_{2,\epsilon,2} + I_{2,\epsilon,3}.$$
At first, we estimate $I_{2,\epsilon,3}$. By using Assumption S2, Fubini's theorem and Cauchy-Schwarz's inequality, we have

$$
\nonumber\mathbb{E}|I_{2,\epsilon,3}|^2 \lesssim \Bigg(\frac{1}{\epsilon}\int_\epsilon^T\int_{s+\epsilon}^T\int_{s-\epsilon}^s(s-r_1)^\eta (r_2 - r_1)^{\alpha}dr_1dr_2ds\Bigg)^2.
$$
A direct computation shows that

$$ \frac{1}{\epsilon}\int_\epsilon^T\int_{s+\epsilon}^T\int_{s-\epsilon}^s(s-r_1)^\eta (r_2 - r_1)^{\alpha}dr_1dr_2ds\lesssim \epsilon^{\eta + \alpha +1}.$$
Therefore,

\begin{equation}\label{T4}
\mathbb{E}|I_{2,\epsilon,3}|^2\lesssim \epsilon^{2(\eta + \alpha + 1)}.
\end{equation}


We now investigate

$$I_{2,\epsilon,1} - \frac{1}{2}\int_0^T\int_0^{s}\textit{tr}[\mathbf{D}_{r_1}Y_{s} - \mathbf{D}_{s-}Y_{s}]\partial^2R(r_1,s)dr_1ds$$
 $$ + I_{2,\epsilon,2}-\frac{1}{2}\int_0^T\int_0^{s}\textit{tr}[\mathbf{D}_{r_1}Y_{s} - \mathbf{D}_{s-}Y_{s}]\partial^2R(r_1,s)dr_1ds.$$

It is convenient to split it as

$$
\frac{1}{2}\int_0^T\int_0^{s}\textit{tr}[\mathbf{D}_{r_1}Y_{s} - \mathbf{D}_{s-}Y_{s}]\partial^2R(r_1,s)dr_1ds
$$
$$=\frac{1}{2}\int_0^\epsilon \int_0^{s}\textit{tr}[\mathbf{D}_{r_1}Y_{s} - \mathbf{D}_{s-}Y_{s}]\partial^2R(r_1,s)dr_1ds
$$
$$+ \frac{1}{2}\int_\epsilon^{T-\epsilon}\int_0^{s-\epsilon}\textit{tr}[\mathbf{D}_{r_1}Y_{s} - \mathbf{D}_{s-}Y_{s}]\partial^2R(r_1,s)dr_1ds
$$
$$+\frac{1}{2}\int_\epsilon^{T-\epsilon}\int_{s-\epsilon}^{s}\textit{tr}[\mathbf{D}_{r_1}Y_{s} - \mathbf{D}_{s-}Y_{s}]\partial^2R(r_1,s)dr_1ds
$$
$$+\frac{1}{2}\int_{T-\epsilon}^T \int_0^{s}\textit{tr}[\mathbf{D}_{r_1}Y_{s} - \mathbf{D}_{s-}Y_{s}]\partial^2R(r_1,s)dr_1ds$$ 
$$=: J_{1,\epsilon} + J_{2,\epsilon} + J_{3,\epsilon} + J_{4,\epsilon}.
$$
At first, we observe Fubini's theorem, Assumption S2 and Cauchy-Schwartz's inequality yield

\begin{eqnarray*}\label{jr2}
\mathbb{E}|J_{3,\epsilon}|^2&\lesssim& \Bigg(\int_\epsilon^{T-\epsilon}\int_{s-\epsilon}^{s} (s-r_1)^\eta |\partial^2R(r_1,s)|dr_1ds\Bigg)^2\\
&\lesssim& \Big(\int_\epsilon^T \int_{s-\epsilon}^s(s-r_1)^{\eta+\alpha}dr_1ds\Big)^2 \lesssim \epsilon^{2(\eta+\alpha +1)},
\end{eqnarray*}
for every $\epsilon >0$. Similarly,

\begin{equation}\label{jr1}
\mathbb{E}|J_{1,\epsilon}|^2\le \Big( \int_0^\epsilon \int_0^s (s-r_1)^\eta |\partial^2 R(r_1,s)|dr_1ds\Big)^2
\lesssim \epsilon^{2(\alpha+\eta+2)}
\end{equation}
and

\begin{equation}\label{jr3}
\mathbb{E}|J_{4,\epsilon}|^2\le \Big( \int_{T-\epsilon}^T \int_0^s (s-r_1)^\eta |\partial^2 R(r_1,s)|dr_1ds\Big)^2
\lesssim  \epsilon^{2(\alpha+\eta+2)}.
\end{equation}
Now, we observe we can write

\begin{eqnarray*}
I_{2,\epsilon,1} - J_{2,\epsilon} = \frac{1}{2}\int_\epsilon^{T-\epsilon} \int_0^{s-\epsilon}\textit{tr}[\mathbf{D}_{r_1}Y_s - \mathbf{D}_{s-}Y_s]g_\epsilon(r_1,s)
dr_1ds
\end{eqnarray*}
where $g_\epsilon(r_1,s):= \frac{1}{\epsilon} \int_s^{s+\epsilon}\partial^2 R(r_1,r_2)dr_2 - \partial^2R(r_1,s)$ for $0\le r_1 < s-\epsilon$. By Fubini's theorem and using Assumption S2 jointly with Cauchy-Schwartz's inequality, we have

\begin{eqnarray*}
\mathbb{E}|I_{2,\epsilon,1} - J_{2,\epsilon}|^2 &=& \mathbb{E}\int_{Q_\epsilon\times Q_\epsilon}\textit{tr}[\mathbf{D}_{r_1}Y_s - \mathbf{D}_{s-}Y_s]\textit{tr}[\mathbf{D}_{v_1}Y_z - \mathbf{D}_{z-}Y_z] g_\epsilon(r_1,s)\\
&\times& g_\epsilon(v_1,z)dr_1ds dv_1dz \lesssim \Bigg(\int_{Q_\epsilon} (s-r_1)^\eta|g_\epsilon(r_1,s)|dr_1ds\Bigg)^2,
\end{eqnarray*}
where $Q_\epsilon = \{(x,y); 0 \le x < y-\epsilon, \epsilon < y < T-\epsilon\}$. By using the fact that $s\mapsto \frac{\partial^3 R}{\partial s^2\partial r} (r_1,s)$ is continuous on $(r_1,T)$, we can make use of Taylor expansion to estimate $g_\epsilon$. We observe for each $r_1 < s < s+\epsilon$, there exists $\bar{s}_\epsilon$ with $r_1 < s < \bar{s}_\epsilon < s+\epsilon$ realizing

$$g_\epsilon(r_1,s) = \frac{1}{2}\frac{\partial^3 R}{\partial s^2 \partial r_1}(r_1,\bar{s}_\epsilon)\epsilon; \quad r_1 < s < \bar{s}_\epsilon < s+\epsilon < T.$$
The function $(\cdot  - r_1)^{\alpha-1}$ is decreasing and hence

$$|g_\epsilon(r_1,s)|\le \frac{1}{2}\Big|\frac{\partial^3 R}{\partial s^2 \partial r_1}(r_1,\bar{s}_\epsilon)\Big|\epsilon\lesssim (s-r_1)^{\alpha-1}\epsilon, $$
for every $(r_1, s) \in Q_\epsilon$. Therefore,

\begin{equation}\label{T5}
\mathbb{E}| I_{2,\epsilon,1} - J_{2,\epsilon}    |^2\lesssim \Bigg(\epsilon \int_{Q_\epsilon}(s-r_1)^{\eta + \alpha -1}dr_1ds \Bigg)^2
\lesssim \epsilon^{2(\eta + \alpha +1)},
\end{equation}
for every $\epsilon >0$ sufficiently small. In view of (\ref{T1}), (\ref{T4}), (\ref{jr2}), (\ref{jr1}), (\ref{jr3}) and (\ref{T5}), it remains to estimate $I_{2,\epsilon,2}-J_{2,\epsilon}$. For this purpose, we write


$$I_{2,\epsilon,2}  = \frac{1}{2\epsilon}\int_\epsilon^{T-\epsilon}  \int_{0}^{s-\epsilon}\int_{s-\epsilon}^{s} \textit{tr}[\mathbf{D}_{r_1}Y_s - \mathbf{D}_{r_2}Y_s] \big\{\partial^2 R(r_1,r_2)   - \partial^2 R(r_1,s)\big\}dr_2dr_1ds$$
 $$+\frac{1}{2\epsilon}\int_\epsilon^{T-\epsilon}  \int_{0}^{s-\epsilon}\int_{s-\epsilon}^{s} \textit{tr}[\mathbf{D}_{r_1}Y_s - \mathbf{D}_{r_2}Y_s] \partial^2 R(r_1,s)dr_2dr_1ds.$$
Mean value theorem yields

$$\partial^2 R(r_1,s) - \partial^2 R(r_1,r_2) = \frac{\partial^3 R}{\partial s^2 \partial r_1} (r_1, \bar{r})(s-r_2),$$
on $r_1 < s-\epsilon < r_2 < \bar{r} < s$. Therefore,

\begin{equation}\label{deltaSR}
|\partial^2 R(r_1,s) - \partial^2 R(r_1,r_2)| \le  \Big|\frac{\partial^3 R}{\partial s^2 \partial r_1} (r_1, \bar{r})\Big|\epsilon\lesssim \epsilon (r_2-r_1)^{\alpha-1},
\end{equation}
on $r_1 < s-\epsilon < r_2 < \bar{r}< s$. Therefore, Assumption S2 and (\ref{deltaSR}) allow us to apply Fubini's theorem and we get


$$\mathbb{E}\Bigg|\frac{1}{2\epsilon}\int_\epsilon^{T-\epsilon}  \int_{0}^{s-\epsilon}\int_{s-\epsilon}^{s} \textit{tr}[\mathbf{D}_{r_1}Y_s - \mathbf{D}_{r_2}Y_s] \big\{ \partial^2 R(r_1,s) - \partial^2 R(r_1,r_2) \big\}dr_2dr_1ds\Bigg|^2$$
\begin{equation}\label{T8}
\lesssim \Bigg(\int_\epsilon^{T-\epsilon}  \int_{0}^{s-\epsilon}\int_{s-\epsilon}^{s} (r_2-r_1)^{\eta+\alpha-1}dr_2dr_1ds\Bigg)^2
\lesssim \epsilon^{2(\eta +\alpha+1)},
\end{equation}
as $\epsilon \rightarrow 0^+$. Next, we observe
$$\frac{1}{2\epsilon}\int_\epsilon^{T-\epsilon}  \int_{0}^{s-\epsilon}\int_{s-\epsilon}^{s} \textit{tr}[\mathbf{D}_{r_1}Y_s - \mathbf{D}_{r_2}Y_s] \partial^2 R(r_1,s)dr_2dr_1ds - J_{2,\epsilon}$$
\begin{equation}\label{sf1}
=\frac{1}{2}\int_\epsilon^{T-\epsilon} \int_0^{s-\epsilon}\Big\{ \textit{tr}[\mathbf{D}_{s-}Y_s] - \frac{1}{\epsilon}\int_{s-\epsilon}^s \textit{tr}[\mathbf{D}_{r_2}Y_s]dr_2\Big\}\partial^2 R(r_1,s)dr_1ds.
\end{equation}
By mean value theorem, we can write

\begin{equation}\label{sf2}
\textit{tr}[\mathbf{D}_{s-}Y_s]-\frac{1}{\epsilon}\int_{s-\epsilon}^s \textit{tr}[\mathbf{D}_{r_2}Y_s]dr_2 = \textit{tr}[\mathbf{D}_{s-}Y_s]-\textit{tr}[\mathbf{D}_{s'_\epsilon}Y_s],
\end{equation}
on $r_1 < s-\epsilon < s'_\epsilon < s$. By (\ref{sf1}), (\ref{sf2}) and again by using Fubini's theorem and Assumption S2, we arrive at

$$\mathbb{E}\Bigg|\frac{1}{2\epsilon}\int_\epsilon^{T-\epsilon}  \int_{0}^{s-\epsilon}\int_{s-\epsilon}^{s} \textit{tr}[\mathbf{D}_{r_1}Y_s - \mathbf{D}_{r_2}Y_s] \partial^2 R(r_1,s)dr_2dr_1ds - J_{2,\epsilon}\Bigg|^2$$
\begin{equation}\label{T9}
\lesssim \Bigg( \epsilon^\eta \int_\epsilon^{T-\epsilon} \int_{0}^{s-\epsilon}(s-r_1)^\alpha dr_1ds \Bigg)^2 \lesssim \epsilon^{2(\eta + \alpha +1)},
\end{equation}
as $\epsilon\rightarrow 0^+$. The estimates (\ref{T8}) and (\ref{T9}) show

\begin{equation}\label{T10}
\mathbb{E}|I_{2,\epsilon,2} - J_{2,\epsilon}|^2\lesssim \epsilon^{2(\alpha + \eta + 1)}\rightarrow 0,
\end{equation}
as $\epsilon \rightarrow 0^+$. The estimates (\ref{T1}), (\ref{T4}), (\ref{jr2}), (\ref{jr1}), (\ref{jr3}), (\ref{T5}), (\ref{T10}) allow us to conclude the proof.
\end{proof}

\

\textbf{Proof of Theorem \ref{mainTH}:} Fix $\frac{1}{4} < H < \frac{1}{2}$, $0 < \gamma \le H$ and $\eta>0$ such that $2\gamma +2H-1>0$ and $\eta +2H-1>0$. Recall, we can write
\begin{eqnarray}
\label{div0}\frac{1}{2\epsilon}\int_0^T \langle Y_s, X_{s+\epsilon} - X_{s-\epsilon} \rangle ds&=& \int_0^T \bar{Y}^\epsilon_u \boldsymbol{\delta} X_u\\
\nonumber& + & \frac{1}{2\epsilon}\int_\epsilon^{T-\epsilon} \sum_{i=1}^d \langle \mathbf{D}Y^i_s, e_i\mathds{1}_{[s-\epsilon, s+\epsilon]} \rangle_{L_R(\mathbb{R}^d)}ds\\
\nonumber&+&O_{L^2(\mathbb{P})}(\epsilon^{2H}),
\end{eqnarray}
where $\bar{Y}$ is given by (\ref{Ybar}). By applying Proposition \ref{epsilonlemma} and Lemmas \ref{Tr1lemma} and \ref{ratepr2} above, we get

$$
\int_0^T Y_s d^0X_s = \int_0^T Y_s \boldsymbol{\delta}X_s + H \int_0^T \textit{tr}[\mathbf{D}_{s-}Y_s]s^{2H-1}ds$$
$$
+\int_{0 \le r_1 < r_2 \le T}\textit{tr}[\mathbf{D}_{r_1}Y_{r_2} - \mathbf{D}_{r_2-}Y_{r_2}]\partial^2R(r_1,r_2)dr_1dr_2.
$$
In addition,
\begin{equation}\label{rth2}
\mathbb{E}\Bigg| \int_0^T Y_sd^0X_s - I^0(\epsilon,Y,dX)(T)  \Bigg|^2 \lesssim \{\epsilon^{\frac{6H-1}{2}} + \epsilon^{2\gamma + 2H-1} + \epsilon^{2(\eta + 2H -1)}\},
\end{equation}
for every $\epsilon>0$ sufficiently small. The leading term in the right-hand side of (\ref{rth2}) is $\epsilon^{2\gamma + 2H-1} + \epsilon^{2(\eta + 2H -1)}$. In case $\frac{1}{3} < H < \frac{1}{2}$ and $(Y,Y') \in \mathcal{D}_X(\mathbb{R}^d)$ satisfies assumptions 1, 2 and 3 of Theorem \ref{gaussiancase}, then $(Y,Y')\in \mathcal{D}_X(\mathbb{R}^d)$ is rough stochastically integrable and the estimate (\ref{rth2}) holds for the stochastic rough integral as well. This concludes the proof of Theorem \ref{mainTH}.

\

\

\noindent \textbf{Funding}. This research was supported by MATH-AmSud 2018 (grant 88887.197425/2018-00) and Funda\c{c}\~ao de Apoio a Pesquisa do Destrito Federal (FAPDF grant 00193-00000229/2021-21).
The research of the second named author was partially supported by
the ANR-22-CE40-0015-01 (SDAIM).

\bibliographystyle{plain}
\bibliography{../../../../BIBLIO_FILE/ref}       

\def\cprime{$'$}
\begin{thebibliography}{10}

\bibitem{alosleon}
Elisa Al{\`o}s, Jorge~A. Le{\'o}n, and David Nualart.
\newblock Stochastic {S}tratonovich calculus f{B}m for fractional {B}rownian
  motion with {H}urst parameter less than {$1/2$}.
\newblock {\em Taiwanese J. Math.}, 5(3):609--632, 2001.

\bibitem{alos2001stochastic}
Elisa Al\`os, Olivier Mazet, and David Nualart.
\newblock Stochastic calculus with respect to {G}aussian processes.
\newblock {\em Ann. Probab.}, pages 766--801, 2001.

\bibitem{alos2003}
Elisa Al{\`o}s and David Nualart.
\newblock Stochastic integration with respect to the fractional {B}rownian
  motion.
\newblock {\em Stoch. stoch. rep.}, 75(3):129--152, 2003.

\bibitem{butkovsky2021}
Oleg Butkovsky, Konstantinos Dareiotis, and M{\'a}t{\'e} Gerencs{\'e}r.
\newblock Approximation of {SDE}s: a stochastic sewing approach.
\newblock {\em Probab. Theory Relat. Fields}, 181(4):975--1034, 2021.

\bibitem{cassfrizvic}
Thomas Cass, Peter~K. Friz, and Nicolas Victoir.
\newblock Non-degeneracy of {W}iener functionals arising from rough
  differential equations.
\newblock {\em Trans. Am. Math. Soc.}, 361(6):3359--3371, 2009.

\bibitem{cass2015sm}
Thomas Cass, Martin Hairer, Christian Litterer, and Samy Tindel.
\newblock Smoothness of the density for solutions to {G}aussian rough
  differential equations.
\newblock {\em Ann. Probab.}, 43(1):188--239, 2015.

\bibitem{cass2019}
Thomas Cass and Nengli Lim.
\newblock A {S}tratonovich-{S}korohod integral formula for {G}aussian rough
  paths.
\newblock {\em Ann. Probab.}, 47(1):1--60, 2019.

\bibitem{cass2021}
Thomas Cass and Nengli Lim.
\newblock Skorohod and rough integration for stochastic differential equations
  driven by {V}olterra processes.
\newblock In {\em Ann. inst. Henri Pointcar\'e (B) Probab. Stat.}, volume~57,
  pages 132--168. Institut Henri Poincar{\'e}, 2021.

\bibitem{cass2013int}
Thomas Cass, Christian Litterer, and Terry Lyons.
\newblock Integrability and tail estimates for {G}aussian rough differential
  equations.
\newblock {\em Ann. Probab.}, 41(4):3026--3050, 2013.

\bibitem{cheridito2005}
Patrick Cheridito and David Nualart.
\newblock {Stochastic integral of divergence type with respect to fractional
  {B}rownian motion with {H}urst parameter \(H \in (0,\frac {1}{2})\)}.
\newblock {\em {Ann. Inst. Henri Pointcar\'e (B) Probab. Stat}},
  41(6):1049--1081, 2005.

\bibitem{coutin2007}
Laure Coutin, Peter~K. Friz, and Nicolas Victoir.
\newblock Good rough path sequences and applications to anticipating stochastic
  calculus.
\newblock {\em Ann. Probab.}, 35(3):1172--1193, 2007.

\bibitem{coutin2002}
Laure Coutin and Zhongmin Qian.
\newblock Stochastic analysis, rough path analysis and fractional {B}rownian
  motions.
\newblock {\em Probab. Theory Relat Fields.}, 122(1):108--140, 2002.

\bibitem{coviello2bis}
Rosanna Coviello, Cristina Di~Girolami, and Francesco Russo.
\newblock On stochastic calculus related to financial assets without
  semimartingales.
\newblock {\em Bull. Sci. Math.}, 135(6-7):733--774, 2011.

\bibitem{coviello1}
Rosanna Coviello and Francesco Russo.
\newblock Nonsemimartingales: stochastic differential equations and weak
  {D}irichlet processes.
\newblock {\em Ann. Probab.}, 35(1):255--308, 2007.

\bibitem{deya2012}
Aur{\'e}lien Deya, Andreas Neuenkirch, and Samy Tindel.
\newblock A {M}ilstein-type scheme without {L}{\'e}vy area terms for {SDE}s
  driven by fractional {B}rownian motion.
\newblock In {\em Ann. Inst. Henri Pointcar\'e (B) Probab. Stat.}, volume~48,
  pages 518--550, 2012.

\bibitem{er2}
Mohammed Errami and Francesco Russo.
\newblock {$n$}-covariation, generalized {D}irichlet processes and calculus
  with respect to finite cubic variation processes.
\newblock {\em Stochastic Process. Appl.}, 104(2):259--299, 2003.

\bibitem{feyel2006}
Denis Feyel and Arnaud de~La~Pradelle.
\newblock Curvilinear integrals along enriched paths.
\newblock {\em Electron. J. Probab.}, 11:860--892, 2006.

\bibitem{friz2016}
Peter~K. Friz, Benjamin Gess, Archil Gulisashvili, and Sebastian Riedel.
\newblock The {J}ain--{M}onrad criterion for rough paths and applications to
  random {F}ourier series and non-{M}arkovian {H}{\"o}rmander theory.
\newblock {\em Ann. Probab.}, 44(1):684--738, 2016.

\bibitem{hairerbook}
Peter~K. Friz and Martin Hairer.
\newblock {\em A course on rough paths}.
\newblock Universitext. Springer, Cham, 2020.
\newblock With an introduction to regularity structures.

\bibitem{friz2021}
Peter~K. Friz, Antoine Hocquet, and Khoa L{\^e}.
\newblock Rough stochastic differential equations.
\newblock {\em ArXiv preprint arXiv:2106.10340}, 2021.

\bibitem{friz2014}
Peter~K. Friz and Sebastian Riedel.
\newblock Convergence rates for the full {G}aussian rough paths.
\newblock In {\em Ann. Inst. Henri Pointcar\'e (B) Probab. Stat}, volume~50,
  pages 154--194, 2014.

\bibitem{friz2010}
Peter~K. Friz and Nicolas Victoir.
\newblock Differential equations driven by {G}aussian signals.
\newblock In {\em Ann. {I}nst. Henri Pointcar\'e (B) Probab. Stat.}, volume~46,
  pages 369--413, 2010.

\bibitem{friz}
Peter~K. Friz and Nicolas~B. Victoir.
\newblock {\em Multidimensional stochastic processes as rough paths}, volume
  120 of {\em Cambridge Studies in Advanced Mathematics}.
\newblock Cambridge University Press, Cambridge, 2010.
\newblock Theory and applications.

\bibitem{leon2022}
Johanna Garzon, Jorge~A. Le\'on, and Soledad Torres.
\newblock Forward integration of bounded variation coefficients with respect to
  {H}\"older continuous process.
\newblock {\em To appear: Bernoulli}, 2022.

\bibitem{gor}
A.O. Gomes, A.~Ohashi, F.~Russo, and A.~Teixeira.
\newblock Rough paths and regularization.
\newblock {\em Journal of Stochastic Analysis (JOSA).}, 2 (4):1--21, 12 2021.

\bibitem{gradno}
Mihai Gradinaru and Ivan Nourdin.
\newblock Approximation at first and second order of {$m$}-order integrals of
  the fractional {B}rownian motion and of certain semimartingales.
\newblock {\em Electron. J. Probab.}, 8:no. 18, 26 pp. (electronic), 2003.

\bibitem{gnrv}
Mihai Gradinaru, Ivan Nourdin, Francesco Russo, and Pierre Vallois.
\newblock {$m$}-order integrals and generalized {I}t\^o's formula: the case of
  a fractional {B}rownian motion with any {H}urst index.
\newblock {\em Ann. Inst. Henri Poincar\'e (B) Probab. Stat.}, 41(4):781--806,
  2005.

\bibitem{grv}
Mihai Gradinaru, Francesco Russo, and Pierre Vallois.
\newblock Generalized covariations, local time and {S}tratonovich {I}t\^o's
  formula for fractional {B}rownian motion with {H}urst index {$H\ge\frac14$}.
\newblock {\em Ann. Probab.}, 31(4):1772--1820, 2003.

\bibitem{gubinelli2004}
Massimiliano Gubinelli.
\newblock Controlling rough paths.
\newblock {\em J. Funct. Anal.}, 216(1):86--140, 2004.

\bibitem{tindelhu}
Yaozhong Hu, Maria Jolis, and Samy Tindel.
\newblock On {S}tratonovich and {S}korohod stochastic calculus for {G}aussian
  processes.
\newblock {\em Ann. Probab.}, 41(3A):1656--1693, 2013.

\bibitem{hu2016}
Yaozhong Hu, Yanghui Liu, and David Nualart.
\newblock Rate of convergence and asymptotic error distribution of euler
  approximation schemes for fractional diffusions.
\newblock {\em Ann. Appl. Probab.}, 26(2):1147--1207, 2016.

\bibitem{krukrusso}
Ida Kruk and Francesco Russo.
\newblock Malliavin-{S}korohod calculus and {P}aley-{W}iener integral for
  covariance singular processes.
\newblock {\em ArXiv preprint arxiv:1011.6478v1}, 2010.

\bibitem{kruk2007}
Ida Kruk, Francesco Russo, and Ciprian~A Tudor.
\newblock Wiener integrals, {M}alliavin calculus and covariance measure
  structure.
\newblock {\em J. Funct. Anal.}, 249(1):92--142, 2007.

\bibitem{le2020}
Khoa L{\^e}.
\newblock A stochastic sewing lemma and applications.
\newblock {\em Electron. J. Probab.}, 25:1--55, 2020.

\bibitem{tindelliu}
Yanghui Liu, Zachary Selk, and Samy Tindel.
\newblock Convergence of trapezoid rule to rough integrals.
\newblock {\em ArXiv preprint arXiv:2005.06500}, 2020.

\bibitem{liu2019}
Yanghui Liu and Samy Tindel.
\newblock First-order {E}uler scheme for {SDE}s driven by fractional {B}rownian
  motions: the rough case.
\newblock {\em Ann. Appl. Probab.}, 29(2):758--826, 2019.

\bibitem{lyons}
Terry~J. Lyons.
\newblock Differential equations driven by rough signals.
\newblock {\em Rev. Mat. Iberoamericana}, 14(2):215--310, 1998.

\bibitem{matsuda2022}
Toyomu Matsuda and Nicolas Perkowski.
\newblock An extension of the stochastic sewing lemma and applications to
  fractional stochastic calculus.
\newblock {\em ArXiv preprint arXiv:2206.01686}, 2022.

\bibitem{nourdin2007}
Andreas Neuenkirch and Ivan Nourdin.
\newblock Exact rate of convergence of some approximation schemes associated to
  {SDE}s driven by a fractional {B}rownian motion.
\newblock {\em J. Theor. Probab.}, 20(4):871--899, 2007.

\bibitem{nualart2006}
David Nualart.
\newblock {\em The {M}alliavin calculus and related topics}.
\newblock Probability and its Applications (New York). Springer-Verlag, Berlin,
  second edition, 2006.

\bibitem{np1988}
David Nualart and {\'E}tienne Pardoux.
\newblock Stochastic calculus with anticipating integrands.
\newblock {\em Probab. Theory. Relat. Fields.}, 78(4):535--581, 1988.

\bibitem{perkowski}
Nicolas Perkowski and David~J Pr{\"o}mel.
\newblock Pathwise stochastic integrals for model free finance.
\newblock {\em Bernoulli}, 22(4):2486--2520, 2016.

\bibitem{perlman1974}
Michael~D. Perlman.
\newblock Jensen's inequality for a convex vector-valued function on an
  infinite-dimensional space.
\newblock {\em J. Multivar. Anal.}, 4(1):52--65, 1974.

\bibitem{russo2006bifractional}
Francesco Russo and Ciprian Tudor.
\newblock On bifractional {B}rownian motion.
\newblock {\em Stoch. Process. Their Appl.}, 116, 05 2006.

\bibitem{russo1993forward}
Francesco Russo and Pierre Vallois.
\newblock Forward, backward and symmetric stochastic integration.
\newblock {\em Probab. Theory Relat Fields.}, 97(3):403--421, 1993.

\bibitem{russo2007elements}
Francesco Russo and Pierre Vallois.
\newblock Elements of stochastic calculus via regularization.
\newblock In {\em S{\'e}minaire de Probabilit{\'e}s XL}, pages 147--185.
  Springer, 2007.

\bibitem{Russo_Vallois_Book}
Francesco Russo and Pierre Vallois.
\newblock {\em Stochastic Calculus via Regularizations}.
\newblock Springer International Publishing, 2022.

\bibitem{tindelsong}
Jian Song and Samy Tindel.
\newblock Skorohod and {S}tratonovich integrals for controlled processes.
\newblock {\em Stoch. Process. Their Appl.}, 2022.

\bibitem{wzakai}
Eugene Wong and Moshe Zakai.
\newblock On the relation between ordinary and stochastic differential
  equations.
\newblock {\em Int. J. Eng. Sci.}, 3(2):213--229, 1965.

\end{thebibliography}

\end{document}